\documentclass[11pt,reqno]{amsart}
\usepackage[margin=1in]{geometry}
\usepackage{enumitem}
\usepackage{amssymb, amsmath,latexsym,amsfonts,amsbsy, amsthm,mathtools,graphicx,color}
\usepackage{float}
\usepackage{tikz}
\usetikzlibrary{arrows, decorations.markings,matrix,trees}
\usepackage{hyperref}
\def\Xint#1{\mathchoice
  {\XXint\displaystyle\textstyle{#1}}%
  {\XXint\textstyle\scriptstyle{#1}}%
  {\XXint\scriptstyle\scriptscriptstyle{#1}}%
  {\XXint\scriptscriptstyle\scriptscriptstyle{#1}}%
  \!\int}
\def\XXint#1#2#3{{\setbox0=\hbox{$#1{#2#3}{\int}$}
  \vcenter{\hbox{$#2#3$}}\kern-.5\wd0}}

\def\dashint{\Xint-}
\newcommand{\al}{\alpha}       
\newcommand{\vt}{\vartheta}    
\newcommand{\lda}{\lambda}
\newcommand{\om}{\Omega}            
\newcommand{\pa}{\partial}
\newcommand{\va}{\varepsilon}       
\newcommand{\ud}{\mathrm{d}}
\newcommand{\be}{\begin{equation}} 
\newcommand{\ee}{\end{equation}}
\newcommand{\w}{\omega}      
\newcommand{\Lda}{\Lambda}

\newcommand{\bA}{\mathbb{A}}

\newcommand{\cB}{\mathcal{B}}

\newcommand{\cC}{\mathcal{C}}

\newcommand{\cD}{\mathcal{D}}

\newcommand{\cE}{\mathcal{E}}

\newcommand{\cF}{\mathcal{F}}

\newcommand{\bG}{\mathbb{G}}

\newcommand{\cG}{\mathcal{G}}

\newcommand{\cL}{\mathcal{L}} 

\newcommand{\Z}{\mathbb{Z}}

\newcommand{\M}{\mathcal{M}}

\newcommand{\cN}{\mathcal{N}}

\newcommand{\R}{\mathbb{R}}   

\newcommand{\cR}{\mathcal{R}}

\newcommand{\Ss}{\mathbb{S}}

\newcommand{\wc}{\rightharpoonup}        

\newcommand{\HH}{\mathcal{H}}

\newcommand{\vp}{\varphi}

\newcommand{\ga}{\gamma}

\newcommand{\sg}{\sigma} 
\newcommand{\ift}{\infty} 
\newcommand{\wt}{\widetilde}

\newcommand{\f}{\frac}

\newcommand{\ol}{\overline}
\newcommand{\Ra}{\Rightarrow}
\newcommand{\op}{\operatorname}
\newcommand{\Sg}{\Sigma}

\newcommand{\na}{\nabla}

\DeclareMathOperator{\dist}{dist}

\DeclareMathOperator{\diam}{diam}

\DeclareMathOperator{\supp}{supp}

\DeclareMathOperator{\diag}{diag}

\DeclareMathOperator{\sing}{sing}

\DeclareMathOperator{\loc}{loc}
\DeclareMathOperator{\argmin}{argmin}

\def\<{\langle}\def\>{\rangle}
\def\({\left(}\def\){\right)}
\def\[{\left[}\def\]{\right]}
\numberwithin{equation}{section}
\theoremstyle{plain}
\newtheorem{thm}{Theorem}[section]
\newtheorem{cor}[thm]{Corollary}

\newtheorem{lem}[thm]{Lemma}
\newtheorem{prop}[thm]{Proposition}

\theoremstyle{definition}
\newtheorem{defn}[thm]{Definition}

\newtheorem{rem}[thm]{Remark}

\title[Semilinear elliptic equation with singular nonlinearities]{Fine structure of rupture set for semilinear elliptic equation with singular nonlinearity}

\author{}

\author{Wei Wang}
\address{School of Mathematical Sciences, Peking University, Beijing 100871, China}
\email{2201110024@stu.pku.edu.cn}

\author{Zhifei Zhang}
\address{School of Mathematical Sciences, Peking University, Beijing 100871, China}
\email{zfzhang@math.pku.edu.cn}

\date{}

\begin{document}

\begin{abstract}
In this paper, we study the stationary solutions of semilinear elliptic equation with singular nonlinearity 
$$
\Delta u=u^{-p}+f,\,\,u\geq 0\quad\text{in }\om\subset\mathbb{R}^n,
$$
where $ n\geq 2 $, $ p>1 $, $ \om $ is a bounded domain, and $ f\in L^q(\om) $ with $ \f{1}{2}+\f{1}{2p}<\f{q}{n} $.  We establish a sharp estimate for the Minkowski content of the rupture set $ \{u=0\} $ and demonstrate that this set is $ (n-2) $-rectifiable. For this, we examine the stratification of the rupture set based on the symmetry properties of tangent functions, leading to the proof of $ k $-rectifiability for each $ k $-stratum. As a significant byproduct of our analysis, we improve the integrability of $ D^ju $ with $ j\in\Z_+ $ to the optimal Lorentz space $ L^{\f{2(p+1)}{j(p+1)-2},\ift} $, under the assumption that $ D^{j-1}f $ is bounded. As an application of our results in the static case of the equation, for a class of suitable weak solutions to the three-dimensional evolutional problem
$$
\pa_tu=\Delta u-u^{-p},\,\,u\geq 0\quad\text{in }(\om\subset\mathbb{R}^3)\times(0,T),
$$
where $ p>3 $ and $ T>0 $,  we show that $ \{u(\cdot,t)=0\} $ is $ 1 $-rectifiable for a.e. $ t\in(0,T) $.
\end{abstract}

\maketitle
\tableofcontents
\section{Introduction}

\subsection{Background and related results} In this paper, we investigate a semilinear elliptic equation with singular nonlinearity
\be
\Delta u=u^{-p}+f,\,\,u\geq 0\quad\text{in }\om\subset\R^n,\,\,n\geq 2,\,\,p>1,\label{MEMSeq}
\ee
where $ \om $ is a domain and $ f\in L_{\loc}^1(\om) $. This equation arises from different scientific contexts and has a rich theoretical background. 

In the context of thin film theory, the equation \eqref{MEMSeq} describes specific steady-state scenarios, where the value of $ u $ corresponds to the height of the air-liquid interface. Guo and Wei gave short discussions on the derivation for this equation in the introduction section of \cite{GW06}. For additional literature on thin films, particularly concerning the one-dimensional case, we recommend the studies by Bertozzi and Pugh \cite{BP98,BP00}, which explore dynamical problems, as well as the investigations of Laugesen-Pugh \cite{LP00a,LP00b}, which focus on steady-state solutions. 

For $ p=2 $ in \eqref{MEMSeq}, the equation characterizes a simplified model in micro-electromechanical systems (MEMS). Here, the scalar $ u $ represents the deflection of the membrane in the device. This field is well-established and significant in modern technology, playing a pivotal role in various devices such as sensors and actuators. For a comprehensive overview of the physical principles underlying this model and its subsequent advancements, we refer to the monograph \cite{PB02} by Pelesko and Bernstein. Abdel-Rahman, Younis, and Nayfeh provided further insights in \cite{NYA05}. From a mathematical perspective, various properties and open problems related to the MEMS issue are examined and introduced in the monograph by Esposito, Ghoussoub, and Guo \cite{EGG10}, along with the survey paper by Lauren\c{c}ot-Walker \cite{LW17}.

While our focus on the equation \eqref{MEMSeq} is primarily the case of $ p>1 $, there are also research works under the relaxed assumption $ p\geq 0 $. Although the case $ 0<p<1 $ is not the central theme of this paper, we will briefly review the relevant results for completeness. For example, when $ p=1 $, the equation has roots in the study of singular minimal hypersurfaces, particularly under certain symmetry assumptions (see \cite{Mea04} for a detailed discussion). For $ p=0 $, we can interpret \eqref{MEMSeq} as $ \Delta u=\chi_{\{u>0\}} $, which originates from the obstacle problem. Significant contributions to this topic are in works such as \cite{Caf98,FS19,FRS20,Wei99}, which examine both the properties of solutions and the geometric structure of the free boundary $ \pa\{u>0\} $. The case for $ 0<p<1 $ is a generalization of the case with $ p=0 $. For additional results and insights on this setting, we refer to \cite{Phi83,ST19}. In \cite{ST19}, the authors investigated a two-phase problem and offered comparisons with the case of $ p>1 $.

In the study of thin films, for the solution of \eqref{MEMSeq}, the set $ \{u=0\} $ is of main concern since it describes the rupture phenomenon. For the remainder of this paper, we will refer to $ \{u=0\} $ as the rupture set of $ u $. From a mathematical standpoint, if $ f\in C^{\ift}(\om) $ and $ u\in C^0(\om) $, then for any $ x\in\{u>0\} $, standard results in elliptic equations indicate that there exists $ r>0 $ such that $ u\in C^{\ift}(B_r(x)) $. As a result, $ u $ is smooth within the region where $ u $ is positive. As a result, the rupture set is the singularity set of $ u $. Measuring the size of this rupture set is a significant problem in the analysis for solutions of \eqref{MEMSeq}, particularly in estimating the Hausdorff dimension of $ \{u=0\} $ under different conditions imposed on the solution. Recent literature has extensively explored this topic, revealing intricate relationships between the properties of solutions and the geometric characteristics of the rupture set $ \{u=0\} $. In \cite{JL04}, Jiang and Lin examined weak solutions of \eqref{MEMSeq}. In fact, they considered a more general model that contains weak solutions of \eqref{MEMSeq}. Here, if $ \om\subset\R^n $ is a domain, the function $ u\in (H_{\loc}^1\cap L_{\loc}^{-p})(\om) $ is defined as a weak solution of \eqref{MEMSeq} with respect to $ f\in L_{\loc}^1(\om) $ in the distributional sense if $ u\geq 0 $ a.e. in $ \om $ and for any $ \vp\in C_0^{\ift}(\om) $, $ u $ satisfies the integral identity
\be
\int_{\om}(\na u\cdot\na\vp+(u^{-p}+f)\vp)=0.\label{WeakConMEMS}
\ee
The results established in \cite{JL04} demonstrated that the Hausdorff dimension of the rupture set $ \{u=0\} $ for such a weak solution is at most $ n-2+\f{4}{p+2} $. Subsequently, Dupaigne, Ponce, and Porretta improved this estimate to $ n-2+\f{2}{p+1} $ in \cite{DPP06}. In addition to weak solutions, another crucial class of solutions is the finite energy solution. Following the terminology in \cite{JL04}, we define $ u\in (C_{\loc}^0\cap H_{\loc}^1\cap L_{\loc}^{1-p})(\om) $ as a finite energy solution of \eqref{MEMSeq} if $ u\geq 0 $ in $ \om $ and $ \Delta u=u^{-p}+f $ in $ \{u>0\} $ in the sense of distribution, namely, \eqref{WeakConMEMS} holds for any $ \vp\in C_0^{\ift}(\{u>0\}) $. If we set the corresponding energy functional
\be
\cF_f(u,\om):=\int_{\om}\(\f{|\na u|^2}{2}-\f{u^{1-p}}{p-1}+fu\),\label{FunctEll}
\ee
then for such solutions, this functional is locally finite. For this setting, Guo and Wei \cite{GW08} showed that the Hausdorff dimension is at most $ n-2+\f{4}{p+1} $. Later, D\'{a}vila and Ponce refined this estimate further to $ n-2+\f{2}{p+1} $ in \cite{DP08}. To delve deeper into the micro behavior of the equation \eqref{MEMSeq}, motivated by \cite{Eva91} on harmonic maps, Guo and Wei  \cite{GW06}  introduced the concept of the stationary solution defined as follows.

\begin{defn}[Stationary solution]
Let $ \om\subset\R^n $ be a domain. $ u\in(H_{\loc}^1\cap L_{\loc}^{-p})(\om) $ is called a stationary solution of \eqref{MEMSeq} with respect to $ f\in L_{\loc}^2(\om) $ if $ u $ is a weak solution of \eqref{MEMSeq} with respect to $ f $, and satisfies the stationary condition, namely,
\be
\int_{\om}\left[\(\f{|\na u|^2}{2}-\f{u^{1-p}}{p-1}\)\op{div}Y-DY(\na u,\na u)-f(Y\cdot\na u)\right]=0\label{StaConMEMS}
\ee
for any $ Y\in C_0^{\ift}(\om,\R^n) $, where
$$
DY(\na u,\na u)=\sum_{i,j=1}^n\pa_iY^j\pa_iu\pa_ju.
$$
\end{defn}

In fact, given the functional \eqref{FunctEll}, the stationary condition \eqref{StaConMEMS} is equivalent to that $ u $ is a critical point of $ \cF_f(\cdot,\om) $ under inner perturbations of $ \om $. Precisely,
$$
\left.\f{\ud}{\ud t}\right|_{t=0}\cF_f(u(\cdot+tY(\cdot)),\om)=0
$$
for any $ Y\in C_0^{\ift}(\om,\R^n) $. For the stationary solution of \eqref{MEMSeq}, the most remarkable result is obtained by D\'{a}vila, Wang, and Wei \cite{DWW16}. 
They proved the following theorem.

\begin{thm}[\cite{DWW16}, Theorem 1.2]\label{thmDWW}
Assume that $ u\in(C_{\loc}^{0,\al}\cap H_{\loc}^1\cap L_{\loc}^{-p})(B_1) $ is a stationary solution of \eqref{MEMSeq} with respect to $ f\equiv 0 $. Then, the rupture set $ \{u=0\} $ is a relatively closed set with Hausdorff dimension no more than $ n-2 $. Moreover, if $ n=2 $, then $ \{u=0\} $ is discrete.
\end{thm}

\begin{rem}
In this theorem, the parameter $ \al=\al_p:=\f{2}{p+1} $ is the index associated with the equation \eqref{MEMSeq}, and we will consistently use this notation in the rest of this paper. The assumption of the $ \al $-H\"{o}lder continuity corresponds to the $ C_{\loc}^{1,1} $ regularity in the obstacle problem (see \cite{Caf98} for more details).
\end{rem}

\begin{rem}
In the case of two dimensions, 
\be
u(x)=u(|x|)=\al^{-\al}|x|^{\al}\label{ralsolu}
\ee
is a stationary solution for $ \Delta u=u^{-p} $ and here the rupture set is given by $ \{u=0\}=\{(0,0)\} $. Thus, the results in Theorem \ref{thmDWW} are sharp.
\end{rem}

The key ingredient in establishing Theorem \ref{thmDWW} is the monotonicity formula, a consequence of the stationary condition \eqref{StaConMEMS}. In \cite{DWW16}, the authors adopted the density
\be
\theta(u;x,r):=r^{2-2\al-n}\int_{B_r(x)}\(\f{|\na u|^2}{2}-\f{u^{1-p}}{p-1}\)-\f{\al r^{-2\al-n}}{2}\int_{\pa B_r(x)}u^2\ud\HH^{n-1},\label{classDen}
\ee
with $ x\in B_1 $ and $ r\in(0,1-|x|) $. A similar density was first introduced in the study of stationary solutions of semilinear elliptic equations by Pacard in \cite{Pac93}. Also note that when $ p=0 $, this density becomes equivalent to that introduced by Weiss \cite{Wei99} for the obstacle problem. By \eqref{StaConMEMS}, we have
\be
\f{\ud}{\ud r}\theta(u;x,r)=r^{-2\al-n}\int_{\pa B_r(x)}|(y-x)\cdot\na u-\al u|^2\ud\HH^{n-1}(y)\geq 0,\label{classDenfor}
\ee
which implies that $ \theta(u;x,\cdot) $ is nondecreasing. Using this monotonicity formula, Theorem \ref{thmDWW} follows from the well-known Federer's dimension reduction principle. For an additional application of this method in the context of harmonic maps, one can refer to the work by Schoen and Uhlenbeck \cite{SU82}. In \S\ref{StratificationSection}, we will provide another proof of this theorem for the more general case where $ f\not\equiv 0 $ in \eqref{MEMSeq} by using standard stratification results and approaches developed by White \cite{Whi97}.

Besides estimating the Hausdorff dimension of the rupture set of \eqref{MEMSeq}, there are different topics of interest related to similar types of equations. Readers can refer to \cite{DW12,DH96,DG09,GGL06,GW07,LGG05,MW08,Wan12} for further discussions.

\subsection{Main results}

Based on Theorem \ref{thmDWW}, a natural question arises regarding the fine structure of $ \{u=0\} $, particularly its regularity properties. Before stating our main results, we define the Minkowski content, Minkowski dimension, and the concept of rectifiability.

\begin{defn}[Minkowski content and dimension]\label{Minkdef}
Let $ k\in\Z\cap[0,n] $, $ r>0 $, and $ S\subset\R^n $. The $ k $-dimensional Minkowski $ r $-content of $ S $ is given by
$$
\op{Min}_r^k(S):=(2r)^{k-n}\cL^n(B_r(S)),
$$
where 
$$
B_r(S):=\bigcup_{x\in S}B_r(x)=\{y\in\R^n:\dist(y,S)<r\} 
$$
is the $ r $-neighborhood of $ S $. Let the upper and lower Minkowski content be
$$
\ol{\op{Min}}_0^k(S):=\limsup_{r\to 0^+}\op{Min}_r^k(S)\quad\text{and}\quad\underline{\op{Min}}_0^k(S):=\liminf_{r\rightarrow 0^+}\op{Min}_r^k(S).
$$
The Minkowski dimension (or box-dimension) of $ S $ is defined by
$$
\dim_{\op{Min}}S:=\inf\{k\geq 0:\ol{\op{Min}}_0^k(S)=0\}.
$$
\end{defn}

\begin{defn}[Rectifiability]
Let $ N\in\Z_+ $ and $ k\in\Z\cap[1,N] $. We call a set $ M\subset\R^N $ countably $ k $-rectifiable (or simply rectifiable) if
$$
M\subset M_0\cup\bigcup_{i\in\Z_+} f_i(\R^k),
$$
where $ \HH^k(M_0)=0 $, and $ f_i:\R^k\to\R^N $ is a Lipschitz map for any $ i\in\Z_+ $. 
\end{defn}

Now, we present the main results of this paper as follows.

\begin{thm}\label{main1}
Let $ q>0 $ be such that $ \f{1}{2}+\f{1}{2p}<\f{q}{n} $.
Assume that $ u\in(C_{\loc}^{0,\al}\cap H_{\loc}^1\cap L_{\loc}^{-p})(B_4) $ is a stationary solution of \eqref{MEMSeq} with $ f\in L_{\loc}^q(B_4) $, satisfying 
\be
\|u\|_{L^1(B_2)}+\|f\|_{L^q(B_2)}\leq\Lda.\label{uholderLq}
\ee
Then, the following properties hold.
\begin{enumerate}[label=$(\theenumi)$]
\item There exist $ \va,C>0 $, depending only on $ \Lda,n,p $, and $ q $ such that
\be
\cL^n(B_r(\{u<\va r^{\al}\}\cap B_1))\leq Cr^2\label{Mincontent}
\ee
for any $ 0<r<1 $. In particular, this estimate implies that the Minkowski dimension of $ \{u=0\}\cap B_1 $ is at most $ (n-2) $.
\item If $ f\in W_{\loc}^{j-1,\ift}(B_4) $ for some $ j\in\Z_+ $ such that $
\|f\|_{W^{j-1,\ift}(B_2)}\leq\Lda' $, then $ D^ju\in L^{\f{2}{j-\al},\ift}(B_1) $. In particular, we have
\be
\sup\left\{\lda>0:\lda^{\f{2}{j-\al}}\cL^n(\{x\in B_1:|D^ju(x)|>\lda\})\right\}\leq C',\label{enhancement}
\ee
where $ C'>0 $ depends only on $ \Lda,\Lda',j,n,p $, and $ q $.
\item $ \{u=0\} $ is $ (n-2) $-rectifiable, and for $ n=2 $, $ \{u=0\} $ is a discrete set. 
\end{enumerate}
\end{thm}

This result is the first to provide a characterization for the rectifiability and estimates of the Minkowski content of the rupture set $ \{u=0\} $, together with the enhancement on the integrability of $ D^ju $ for $ j\in\Z_+ $. Let us give more remarks.

\begin{rem}
The results in this theorem are a priori since we assume that the solution is in $ C_{\loc}^{0,\al} $. It is still an open problem that if a stationary solution of \eqref{MEMSeq} with some specific assumptions of $ f $ is H\"{o}lder continuous.
\end{rem}

\begin{rem}
By applying standard interpolation inequalities related to Lorentz spaces, we can deduce from \eqref{enhancement} that $ D^ju\in L^{\f{2}{j-\al}-}(B_1) $, namely, for any $ 0<s<\f{2}{j-\al} $, there exists $ C>0 $, depending only on $ \Lda,\Lda',j,n,p,q $, and $ s $ such that $
\|D^ju\|_{L^s(B_1)}\leq C $. For $ j=1 $, this leads to $ u\in W^{1,\f{2(p+1)}{p-1}-}(B_1) $,  which is a substantial improvement over the $ H^1 $ regularity. Since the function defined by \eqref{ralsolu} is a stationary solution of \eqref{MEMSeq} with $ f\equiv 0 $, the estimates \eqref{Mincontent}, \eqref{enhancement} and the $ (n-2) $-rectifiability property are all sharp.
\end{rem}

\begin{rem}
In the proof of the rectifiability of $ \{u=0\} $, we actually demonstrate that any $ k $-stratum of $ \{u=0\} $ is $ k $-rectifiable with $ k\in\Z\cap[0,n-2] $ and treat $ \{u=0\} $ as the top stratum. Here the $ k $-stratum $ S^k(u) $ of $ u $ consists of all points where the tangent functions fail to be invariant with respect to all $ (k+1) $-dimensional subspaces of $ \R^n $. We will present the explicit definition of this stratification in \S\ref{StratificationSection}.
\end{rem}

Motivated by the study of harmonic heat flows using approximate harmonic maps as explored in \cite{Mos03}, we can utilize the properties established in Theorem \ref{main1} to examine the evolutional problem associated with \eqref{MEMSeq} in three dimensions. For a domain $ \om\subset\R^n $ and $ T>0 $, we define $ \om_T=\om\times(0,T) $, and consider the evolutional problem related to \eqref{MEMSeq} as
\be
\pa_tu-\Delta u=-u^{-p},\,\,u\geq 0\quad\text{in }\om_T,\label{dynaMEMS}
\ee
where $ p>1 $. Following \cite{WY24}, we define a special class of solutions for \eqref{dynaMEMS}.

\begin{defn}\label{defsuita}
A function $ u:\om_T\to\R_{\geq 0} $ is called a suitable weak solution of \eqref{dynaMEMS} if the following properties hold.
\begin{enumerate}
\item $ u\in H_{\loc}^1(\om_T)\cap L_{\loc}^{-p}(\om_T) $.
\item $ u $ is a weak solution of \eqref{dynaMEMS}, that is, for any $ \vp\in C_0^{\ift}(\om_T) $,
$$
\int_{\om_T}(\pa_tu\vp+\na u\cdot\na\vp+u^{-p}\vp)=0.
$$
\item $ u $ is stationary in the sense that for any $ Y\in C_0^{\ift}(\om_T,\R^n) $,
$$
\int_{\om_T}\[\left(\f{|\na u|^2}{2}-\f{u^{1-p}}{p-1}\)\op{div}Y-DY(\na u,\na u)-\pa_tu(\na u\cdot Y)\]=0.
$$
\item $ u $ satisfies the localized energy inequality, that is, for any $ \vp\in C_0^{\ift}(\om) $ and $ \psi\in C_0^{\ift}(0,T) $ with $ \psi\geq 0 $, there holds
\begin{align*}
&\int_0^T\int_{\om}\(\f{|\na u|^2}{2}-\f{u^{1-p}}{p-1}\)\vp^2\pa_t\psi\geq\int_0^T\int_{\om}|\pa_tu|^2\vp^2\psi\\
&\quad\quad+2\int_0^T\int_{\om}\pa_t u(\na u\cdot\na\vp)\vp\psi-2\int_0^T\int_{\om}\(\f{|\na u|^2}{2}-\f{u^{1-p}}{p-1}\)\vp\psi\pa_t\vp.
\end{align*}
\end{enumerate}
\end{defn}

\begin{thm}\label{thmmain2}
Assume that $ B_1\subset\R^3 $, $ p>3 $, and $ u\in(C_{\loc}^{\al,\f{\al}{2}}\cap H_{\loc}^1\cap L_{\loc}^{-p})(B_1\times(0,T)) $ is a suitable weak solution of \eqref{dynaMEMS}. Then for a.e. $ t\in(0,T) $, $ \{u(\cdot,t)=0\} $ is $ 1 $-rectifiable. 
\end{thm}

\begin{rem}
This theorem follows directly from Theorem \ref{main1}. Indeed, for a.e. $ t\in(0,T) $, $ \pa_tu(\cdot,t)\in L^2(B_1) $. The assumption $ n=3 $ and $ p>3 $ imply that $ \f{2}{3}>\f{1}{2}+\f{1}{2p} $. Thus, $\{u(\cdot,t)=0\} $ is $ 1 $-rectifiable. 
\end{rem}

In \cite{WY24}, Wang and Yi extended the results of Theorem \ref{thmDWW} to the evolutional problem represented by \eqref{dynaMEMS}. Their results primarily concentrate on estimates of the parabolic Hausdorff dimension for the rupture set. Although Theorem \ref{thmmain2} addresses only the case for $ n=3 $ and $ p>3 $, it gives a novel result concerning the rectifiability of almost every slice of the rupture set.

\subsection{Difficulties and our strategies}\label{subdiffi}

Examining singular sets of nonlinear elliptic equations, especially those associated with variational problems, produces a significant area of research. The study of harmonic maps is particularly noteworthy. In our model, the results regarding the rectifiability and estimates of the Minkowski contents of rupture sets are in line with those established for harmonic maps. However, due to the distinct structures of the equation involved, the analysis for stationary solutions of \eqref{MEMSeq} presents more challenges. To illustrate these difficulties, we will first review relevant results in the context of harmonic maps, enabling a comparison with our conclusions in Theorems \ref{main1} and Theorem \ref{thmmain2}. \smallskip

Let $ \om\subset\R^n $ be a bounded domain. Recall that for $ \cN\hookrightarrow\R^d $ being a real, smooth, compact manifold (here the notation $ \hookrightarrow $ indicates that we can isometrically embed $ \cN $ into $ \R^d $), the harmonic map refers to the critical point of the variational problem for the Dirichlet energy
\be
\cE(\Phi,\om):=\int_{\om}|\na\Phi|^2,\quad\Phi=(\Phi^1,\Phi^2,...,\Phi^d)\in H^1(\om,\cN).\label{Energyfun}
\ee

\begin{defn}
$ \Phi\in H^1(\om,\cN) $ is a local minimizer of \eqref{Energyfun} if for any $ B_r(x)\subset\subset \om $, and $ \Psi\in H^1(B_r(x),\cN) $ with $ \Phi=\Psi $ on $ \pa B_r(x) $ in the sense of trace, there holds
$$
\int_{B_r(x)}|\na\Phi|^2\leq\int_{B_r(x)}|\na\Psi|^2.
$$
\end{defn}

\begin{rem}
Assume that $ \pa\om $ is Lipschitz. Let $ g\in H^{\f{1}{2}}(\pa\om,\cN) $. By standard direct methods in variational problems, there exists $ \Phi\in H^1(\om,\cN) $ solving the variational problem
$$
\min\{\cE(\Psi,\om):\Psi\in H^1(\om,\cN),\text{ and }\Psi=g\text{ on }\pa\om\text{ in the sense of trace}\}.
$$
Here we call $ \Phi $ the global minimizer of $ \cE(\cdot,\om) $. By this definition, we see that global minimizers are local minimizers.
\end{rem}

\begin{defn}
Assume that $ \Phi\in H^1(\om,\cN) $. We call $ \Phi $ a weakly harmonic map if for any $ \vp=(\vp^i)_{i=1}^d\in C_0^{\ift}(\om,\R^d) $,
\be
\int_{\om}(\na\Phi\cdot\na\vp-A(\Phi)(\na\Phi,\na\Phi)\cdot\vp)=0,\label{weakhar}
\ee
where $ A(y)(\cdot,\cdot):T\cN\times T\cN\to(T\cN)^{\perp} $ is the second fundamental form of $ \cN $ at the point $ y\in\cN $. Additionally, $ u $ is a stationary harmonic map if $ u $ is a weakly harmonic map and
\be
\int_{\om}(|\na\Phi|^2\op{div}Y-2DY(\na\Phi,\na\Phi))=0\label{stahar}
\ee
for any $ Y\in C_0^{\ift}(\om,\R^n) $.
\end{defn}

\begin{rem}
The conditions \eqref{weakhar} and \eqref{stahar} for harmonic maps are paralleled to \eqref{WeakConMEMS} and \eqref{StaConMEMS}, corresponding to outer and inner perturbations of the Dirichlet functional \eqref{Energyfun} (see Chapter 10.1 of \cite{GM05} for more details). 
\end{rem}

In various settings of harmonic maps, the singular set $ \sing(\Phi) $ for $ \Phi\in H^1(\om,\cN) $ is defined as the collection of points where $ \Phi $ fails to be continuous in any small neighborhood. Preliminary results by Schoen and Uhlenbeck \cite{SU82} demonstrated the partial regularity of local minimizers of the Dirichlet energy \eqref{Energyfun}, revealing that the Hausdorff dimension of $ \sing(\Phi) $ does not exceed $ n-3 $. Subsequently, Bethuel \cite{Bet93} established analogous results for stationary harmonic maps, proving that their Hausdorff dimension is at most $ n-2 $. The estimates provided in Theorem \ref{thmDWW} are in the same spirit as these two pivotal works.

Assuming that the target manifold $ \cN $ is analytic,  Simon \cite{Sim95} proved that the singularity set of a local minimizer is $ (n-3) $-rectifiable. For further insights, readers may refer to the book \cite{Sim96}. The arguments in the proof highly rely on the fact that when $ \cN $ is analytic, any tangent map of the local minimizer $ \Phi $ at $ \sing(\Phi) $ is unique. In his work, Simon also established the $ k $-rectifiability for any $ k $-stratum. In \cite{Lin99}, Lin investigated the concentration set for a sequence of stationary harmonic maps. Specifically, for such a sequence $ \{\Phi_i\} $ with uniform $ H^1 $-bound, up to a subsequence, $ \Phi_i\wc\Phi_{\ift} $ weakly in $ H^1(\om,\R^d) $ and 
$$
|\na\Phi_i|^2\ud x\wc^*|\na\Phi_{\ift}|^2\ud x+\nu,
$$
in the sense or Radon measures, where $ \nu $ is a nonnegative Radon measure on $ \om $. The concentration set $ \Sg $ is defined by
$$
\Sg:=\bigcap_{r>0}\left\{x\in\om:\liminf_{i\to+\ift}\Theta(\Phi;x,r)>\va_0^2\right\},
$$
where for $ \Phi\in H^1(\om,\cN) $,
\be
\Theta(\Phi;x,r):=r^{2-n}\int_{B_r(x)}|\na\Phi|^2,\label{Harmonicden}
\ee
and $ \va_0=\va_0(n,\cN)>0 $ is a constant related to the partial regularity. The results in \cite{Lin99} presented that $ \Sg $ is $ (n-2) $-rectifiable and $
\nu=\Theta(\nu,x)\HH^{n-2}\llcorner\Sg $, where $ \Theta(\nu,x)\geq\va_0^2 $ for any $ x\in\om $. We refer to \cite{LW99}  for similar results related to the Ginzburg-Landau model. Recently, a significant breakthrough in analyzing singular sets for harmonic maps was achieved by Naber and Valtorta in \cite{NV17}, where They remove the analyticity assumption on $ \cN $ made in \cite{Sim95}. Furthermore, they provided interior estimates for the $ (n-3) $-dimensional Minkowski content of the singular set of local minimizers. Their approach, known as quantitative stratification, incorporates new tools from geometric measure theory, specifically Reifenberg-type theorems.

Returning to the equation \eqref{MEMSeq}, we find that there are inherent difficulties in applying techniques in \cite{Lin99,NV17,Pre87,Sim95} to our problem. Firstly, for points within the rupture set, as discussed in \S 5 of \cite{DWW16}, the tangent function may vary depending on different blow-up scales, making Simon's method inapplicable. Moreover, the methods presented in \cite{Lin99,NV17,Pre87} are also unsuitable due to the structural differences between the energy densities. Recall that when $ f\equiv 0 $, for $ u\in(C_{\loc}^{0,\al}\cap H_{\loc}^1\cap L_{\loc}^{-p})(\om) $ being a stationary solution of \eqref{MEMSeq}, the nondecreasing energy density $ \theta(u;\cdot,\cdot) $ is given by \eqref{classDen}. For the study of stationary harmonic maps, the widely employed density \eqref{Harmonicden} enjoys the monotonicity formula
\be
\f{\ud}{\ud r}\Theta(\Phi;x,r)=2r^{-n}\int_{\pa B_r(x)}|(y-x)\cdot\na\Phi|^2\ud\HH^{n-1}(y)\geq 0.\label{harmonoton}
\ee
The primary distinction between these two densities is that $ \theta(u;\cdot,\cdot) $ can take on negative values. Indeed, Lemma 5.5 in \cite{DWW16} implies that
\be
x\in\{u>0\}\,\,\Leftrightarrow\,\,\lim_{r\to 0^+}\theta(u;x,r)=-\ift.\label{thetauxrbound}
\ee
The arguments in \cite{Lin99,NV17,Pre87} rely on the nonnegativity of the density \eqref{Harmonicden} for harmonic maps so they are not valid for the proof of the model in our paper. In particular, as discussed in Section 1.2 of \cite{FWZ24}, the proofs in \cite{NV17} relies on a crucial property that if $ B_r(x)\subset B_s(y)\subset\om $, then
\be
\Theta(\Phi;x,r)\leq\(\f{r}{s}\)^{n-2}\Theta(\Phi;y,s).\label{simesti}
\ee
This property does not necessarily hold for the density defined in \eqref{classDen}, further complicating the application of existing methods to our analysis.

After the work in \cite{NV17}, several studies, including \cite{HSV19} and \cite{NV18}, have refined the methodologies presented in that paper, making them less reliant on the nonnegativity condition and inequalities akin to \eqref{simesti}. In these later works, the authors emphasize the importance of the boundedness of the density, which is evident in the context of harmonic maps, as expressed in \eqref{Harmonicden}. Noting that for $ x\in\{u=0\} $, by \eqref{thetauxrbound}, we see from \eqref{classDenfor} that the density \eqref{classDen} is bounded within $ \{u=0\} $. Therefore, we can restrict our attention to the rupture set and follow the frameworks in \cite{NV18}. Concentrating on the rupture set is advantageous, as it facilitates the establishment of its rectifiability. Moreover, it is not hard to show that $ \cL^n(B_r(\{u=0\}\cap B_1))\leq Cr^2 $ under the assumption of Theorem \ref{main1}, which is less than \eqref{Mincontent}. It is natural to consider if there is some estimate on the increasing rate of $ u $ near the rupture set, such as for $ x\in\{u>0\} $, $ u(x)\gtrsim(\dist(x,\{u=0\}))^{\al} $. However, such an estimate is not necessarily true in our model for stationary solutions of \eqref{MEMSeq}. In certain special cases, such as the radial solution, \cite{DH96} gave similar results. For the general case in this paper, to our knowledge, one can only obtain a much weaker nondegeneracy estimate (see Lemma \ref{nondege}). The main difficulty in the proof of the estimate \eqref{Mincontent} is that the left-hand side involves the $ r $-neighborhood of points in $ \{0<u<\va r^{\al}\} $. To overcome this point, we first modify the blow-up analysis developed in \cite{DWW16} and develop an alternative result to separately deal with points in the domain, depending on the value $ u $. Recall that in \cite{DWW16}, the authors applied the blow-up procedure for the sequence $ \{r^{-\al}u(x+ry)\}_{r>0} $. The shortcoming of it is that if $ u(x)>0 $, the sequence does not have a limit for any subsequence. As a result, it does not fit well with the quantitative stratification arguments in \cite{NV18}. In our paper, on the other hand, we consider the blow-up sequence $ \{r^{-\al}(u(x+ry)-u(x))\}_{r>0} $, which solves this problem. Next, let us provide an intuitive overview of our alternative result in the proof. Fix a point $ x\in\om $ and a scale $ 0<r<\dist(x,\om) $. If $ 0\leq u(x)\ll r^{\al} $, the behavior of $ u $ within $ B_r(x) $ resembles the case that $ u(x)=0 $, thus, we can apply methods in line with \cite{NV18}. On the other hand, if $ u(x)\gtrsim r^{\al} $ we can utilize standard regularity theory for elliptic equations to find a small ball $ B_{\delta r}(x) $ such that $ u $ exhibits nice properties. By combining these two ingredients, we can effectively close our arguments. To our knowledge, these improvements and modifications are novel and represent a first-time application in this context.

Finally, we would like to highlight the differences between the model in our paper and the study of specific level sets of minimizers in some variational problems, such as those examined by Alper in \cite{Alp18} and by De Lellis, Marchese, Spadaro, and Valtorta in \cite{DMSV18}. Although up to limits in some sense, our model has some inner connections related to those two mentioned above (see \cite{DWW16} and \cite{WY24} for references), the problems for the stationary solutions of \eqref{MEMSeq} are distinguished from them. In our paper, we treat the rupture set $ \{u=0\} $ as the singular set and consider its stratification. We employ the monotonicity formula similar to that in \eqref{classDen}, while the authors of \cite{Alp18} and \cite{DMSV18} primarily utilize the Almgren frequency in their analysis. In \S\ref{Preliminaries} and \ref{converSec}, we introduce similar quantities to illustrate some properties from other perspectives, though we will not rely on them to prove the main theorems.

\subsection{Further results}

The equation \eqref{MEMSeq} is relatively straightforward, yet it opens up various avenues for generalization. In the study of thin films and MEMS problem, researchers have developed more complex equations than \eqref{MEMSeq} that provide comprehensive descriptions of the underlying physical and mathematical phenomena, as highlighted in \cite{EGG10} and \cite{LW17}. Another significant direction for generalization involves the multi-phase models discussed in \cite{ST19}. We anticipate that models exhibiting structural similarities to the stationary solutions of \eqref{MEMSeq}, as analyzed in our paper, could yield results comparable to those presented in Theorem \ref{main1}.

In Theorem \ref{thmmain2}, we specifically consider the case when $ n=3 $ and $ p>3 $. It remains an open problem to determine the rectifiability of the rupture set $ \{u(\cdot,t)=0\} $ for a.e. $ t\in(0,T) $ for suitable weak solutions as defined in Definition \ref{defsuita}. Additionally, while the energy density introduced in \cite{WY24} is nondecreasing, it may also achieve negative values, preventing the application of methods typically used in the study of harmonic heat flows, such as those detailed by Lin and Wang in \cite{LW02}. It is worth noting that preliminary results related to space-time estimates for the parabolic Minkowski content, which incur a loss of $ \va $-room, can be derived by using analogous findings from \cite{CHN13} and \cite{CHN15} by Cheeger, Haslhofer, and Naber.

\subsection{Organization of the paper} The paper is divided into two parts, and the structure is presented as follows.

\begin{itemize}
\item Part I is dedicated to establishing classical stratification results, in line with \cite{Whi97}.
\begin{itemize}
\item In \S \ref{Preliminaries}, we collect fundamental a priori estimates for solutions of \eqref{MEMSeq} and introduce the monotonicity formula.
\item In \S \ref{converSec}, we utilize the estimates developed in the previous section to perform a blow-up analysis and present some applications.
\item In \S \ref{StratificationSection}, using results in \S \ref{Preliminaries} and \S \ref{converSec}, we examine the classical stratification results of the rupture set $ \{u=0\} $.
\end{itemize}

\item Part II will focus on quantitative stratification theory, following the frameworks established in \cite{NV17} and \cite{NV18}.
\begin{itemize}
\item In \S \ref{SettingsQuantitative}, we provide definitions and outline main results of the quantitative stratification related to the problem \eqref{MEMSeq}.
\item In \S \ref{quantitatvePro}, we establish essential properties of quantitative stratification, which will be used in later proof.
\item In \S \ref{Reifenbergtype}, we review Reifenberg-type theorems in geometric measure theory and adapt them to better fit our application context.
\item In \S \ref{L2best}, we introduce the $ L^2 $-best approximation results, which illustrate the relationships between Reifenberg-type theorems and the monotonicity formula.
\item In \S \ref{CoveringSection}, we present a few covering lemmas derived from the results of the preceding sections. These lemmas play a crucial role in the proof of main theorems.
\item Finally, in \S \ref{Mainproof}, we employ the covering lemmas from \S \ref{CoveringSection} to complete the proof of main theorems.
\end{itemize}
\end{itemize}

\subsection{Notations and conventions}
\begin{itemize}
\item Throughout this paper, we will use $ C $ to denote positive constants. Sometimes to emphasize that $ C $ depends on parameters $ a,b,... $, we use the notation $ C(a,b,...) $, which may change from line to line.
\item We will use the Einstein summation convention in this paper, summing the repeated index without the sum symbol.
\item The inequality $ a>0 $ does not contain the possibility that $ a=+\ift $.
\item Let $ \beta\in(0,1] $ and $ K\subset\R^n $ be a compact set. A function $ f\in C^{0,\beta}(K) $ if 
$$
\|f\|_{C^{0,\beta}(K)}:=\|f\|_{L^{\ift}(K)}+[f]_{C^{0,\beta}(K)}<+\ift,
$$
where 
$$
[f]_{C^{0,\beta}(K)}:=\sup_{x,y\in K}\f{|f(x)-f(y)|}{|x-y|^{\beta}}. 
$$
Here $ C^{0,\beta}(K) $ is a Banach space equipped with the norm $ \|\cdot\|_{C^{0,\beta}(K)} $. For $ \om\subset\R^n $ an open set, we call $ f\in C_{\loc}^{0,\beta}(\om) $ if $ f\in C^{0,\beta}(K) $ for any compact set $ K\subset \om $. If $ \{f_i\}\subset C^{0,\beta}(K) $ with $ \sup_{i\in\Z_+}\|f_i\|_{C^{0,\beta}(K)}\leq\Lda $, then it follows from Arzel\`{a}-Ascoli lemma, up to subsequence, there exists $ f_{\ift}\in C^{0,\beta}(K) $ such that $ \|f\|_{C^{0,\al}(K)}\leq\Lda $ and $ \|f_i-f_{\ift}\|_{L^{\ift}(K)}\to 0 $. In this paper, we will frequently use this simple property.
\item Let $ \beta\in(0,1] $ and $ K\subset\R^{n+1} $ be a compact set. Assume that $ f=f(x,t):K\to\R $, where $ x\in\R^n $, and $ t\in\R $. $ f\in C^{\beta,\f{\beta}{2}}(K) $ if 
$$
\|f\|_{C^{\beta,\f{\beta}{2}}(K)}:=\|f\|_{L^{\ift}(K)}+[f]_{C^{\beta,\f{\beta}{2}}(K)}<+\ift,
$$
where 
$$
[f]_{C^{\beta,\f{\beta}{2}}(K)}:=\sup_{x,y\in K}\f{|f(x,t)-f(y,t)|}{(|x-y|+|t-s|^{\f{1}{2}})^{\beta}}. 
$$
Let $ \om\subset\R^{n+1} $ be an open set. $ f\in C_{\loc}^{\beta,\f{\beta}{2}}(\om) $ if $ f\in C^{\beta,\f{\beta}{2}}(K) $ for any compact set $ K\subset \om $.
\item For $ k\in\Z\cap[1,n] $, the Grassmannian $ \bG(n,k) $ is the set of all $ k $-dimensional subspaces of $ \R^n $, and $ \bA(n,k) $ is the collections of all $ k $-dimensional affine subspaces of $ \R^n $. For $ \{V\}\cup\{V_i\}\subset\bG(n,k) $, we say that $ V_i\to V $ as $ i\to+\ift $ if $ d_{\bG(n,k)}(V_i,V)\to 0 $ as $ i\to+\ift $.
where $ d_{\bG(n,k)}(\cdot,\cdot) $ is the Grassmannian metric. For $ \{L\}\cup\{L_i\}\subset\bA(n,k) $, we say $ L_i\to L $ if $ L_i=x_i+V_i $ and $ L=x+V $, where $ \{V\}\cup\{V_i\}\subset\bG(n,k) $, with $ V_i\to V $ and $ x_i\to x $. 
\item Letting $ \om\subset\R^n $ be any subset with $ x\in\om $ and $ r>0 $, we define 
$$
\eta_{x,r}(\om):=r^{-1}(\om-x):=\{y\in\R^n:x+ry\in\om\}.
$$
\item For open sets $ \{\om_i\} $ and $ \om $ with $ \om_i,\om\subset\R^n $, we say that $ \om_i\to \om $ if for any $ K\subset\subset \om $ and $ x\not\in \om $, when $ i\in\Z_+ $ is sufficiently large, there hold $
K\subset\subset \om_i $ and $ x\notin \om_i $.
\item For $ r>0 $, $ k\in\Z\cap[1,n] $, and $ x\in\R^k $, we let 
$$
B_r^k(x):=\{x\in\R^k:|y-x|<r\}.
$$
If $ k=n $, we drop the superscription. If $ x=0 $, we denote it by $ B_r^k $. $ \HH^k $ is $ k $-dimensional Hausdorff measure on $ \R^n $. We let $ \w_k:=\HH^k(B_1^k) $. When $ k=n $, we denote $ \cL^n=\HH^n $ as the Lebesgue measure and $ \ud\cL^n(x)=\ud x $. If no ambiguity occurs, we will drop $ \ud x $ in integrals. 
\item For a $ k $-dimensional subspace $ V=\op{span}\{v_i\}_{i=1}^k $, where $ \{v_i\}_{i=1}^k $ is an orthonormal basis and $ u\in H^1(\R^n) $, we set 
$$
|V\cdot\na u|^2=\sum_{i=1}^k|v_i\cdot\na u|^2.
$$
\item We have a convention that $ 0 $-dimensional affine subspaces refer to single points.
\item For a $ \cL^n $-measurable set $ A\subset\R^n $ with $ \cL^n(A)<+\ift $, and $ u\in L^1(A) $, we denote the average of integral of $ u $ on $ A $ by $ \dashint_Au:=\f{1}{\cL^n(A)}\int_{A}u $.
\end{itemize}

\part{Classical stratification} 

In this part, our primary goal is to establish the classical stratification results for stationary solutions of \eqref{MEMSeq} based on the work \cite{Whi97}. Before we finally achieve this, we provide some estimates and convergence results as fundamental ingredients.

\section{Preliminaries}\label{Preliminaries}

\subsection{Interior estimates of weak solution} 

Throughout this subsection, we assume that $ u\in(C_{\loc}^{0,\al}\cap H_{\loc}^1\cap L_{\loc}^{-p})(B_2) $ is a weak solution of \eqref{MEMSeq} with respect to $ f\in M_{\loc}^{2\al+n-4,2}(B_2) $, satisfying
\be
[u]_{C^{0,\al}(\ol{B}_1)}+[f]_{M^{2\al+n-4,2}(B_1)}\leq\Lda.\label{uHolderMorrey}
\ee
We will next present some basic interior estimates, which are the foundations of the proceeding analysis. In \cite{DWW16}, the authors established similar properties for the case where $ f\equiv 0 $, and here, we extend these results to the scenario where $ f $ is in the Morrey space. We demonstrate that the proofs of these interior estimates under the condition \eqref{uHolderMorrey} are relatively straightforward, as the assumption $ [u]_{C^{0,\al}(\ol{B}_1)}\leq\Lda $ is pretty strong. In contrast, establishing the a priori H\"{o}lder continuity of $ u $ under weaker regularity assumptions is significantly more complex. We will address this challenge and provide related results in Proposition \ref{proplowerbound}.

\begin{lem}\label{InESLem1}
If $ r>0 $ and $ x\in B_1 $ with $ B_{2r}(x)\subset B_1 $, then
$$
\int_{B_r(x)}(r^{\al}u^{-p}+u^{1-p})\leq Cr^{2\al+n-2},
$$
and
\be
\int_{B_r(x)}u\geq C^{-1}r^{\al+n},\label{nondege}
\ee
where $ C>0 $ depends only on $ \Lda,n $, and $ p $.
\end{lem}

\begin{rem}\label{remnonde}
We refer to \eqref{nondege} as the nondegeneracy property. Consequently, it follows that
\be
\sup_{B_r(x)}u\geq C(\Lda,n,p)^{-1}r^{\al}.\label{supnonde}
\ee
Additionally, there are analogous results on the obstacle problem $ \Delta u=\chi_{\{u>0\}} $, and one can see Lemma 5 of \cite{Caf98} as an example. 
\end{rem}

\begin{proof}[Proof of Lemma \ref{InESLem1}]
For given $ B_{2r}(x)\subset B_1 $, we choose $ \vp\in C_0^{\ift}(B_{2r}(x)) $ such that the following properties hold.
\begin{itemize}
\item $ \vp\equiv 1 $ in $ B_r(x) $, and $ 0\leq\vp\leq 1 $ in $ B_{2r}(x) $.
\item $ r|\na\vp|+r^2|D^2\vp|\leq C(n) $ in $ B_{2r}(x) $.
\end{itemize}
Testing \eqref{MEMSeq} by $ \vp $ gives 
\be
\int_{B_{2r}(x)}u^{-p}\vp=\int_{B_{2r}(x)}(u-u(x))\Delta\vp-\int_{B_{2r}(x)}f\vp.\label{testeta}
\ee
By \eqref{uHolderMorrey} and Cauchy's inequality, we have
\be
\int_{B_{2r}(x)}|f|\leq C\(\int_{B_{2r}(x)}|f|^2\)^{\f{1}{2}}r^{\f{n}{2}}\leq C(\Lda,n)r^{\al+n-2},\label{f2uholder1}
\ee
and
\be
\sup_{B_{2r}(x)}|u-u(x)|\leq \Lda (2r)^{\al}.\label{f2uholder2}
\ee
As a result, it follows from \eqref{testeta} that
$$
\int_{B_r(x)}u^{-p}\leq\int_{B_{2r}(x)}|f|\vp+\int_{B_{2r}(x)}|u-u(x)||\Delta\vp|\leq C(\Lda,n)r^{\al+n-2}.
$$
H\"{o}lder's inequality yields that
\begin{align*}
\int_{B_r(x)}u^{1-p}&\leq\(\int_{B_r(x)}u^{-p}\)^{\f{p-1}{p}}\cL^n(B_r(x))^{\f{1}{p}}\leq C(\Lda,n,p)r^{2\al+n-2},\\
\int_{B_r(x)}u&\geq\(\int_{B_r(x)}1\)^{\f{p+1}{p}}\(\int_{B_r(x)}u^{-p}\)^{-\f{1}{p}}\geq C(\Lda,n,p)^{-1}r^{\al+n}.
\end{align*}
Now we can complete the proof.
\end{proof}

\begin{lem}\label{InESLem2}
If $ r>0 $ and $ x\in B_1 $ with $ B_{4r}(x)\subset B_1 $, then
$$
\int_{B_r(x)}|\na u|^2\leq Cr^{2\al+n-2},
$$
where $ C>0 $ depends only on $ \Lda,n $, and $ p $.
\end{lem}
\begin{proof}
Let $ \vp $ be as in Lemma \ref{InESLem1}. We test \eqref{MEMSeq} with $ (u-u(x))\vp^2 $ and obtain
\begin{align*}
\int_{B_{2r}(x)}|\na u|^2\vp^2&+2\int_{B_{2r}(x)}(\na u\cdot\na\vp)(u-u(x))\vp\\
&+\int_{B_{2r}(x)}u^{-p}(u-u(x))\vp^2+\int_{B_{2r}(x)}f(u-u(x))\vp^2=0.
\end{align*}
According to Cauchy's inequality and that 
\be
ab\leq\delta a^2+(4\delta)^{-1}b^2\label{abdelta}
\ee
for any $ a,b,\delta>0 $, we deduce that for any $ 0<\delta<1 $,
\begin{align*}
\int_{B_{2r}(x)}|\na u|^2\vp^2&\leq \delta\int_{B_{2r}(x)}|\na u|^2\vp^2+C(n)\delta^{-1}\int_{B_{2r}(x)}|u-u(x)|^2|\na\vp|^2\\
&\quad\quad+\int_{B_{2r}(x)}u^{-p}|u-u(x)|\vp^2+\int_{B_{2r}(x)}|f||u-u(x)|\vp^2.
\end{align*}
Letting $ \delta=\f{1}{2} $, it follows from \eqref{f2uholder1}, \eqref{f2uholder2}, and Lemma \ref{InESLem1} that
$$
\int_{B_{2r}(x)}|\na u|^2\vp^2\leq C(\Lda,n,p)r^{2\al+n-2}.
$$
By the definition of $ \vp $, we complete the proof.
\end{proof}

\subsection{Monotonicity formula} In this subsection, we introduce several functionals and the monotonicity formula for stationary solutions of \eqref{MEMSeq}. These concepts will be crucial in effectively characterizing rupture sets. Similar to the approximate harmonic maps as presented in \cite{Mos05} and \cite{NV18}, it is essential to consider the influence of the function $ f $ on the equation. We also need to adjust the density as defined in \eqref{classDen}, compared to the scenario addressed in \cite{DWW16}, where $ f $ is assumed to vanish.

Another notable aspect of the various functionals discussed in this paper is that we apply a mollification process using cut-off functions with higher regularity than the characteristic functions. This technique enhances our capabilities in the analysis by yielding smoother functionals that are easier to work with. For similar applications of this approach, readers can see \cite{HSV19,NV24,Ved21,FWZ24} and the references therein. One of the most remarkable advantages of introducing such modification is that it can prevent the usage of unique continuation property. Interested readers can compare the arguments in our paper with those in \cite{NV18} to see the differences. To begin with, we give the following definition of the cut-off functions we used in this procedure.

\begin{defn}\label{defnofphi}
Let $ \phi:[0,+\ift)\to[0,+\ift) $ be a smooth function satisfying the following properties.
\begin{enumerate}[label=$(\theenumi)$]
\item $ \supp\phi\subset[0,10) $.
\item For any $ t\in[0,+\ift) $, $ \phi(t)\geq 0 $ and $ |\phi'(t)|\leq 100 $.
\item $ -2\leq\phi'(t)\leq -1 $ for any $ t\in[0,8] $.
\item For any $ t\in\R_+ $, $ \phi'(t)\leq 0 $.
\end{enumerate}
For $ x\in\R^n $, we define $ \phi_{x,r},\dot{\phi}_{x,r}:\R^n\to[0,+\ift) $ as
$$
\phi_{x,r}(y):=\phi\(\f{|y-x|^2}{r^2}\)\quad\text{and}\quad\dot{\phi}_{x,r}(y):=\phi'\(\f{|y-x|^2}{r^2}\).
$$
\end{defn}

\begin{rem}
Here, one can regard $ \phi_{x,r} $ as an approximation of $ \chi_{B_r(x)} $.
\end{rem}

Now, we present the mollified functionals as follows.

\begin{defn}\label{Defdifffuncti}
Let $ \ga>0 $ and $ \om\subset\R^n $ be a bounded domain. Assume that $ u\in(C_{\loc}^{0,\al}\cap H_{\loc}^1\cap L_{\loc}^{-p})(\om) $ is a weak solution of \eqref{MEMSeq} with respect to $ f\in M_{\loc}^{2\al+n-4+\ga,2}(\om) $.
Fix $ x\in\om $ and $ 0<r<\f{1}{10}\dist(x,\pa\om) $. We define the localized energy functionals related to \eqref{Energyfun} as
\be
\begin{aligned}
D(u;x,r)&:=r^{2-n}\int_{\R^n}(|\na u|^2+u^{1-p})\phi_{x,r},\\
D_f(u;x,r)&:=r^{2-n}\int_{\R^n}(|\na u|^2+u^{1-p}+fu)\phi_{x,r},\\
F(u;x,r)&:=r^{2-n}\int_{\R^n}\(\f{|\na u|^2}{2}-\f{u^{1-p}}{p-1}\)\phi_{x,r}.
\end{aligned}\label{DDfF}
\ee
The height functional is given by
$$
H(u;x,r):=-r^{-n}\int_{\R^n}u^2\dot{\phi}_{x,r}.
$$
The Almgren-type frequency is
\be
I_f(u;x,r):=\f{D_f(u;x,r)}{H(u;x,r)}.\label{Almgren}
\ee
Finally, the functionals referred to as mollified densities are given by
$$
\vt(u;x,r):=r^{-2\al}(F(u;x,r)-\al H(u;x,r))
$$
and
\be
\begin{aligned}
\vt_f(u;x,r):=\vt(u;x,r)&-\f{r^{2-2\al-n}}{n+2\al-2}\int_{\R^n}((y-x)\cdot\na u-\al u)f\phi_{x,r}\ud y\\
&-\f{2}{(n+2\al-2)^2}\int_0^r\(\rho^{-2\al-n-1}\int_{\R^n}|f|^2|y-x|^4\dot{\phi}_{x,\rho}\ud y\)\ud\rho.
\end{aligned}\label{thetafuxr}
\ee
\end{defn}

Several remarks on the above definitions are in order.

\begin{rem}
The choice of the density in \eqref{thetafuxr} is not unique. Indeed, we can change the constant $ \f{2}{(n+2\al-2)^2} $ in the third term on the right-hand side to another large number, and the monotonicity formula in Proposition \ref{MonFor} still holds. Such a change will not influence the proof of the main results in this paper.
\end{rem}

\begin{rem}
By Definition \ref{defnofphi}, the integrals of various functionals in Definition \ref{Defdifffuncti} are actually on the ball $ B_{10r}(x) $, which are well defined since $ 0<r<\f{1}{10}\dist(x,\pa\om) $. 
\end{rem}

\begin{rem}
Given the definition of $ \vt_f(u;x,r) $, the assumption $ f\in M_{\loc}^{2\al+n-4+\ga,2}(\om) $ is more or less necessary to ensure the integrability of the last two terms in \eqref{thetafuxr}.
\end{rem}

\begin{rem}
We can estimate the difference between $ \vt(u;x,r) $ and $ \vt_f(u;x,r) $. Precisely, it follows from the definition of $ \vt_f(u;x,r) $ in \eqref{thetafuxr} ad Cauchy's inequality that
\be
\begin{aligned}
&|\vt_f(u;x,r)-\vt(u;x,r)|\\
&\leq C\[\(r^{2-2\al-n}\int_{B_{10r}(x)}|\na u|^2\)^{\f{1}{2}}+\(r^{-2\al-n}\int_{B_{10r}(x)}u^2\)^{\f{1}{2}}\]\(r^{4-2\al-n}\int_{B_{10r}(x)}|f|^2\)^{\f{1}{2}}\\
&\quad\quad+C\[\int_0^r\(\rho^{-2\al-n+3}\int_{B_{10\rho}(x)}|f|^2\)\ud\rho\],
\end{aligned}\label{thetaftheta}
\ee
where $ C>0 $ depends only on $ n $ and $ p $.
\end{rem}

The proposition below gives the characterization of rupture sets for weak solutions of \eqref{MEMSeq} by using the functionals $ \vt(u;\cdot,\cdot) $, $ \vt_f(u;\cdot,\cdot) $, and $ I_f(u;\cdot,\cdot) $.

\begin{prop}\label{propupture}
Assume that $ u $ and $ f $ are the same as in Definition \ref{Defdifffuncti}. Let $ x\in\om $. The following properties hold.
\begin{enumerate}[label=$(\theenumi)$]
\item Let $ r_0>0 $ be a given constant and $ 0<r<\min\{r_0,\f{1}{10}\dist(x,\pa\om)\} $. Suppose that $ u $ and $ f $ satisfies that for some $ \Lda>0 $,
$$
[u]_{C^{0,\al}(\ol{B}_{10r}(x))}+[f]_{M^{2\al+n-4+\ga,2}(B_{10r}(x))}\leq\Lda.
$$
Then there exists $ C>0 $, depending only on $ \ga,\Lda,n,p $, and $ r_0 $ such that
\be
\max\{\vt(u;x,r),\vt_f(u;x,r)\}\leq C,\label{maxvtusr}
\ee
In addition, there is $ C_*>0 $, depending only on $ \ga,\Lda,n,p $, and $ r_0 $ such that
\be
\min\{\vt(u;x,r),\vt_f(u;x,r)\}<-C_*\,\,\Ra\,\,\inf_{B_r(x)}u\geq r^{\al}.\label{assMupper}
\ee
\item If $ u(x)>0 $, then
\be
\lim_{\rho\to 0^+}\vt(u;x,\rho)=\lim_{\rho\to 0^+}\vt_f(u;x,\rho)=-\ift,\label{limitiftminus}
\ee
and
\be
\lim_{\rho\to 0^+}I_f(u;x,\rho)=0.\label{Ifre1}
\ee
\item If $ u(x)=0 $, then
\be
\liminf_{\rho\to 0^+}\vt_f(u;x,\rho)=\liminf_{\rho\to 0^+}\vt(u;x,\rho)>-\ift,\label{limitlowerif}
\ee
and
\be
0<\liminf_{\rho\to 0^+}I_f(u;x,\rho)\leq 2\al\leq\limsup_{\rho\to 0^+}I_f(u;x,\rho).\label{Ifpr2}
\ee
\end{enumerate}
\end{prop}

\begin{rem}
Note that in \eqref{limitlowerif}, the limit of $ \vt_f(u;x,\rho) $ and $ \vt(u;x,\rho) $ may not exist for weak solutions of \eqref{MEMSeq}. Later in Proposition \ref{MonFor}, we will present the monotonicity formula, which implies that for stationary solutions, such limits exist.
\end{rem}

\begin{rem}
The property \eqref{assMupper} is vital in the proof of main theorems in this paper since it yields a quantitative characterization of the positive values of the solution $ u $. Intuitively speaking, this implies that if $ \vt(u;x,r) $ or $ \vt_f(u;x,r) $ is sufficiently negatively large, then $ u $ must have a stronger nondegeneracy property in the ball $ B_r(x) $ than those presented in Lemma \ref{InESLem1}.
\end{rem}

Before the proof of Proposition \ref{propupture}, we first establish the following lemma on the connections between $ H(u;x,r) $, $ D_f(u;x,r) $, and $ I_f(u;x,r) $. There are similar calculations and results for different scenarios in the study of nonlinear elliptic equations (see \cite{Alp18,Alp20,DMSV18,TT12} for example).

\begin{lem}\label{lemHI}
Assume that $ u $ and $ f $ are the same as in Definition \ref{Defdifffuncti}. If $ x\in\om $ and $ 0<r<\f{1}{10}\dist(x,\pa\om) $, then
\begin{align}
\f{\ud}{\ud r}H(u;x,r)&=r^{-1}D_f(u;x,r),\label{Mon1}\\
\f{\ud}{\ud r}\log(H(u;x,r))&=r^{-1}I_f(u;x,r).\label{Mon2}
\end{align}
\end{lem}

\begin{proof}
By the definition of $ I_f(u;x,r) $, \eqref{Mon2} is a direct consequence of \eqref{Mon1}, so we only need to show \eqref{Mon1}. Testing \eqref{WeakConMEMS} with $ u\phi_{x,r} $, we obtain 
\be
\int_{\R^n}(|\na u|^2+u^{1-p}+fu)\phi_{x,r}+2r^{-2}\int_{\R^n}((y-x)\cdot\na u)u\dot{\phi}_{x,r}=0,\label{For21}
\ee
As a result, it follow from the definition of $ D_f(u;x,r) $ that
$$
D_f(u;x,r)=-2r^{-n}\int_{\R^n}((y-x)\cdot\na u)u\dot{\phi}_{x,r}.
$$
Through direct calculations, there holds
\begin{align*}
\f{\ud}{\ud r}H(u;x,r)&=-2r^{-n-1}\int_{\R^n} ((y-x)\cdot\na u)u\dot{\phi}_{x,r},
\end{align*}
which implies \eqref{Mon1} and completes the proof.
\end{proof}

\begin{rem}
If $ f\equiv 0 $, the formulae \eqref{Mon1} and \eqref{Mon2} imply that $ H(u;x,r) $ is nondecreasing with respect to the variable $ r\in(0,\f{1}{10}\dist(x,\pa\om)) $.
\end{rem}

\begin{proof}[Proof of Proposition \ref{propupture}]
Fix $ x\in\om $ and $ 0<r<\min\{r_0,\f{1}{10}\dist(x,\pa\om)\} $ with a given number $ r_0>0 $. We assume that for some $ \Lda>0 $,
\be
[u]_{C^{0,\al}(\ol{B}_{10r}(x))}+[f]_{M^{2\al+n-4+\ga,2}(B_{10r}(x))}\leq\Lda.\label{Ldaprimeuf}
\ee
By simple calculations, we have, for any $ 0<\rho\leq r $, 
\begin{align}
\int_{B_{8\rho}(x)}u^2&\geq\f{1}{2}\w_n(8\rho)^nu(x)^2-\int_{B_{\rho}(x)}(u-u(x))^2\geq\f{1}{2}\w_n(8\rho)^n(u(x)^2-2\Lda^2(8\rho)^{2\al}),\label{u8rhobales}\\
\int_{B_{10\rho}(x)}u^2&\leq 2\w_n(8\rho)^nu(x)^2+2\int_{B_{\rho}(x)}(u-u(x))^2\leq 2\w_n(10\rho)^n(u(x)^2+\Lda^2(10\rho)^{2\al}).\label{u10rhobales}
\end{align}
Using \eqref{abdelta}, \eqref{Ldaprimeuf}, \eqref{u8rhobales}, \eqref{u10rhobales}, Lemma \ref{InESLem1}, \ref{InESLem2}, and Definition \ref{defnofphi}, we obtain that for any $ 0<\rho\leq r $,
\be
\begin{aligned}
\vt_f(u;x,\rho)&\leq C\rho^{2-2\al-n}\int_{B_{10\rho}(x)}(|\na u|^2+u^{1-p})-C'\rho^{-2\al-n}\int_{B_{8\rho}(x)}u^2\\
&\quad\quad+C\delta\rho^{-2\al-n}\int_{B_{10\rho}(x)}u^2+C\delta^{-1}\rho^{4-2\al-n}\int_{B_{10\rho}(x)}|f|^2\\
&\quad\quad+C\int_0^{\rho}\(t^{-2\al-n+3}\int_{B_{10t}(x)}|f|^2\)\ud t\\
&\leq C(\ga,\Lda,n,p,r_0)+C(\ga,\Lda,n,p)\delta^{-1}\rho^{\ga}-(C'(n,p)-C''(n,p)\delta)\rho^{-2\al}u(x)^2,
\end{aligned}\label{vtfesti}
\ee
and
\be
\begin{aligned}
\vt(u;x,\rho)&\leq C\rho^{2-2\al-n}\int_{B_{10\rho}(x)}(|\na u|^2+u^{1-p})-C'\rho^{-2\al-n}\int_{B_{8\rho}(x)}u^2\\
&\leq C(\ga,\Lda,n,p,r_0)-C'(n,p)\rho^{-2\al}u(x)^2.
\end{aligned}\label{vtfesti2}
\ee
Choosing $ \delta=\delta(n,p)>0 $ sufficiently small in \eqref{vtfesti}, we can deduce from \eqref{vtfesti} and \eqref{vtfesti2} that
\begin{align*}
\vt(u;x,\rho)&\leq C(\ga,\Lda,n,p,r_0)-C'(n,p)\rho^{-2\al}u(x)^2,\\
\vt_f(u;x,\rho)&\leq C(\ga,\Lda,n,p,r_0)-\f{1}{2}C'(n,p)\rho^{-2\al}u(x)^2+C(\ga,\Lda,n,p)\rho^{\ga}.
\end{align*}
As a result, the inequality \eqref{maxvtusr} follows directly. We now turn to prove \eqref{assMupper}. Indeed, if $ \inf_{B_r(x)}u<r^{\al} $, then the assumption \eqref{Ldaprimeuf} implies that
$$
0\leq\inf_{B_r(x)}u\leq \sup_{B_r(x)}u\leq C(\Lda,n,p)r^{\al}.
$$
Given \eqref{Ldaprimeuf}, we can apply Lemma \ref{InESLem1} and \ref{InESLem2} to obtain 
\be
0\leq F(u;x,r),H(u;x,r)\leq C(\ga,\Lda,n,p,r_0).\label{FuxrFxrim}
\ee
Incorporated with the definitions of $ \vt(u;x,r) $, $ \vt_f(u;x,r) $, and \eqref{thetaftheta}, the estimate \eqref{FuxrFxrim} gives 
$$
\min\{\vt(u;x,r),\vt_f(u;x,r)\}\geq -C(\ga,\Lda,n,p,r_0)
$$
for any $ 0<r<\min\{r_0,\f{1}{10}\dist(x,\pa\om)\} $. If in \eqref{assMupper}, $ C_*>0 $ is sufficiently large, then it is a contradiction.

Assume that $ u(x)>0 $. We first note that the two inequalities \eqref{vtfesti} and \eqref{vtfesti2} directly imply \eqref{limitiftminus} as $ \rho\to 0^+ $. If for \eqref{u8rhobales} and \eqref{u10rhobales}, we choose 
\be
0<\rho<\min\left\{r,\f{1}{8}\(\f{1}{2\Lda}u(x)\)^{\f{1}{\al}}\right\},\label{rassuse}
\ee
then it follows that
\be
0<\f{\int_{B_{10\rho}(x)}u^2}{\int_{B_{8\rho}(x)}u^2}\leq\f{C(n)(u(x)^2+\Lda^2(10\rho)^{2\al})}{(u(x)^2-2\Lda^2(8\rho)^{2\al})}\leq C(n).\label{u8u10}   
\ee
Using Lemma \ref{InESLem1}, \ref{InESLem2}, Definition \ref{defnofphi}, and Cauchy's inequality, we have
\begin{align*}
0&\leq I_f(u;x,\rho)=\f{\rho^{2-2\al-n}\int(|\na u|^2+u^{1-p}+fu)\phi_{x,\rho}}{-\rho^{-2\al-n}\int u^2\dot{\phi}_{x,\rho}}\\
&\leq\f{C(\ga,\Lda,n,p,r_0)+C(n,p)\(\rho^{-2\al-n}\int_{B_{10\rho}(x)}u^2\)^{\f{1}{2}}\(\rho^{4-2\al-n}\int_{B_{10\rho}(x)}|f|^2\)^{\f{1}{2}}}{\rho^{-2\al-n}\int_{B_{8\rho}(x)}u^2}
\end{align*}
for any $ 0<\rho<r $. Given \eqref{Ldaprimeuf} and \eqref{u8u10}, it follows that
$$
0\leq I_f(u;x,\rho)\leq C(\ga,\Lda,n,p,r_0)\[\f{1}{\rho^{-2\al-n}\int_{B_{8\rho}(x)}u^2}+\f{\rho^{\f{\ga}{2}}}{\(\rho^{-2\al-n}\int_{B_{8\rho}(x)}u^2\)^{\f{1}{2}}}\],
$$
where $ \rho>0 $ satisfies \eqref{rassuse}. This, together with \eqref{u8rhobales}, implies \eqref{Ifre1}. 

If $ u(x)=0 $, then we obtain that for any $ \rho\in(0,r] $,
\be
\rho^{-2\al-n}\int_{B_{10\rho}(x)}u^2\leq C(\Lda,n,p).\label{uL2small}
\ee
By \eqref{thetaftheta} and \eqref{uL2small}, we have
\be
|\vt_f(u;x,\rho)-\vt(u;x,\rho)|\leq C(\ga,\Lda,n,p,r_0)\rho^{\ga}.\label{twodenminus}
\ee
It follows from \eqref{uL2small}, Lemma \ref{InESLem1}, and Lemma \ref{InESLem2} that
$ |\vt(u;x,\rho)|\leq C(\ga,\Lda,n,p,r_0) $. As a result, \eqref{limitlowerif} holds. Finally, it remains to show \eqref{Ifpr2}. If
$$
\liminf_{\rho\to 0^+}I_f(u;x,\rho)>2\al,
$$
then Lemma \ref{lemHI} gives 
$$
\f{\ud}{\ud\rho}\log(H(u;x,\rho))>2\al+\va
$$
for some $ \va>0 $ with $ 0<\rho\leq\rho_0 $, where $ \rho_0>0 $ is a sufficiently small constant. Consequently,
\be
H(u;x,\rho)\leq H(u;x,\rho_0)\(\f{\rho}{\rho_0}\)^{2\al+\va}.\label{Huxupper}
\ee
However, due to Lemma \ref{InESLem1}, there holds $
H(u;x,\rho)\geq C(\Lda,n,p)\rho^{2\al} $. Letting $ \rho\to 0^+ $, it is a contradiction to \eqref{Huxupper}. Similarly, if 
$$
\limsup_{\rho\to 0^+}I_f(u;x,\rho)<2\al,
$$
then there exists $ \va'>0 $ and $ \rho_0'>0 $ such that
$$
\f{\ud}{\ud\rho}\log(H(u;x,\rho))<2\al-\va'
$$
for any $ 0<\rho<\rho_0' $. Then
\be
H(u;x,\rho)\geq H(u;x,\rho_0')\(\f{\rho}{\rho_0'}\)^{2\al-\va'}.\label{Huxupper2}
\ee
Since $ u(x)=0 $, with the help of \eqref{Ldaprimeuf} for any $ 0<\rho<\rho_0' $, $
H(u;x,\rho)\leq C(\Lda,n,p)\rho^{2\al} $, which is contradictory to \eqref{Huxupper2} as $ \rho\to 0^+ $. Combining all the above, we complete the proof.
\end{proof}

The following proposition gives the monotonicity formula for the density of stationary solutions of \eqref{MEMSeq}. By this, we can define the limit of $ \vt(u;x,r) $ and $ \vt_f(u;x,r) $ as $ r\to 0^+ $.

\begin{prop}[Monotonicity formula]\label{MonFor}
Let $ \ga>0 $ and assume that $ u\in(C_{\loc}^{0,\al}\cap H_{\loc}^1\cap L_{\loc}^{-p})(\om) $ is a stationary solution of \eqref{MEMSeq} with respect to $ f\in M_{\loc}^{2\al+n-4+\ga,2}(\om) $. If $ x\in\om $ and $ 0<r<\f{1}{10}\dist(x,\pa\om) $, then 
\be
\f{\ud}{\ud r}\vt_f(u;x,r)\geq-r^{-2\al-n-1}\int_{\R^n}|(y-x)\cdot\na u-\al u|^2\dot{\phi}_{x,r}\ud y\geq 0,\label{udthetaudr}
\ee
In particular, $ \vt_f(u;x,\cdot) $ is nondecreasing in $ (0,\f{1}{10}\dist(x,\pa\om)) $. 

Consequently, we can define the limit
\be
\vt(u;x):=\lim_{r\to 0^+}\vt_f(u;x,r).\label{densitylimit}
\ee
In view of \eqref{thetaftheta} and Proposition \ref{propupture}, we have 
$$
\vt(u;x)=\lim_{r\to 0^+}\vt(u;x,r).
$$
\end{prop}

\begin{rem}
The Monotonicity formula plays significant roles in various geometric variational problems. It is a consequence of the Pohozaev-type identity. For solutions of \eqref{MEMSeq} with higher regularity, indeed $ H^2 $ (see \cite{Wan12} for references), it is a consequence by testing \eqref{WeakConMEMS} by $ y\cdot\na u $ and integration by parts. On the other hand, for stationary solutions, the condition \eqref{StaConMEMS} implies the Pohozaev-type identity.
\end{rem}

\begin{rem}\label{udthetaudrem}
If $ f\equiv 0 $, the inequality \eqref{udthetaudr} is changed by
\be
\f{\ud}{\ud r}\vt(u;x,r)=-2r^{-2\al-n-1}\int_{\R^n}|(y-x)\cdot\na u-\al u|^2\dot{\phi}_{x,r}\ud y\geq 0,\label{udthetaudr2}
\ee
which is analogous to the classical one given in \eqref{classDenfor}.
\end{rem}

\begin{proof}[Proof of Proposition \ref{MonFor}]
We will employ the arguments in \S 2 of \cite{GW06} to show this monotonicity formula. One can also refer to a similar proof for semilinear elliptic equations presented in Proposition 2.2 of \cite{FWZ24}. Up to a translation, we assume that $ x=0\in\om $. Define $ n_{r}(y):=\f{y}{r} $ with $ n_{r}^i(y)=\f{y_i}{r} $ for $ i\in\Z\cap[1,n] $, $ \phi_{r}:=\phi_{0,r} $, and $ \pa_{r}u:=n_{r}\cdot\na u $. We also denote 
\begin{align*}
D(r)&:=D(u;0,r),\quad F(r):=F(u;0,r),\quad\vt(r):=\vt(u;0,r),\\
D_f(r)&:=D_f(u;0,r),\quad H(r):=H(u;0,r),\quad\text{and}\quad\vt_f(r):=\vt_f(u;0,r).
\end{align*}
It follows from \eqref{DDfF} and simple computations that
\be
\begin{aligned}
D(r)-2F(r)&=\f{r^{2-n}}{1-\al}\int u^{1-p}\phi_r,\\
D(r)+(p-1)F(r)&=\f{r^{2-n}}{\al}\int|\na u|^2\phi_r.
\end{aligned}\label{DFrelation}
\ee
By Definition \ref{defnofphi}, the vector field $ \phi_{r}(y)y $ is Lipschitz and compactly supported in $B_{10r}(x) $ with $ 0<r<\f{1}{10}\dist(0,\pa\om) $. By approximating arguments, we can apply the stationary condition \eqref{StaConMEMS} with $ Y(y)=\phi_{r}(y)y $ and deduce that
\be
\begin{aligned}
\f{n-2}{2}\int|\na u|^2\phi_{r}&-\f{n}{p-1}\int u^{1-p}\phi_{r}+\int|\na u|^2|n_{r}|^2\dot{\phi}_{r}\\
&-\f{2}{p-1}\int u^{1-p}|n_{r}|^2\dot{\phi}_{r}-2\int|\pa_{r}u|^2\dot{\phi}_{r}-r\int f\pa_{r}u\phi_{r}=0.
\end{aligned}\label{For1}
\ee
Here, for the sake of simplicity, we use $ \int\cdot $ to represent $ \int_{\R^n}\cdot $ and will adopt such a notation throughout this proof. Like \eqref{For21}, by testing \eqref{WeakConMEMS} with $ u\phi_{r} $, we have
\be
\int|\na u|^2\phi_{r}+\int fu\phi_{r}+\int u^{1-p}\phi_{r}+2r^{-1}\int u\pa_{r}u\dot{\phi}_{r}=0.\label{For2}
\ee
Taking derivatives for both sides of \eqref{For2} yields that
$$
r^{-1}\int|\na u|^2|n_{r}|^2\dot{\phi}_{r}+r^{-1}\int fu|n_{r}|^2\dot{\phi}_{r}+r^{-1}\int u^{1-p}|n_{r}|^2\dot{\phi}_{r}+\f{1}{2}\f{\ud}{\ud r}(r^{n-2}D_f(r))=0.
$$
Incorporating with \eqref{For1} and \eqref{For2}, we can eliminate terms
$$
\int|\na u|^2\phi_r\quad\text{and}\quad\int|\na u|^2|n_r|^2\dot{\phi}_r
$$
in the two formulae and obtain 
\begin{align*}
\f{r^{-1}}{1-\al}\int u^{1-p}|n_{r}|^2\dot{\phi}_{r}&+r^{-1}\(\f{n-2}{2}+\f{n}{p-1}\)\int u^{1-p}\phi_{r}\\
&+r^{-1}\int\(2|\pa_{r}u|^2+\f{(n-2)u\pa_{r}u}{r}\)\dot{\phi}_{r}+\f{1}{2}\f{\ud}{\ud r}\(r^{n-2}D_f(r)\)\\
&+r^{-1}\(\f{n-2}{2}\)\int fu\phi_{r}+\int f\pa_{r}u\phi_{r}+r^{-1}\int fu|n_{r}|^2\dot{\phi}_{r}=0.
\end{align*}
We multiply both sides of the above by $ r^{2-2\al-n}$ to obtain
\be
\begin{aligned}
\f{r^{2-2\al-n}}{2}\f{\ud}{\ud r}(r^{n-2}D_f(r))&+\f{1}{2}\f{\ud}{\ud r}(r^{-2\al}(2F(r)-D(r)))\\
&+r^{1-2\al-n}\int\(2|\pa_{r}u|^2+\f{(n-2)u\pa_{r}u}{r}\)\dot{\phi}_{r}\\
&+r^{2-2\al-n}\int f\pa_{r}u\phi_{r}+\f{(n-2)r^{1-2\al-n}}{2}\int fu\phi_{r}\\
&+r^{1-2\al-n}\int fu|n_{r}|^2\dot{\phi}_{r}=0.
\end{aligned}\label{For4}
\ee
Then \eqref{Mon1} implies that
\begin{align}
\f{\ud}{\ud r}(r^{1-2\al}H(r))&=r^{1-2\al}\f{\ud}{\ud r}H(r)-(2\al-1)r^{-2\al}H(r)=r^{-2\al}(D_f(r)-(2\al-1)H(r)),\label{For5}\\
\f{\ud}{\ud r}(r^{-2\al}H(r))&=r^{-2\al}\f{\ud}{\ud r}H(r)-2\al r^{-2\al-1}H(r)=r^{-2\al-1}(D_f(r)-2\al H(r)),\label{For6}
\end{align}
and then
\begin{align*}
\f{1}{2}\f{\ud^2}{\ud r^2}(r^{1-2\al}H(r))&=\f{1}{2}\f{\ud}{\ud r}(r^{-2\al}D_f(r))-\f{(2\al-1)r^{-2\al-1}}{2}(D_f(r)-2\al H(r))\\
&=\f{1}{2}\f{\ud}{\ud r}(r^{-2\al}D_f(r))+(2\al-1)r^{1-2\al-n}\int\(\f{u\pa_{r}u}{r}-\f{\al u^2}{r^2}\)\dot{\phi}_{r}.
\end{align*}
Consequently, we infer from \eqref{For4} that
\begin{align*}
\f{1}{2}\f{\ud^2}{\ud r^2}(r^{1-2\al}H(r))&+\f{1}{2}\f{\ud}{\ud r}(r^{-2\al}(2F(r)-D(r)))\\
&+r^{1-2\al-n}\int\(2|\pa_{r}u|^2+\f{(1-4\al)u\pa_{r}u}{r}-\f{(1-2\al)u^2}{\al r^2}\)\dot{\phi}_{r}\\
&+r^{2-2\al-n}\int f\pa_{r}u\phi_{r}+\f{(n-2)r^{1-2\al-n}}{2}\int fu\phi_{r}\\
&+r^{1-2\al-n}\int fu|n_{r}|^2\dot{\phi}_{r}=0.
\end{align*}
Combined with \eqref{For5} and \eqref{For6}, this implies that
\begin{align*}
&2r^{1-2\al-n}\int|\pa_{r}u-r^{-1}\al u|^2\dot{\phi}_{r}+r^{2-2\al-n}\int f\pa_{r}u\phi_{r}\\
&\quad\quad+\f{(n-2)r^{1-2\al-n}}{2}\int fu\phi_{r}+r^{1-2\al-n}\int fu|n_{r}|^2\dot{\phi}_{r}\\
&\quad\quad\quad\quad=\f{\ud}{\ud r}\[r^{-2\al}\(\f{D(r)}{2}-F(r)\)+\f{r^{-2\al}}{2}H(r)-\f{\ud}{\ud r}\(\f{r^{1-2\al}}{2}H(r)\)\]\\
&\quad\quad\quad\quad=\f{\ud}{\ud r}\[r^{-2\al}\(\f{D(r)}{2}-F(r)\)+\al r^{-2\al}H(r)-\f{r^{-2\al}}{2}D_f(r)\].
\end{align*}
Recalling that $
\vt(r)=r^{-2\al}(F(r)-\al H(r)) $, it can be easily seen from \eqref{DFrelation} that
\begin{align*}
\f{\ud}{\ud r}\vt(r)&=-2r^{1-2\al-n}\int|\pa_{r}u-r^{-1}\al u|^2\dot{\phi}_{r}\\
&\quad\quad+\al r^{1-2\al-n}\int fu\phi_{r}-r^{2-2\al-n}\int f\pa_{r}u\phi_{r}.
\end{align*}
For a specific case that $ f\equiv 0 $, this equality echoes Remark \ref{udthetaudrem}. By direct calculations and the inequality \eqref{abdelta} with $ \delta=\f{1}{2} $, there holds
\begin{align*}
\f{\ud}{\ud r}\vt_f(r)&=\f{\ud}{\ud r}\vt(r)+r^{1-2\al-n}\int(y\cdot\na u-\al u)f\phi_{r}\\
&\quad\quad-\f{2r^{-2\al-n-1}}{n+2\al-2}\int(y\cdot\na u-\al u)f|y|^2\dot{\phi}_{r}+\f{2r^{-2\al-n-1}}{(n+2\al-2)^2}\int|f|^2|y|^4\dot{\phi}_{r}\\
&=-2r^{-2\al-n-1}\int|y\cdot\na u-\al u|^2\dot{\phi}_{r}-\f{2r^{-2\al-n-1}}{n+2\al-2}\int(y\cdot\na u-\al u)f|y|^2\dot{\phi}_{r}\\
&\quad\quad-\f{2r^{-2\al-n-1}}{(n+2\al-2)^2}\int|f|^2|y|^4\dot{\phi}_{r}\\
&\geq-r^{-2\al-n-1}\int|y\cdot\na u-\al u|^2\dot{\phi}_{r},
\end{align*}
and we conclude the proof.    
\end{proof}

Finally, we end this subsection with a direct consequence of Proposition \ref{MonFor}. Note that one cannot obtain such a property without mollifying densities with cut-off functions.

\begin{cor}\label{coruse}
Assume that $ u $ and $ f $ are the same as in Proposition \ref{MonFor}. Let $ x\in\om $ and $ 0<r<\f{1}{10}\dist(x,\pa\om) $. Then
$$
\vt_f(u;x,r)-\vt_f\(u;x,\f{r}{2}\)\geq Cr^{-2\al-n}\int_{B_{4r}(x)}|(y-x)\cdot\na u-\al u|^2\ud y,
$$
where $ C>0 $ depends only on $ n $ and $ p $.
\end{cor}
\begin{proof}
Using Definition \ref{defnofphi} and Proposition \ref{MonFor}, we have
\begin{align*}
\vt_f(u;x,r)-\vt_f\(u;x,\f{r}{2}\)&\geq C\int_{\f{r}{2}}^{r}\(\rho^{-2\al-n-1}\int_{B_{8\rho}(x)}|(y-x)\cdot\na u-\al u|^2\ud y\)\ud\rho\\
&\geq Cr^{-2\al-n}\int_{B_{4r}(x)}|(y-x)\cdot\na u-\al u|^2\ud y
\end{align*}
for any $ 0<r<\f{1}{10}\dist(x,\pa\om) $, which implies the result.
\end{proof}

\section{Compactness and blow-up analysis}\label{converSec}

We initially consider the convergence properties for a sequence of solutions of \eqref{MEMSeq}. Subsequently, we will establish the blow-up analysis of our model based on such properties. These convergence results form the foundation for further discussions, as most of the proofs in this paper heavily rely on compactness arguments, which involve procedures of taking limits for solutions in some specific sense. Finally, at the end of this section, we will present some preliminary applications of the findings in the previous subsections.

\subsection{Compactness results} Let $ R>r>0 $ and $ x\in\R^n $. Suppose in this subsection that $ \{u_i\}\subset(C_{\loc}^{0,\al}\cap H_{\loc}^1\cap L_{\loc}^{-p})(B_R(x)) $ is a sequence of weak solutions of \eqref{MEMSeq} with respect to $ \{f_i\}\subset M_{\loc}^{2\al+n-4,2}(B_R(x)) $, satisfying
\be
\sup_{i\in\Z_+}\([u_i]_{C^{0,\al}(\ol{B}_r(x))}+\|f_i\|_{M^{2\al+n-4,2}(B_r(x))}\)<+\ift.\label{uiAssori}
\ee
We examine two different types of convergence properties of $ \{u_i\} $, based on the uniform boundedness of $ \|u_i\|_{L^2(B_r(x))} $. Specifically, we present the following two propositions.

\begin{prop}\label{propConv}
Assume that
\be
\sup_{i\in\Z_+}\|u_i\|_{L^2(B_r(x))}<+\ift.\label{uiL2bounded}
\ee
Then there exist $ u_{\ift}\in C^{0,\al}(\ol{B}_r(x))\cap(H_{\loc}^1\cap L_{\loc}^{-p})(B_r(x)) $ and $ f_{\ift}\in M^{2\al+n-4,2}(B_r(x)) $ such that up to a subsequence, the following properties hold.
\begin{enumerate}[label=$(\theenumi)$]
\item $ u_i\to u_{\ift} $ strongly in $ (H_{\loc}^1\cap L^{\ift})(B_r(x)) $.
\item $ u_i^{-p}\to u_{\ift}^{-p} $, $ u_i^{1-p}\to u_{\ift}^{1-p} $ strongly in $ L_{\loc}^1(B_r(x)) $.
\item $ f_i\wc f_{\ift} $ weakly in $ L^2(B_r(x)) $.
\end{enumerate}
Moreover, $ u_{\ift} $ is a weak solution of \eqref{MEMSeq} with respect to $ f_{\ift} $. If for any $ i\in\Z_+ $, $ u_i $ is a stationary solution with respect to $ f_i $, then $ u_{\ift} $ is a stationary solution with respect to $ f_{\ift} $.
\end{prop}

\begin{prop}\label{propConv2}
Assume that $ \{x_i\}\subset B_r(x) $ and
\be
\sup_{i\in\Z_+}\|u_i\|_{L^2(B_r(x))}=+\ift.\label{uiL2bounded2}
\ee
Then there exist $ v_{\ift}\in C^{0,\al}(\ol{B}_r(x))\cap H_{\loc}^1(B_r(x)) $ and $ f_{\ift}\in M^{2\al+n-4,2}(B_r(x)) $ such that up to a subsequence, the following properties hold,
\begin{enumerate}[label=$(\theenumi)$]
\item $ v_i:=u_i-u_i(x_i)\to v_{\ift} $ strongly in $ (H_{\loc}^1\cap L^{\ift})(B_r(x)) $.
\item $ f_i\wc f_{\ift} $ weakly in $ L^2(B_r(x)) $.
\end{enumerate}
Moreover, $ \Delta v_{\ift}=f_{\ift} $ in the sense of distribution in $ B_r(x) $. Precisely, for any $ \vp\in C_0^{\ift}(B_r(x)) $,
\be
\int_{B_r(x)}(\na v_{\ift}\cdot\na\vp+f_{\ift}\vp)=0.\label{weakuift}
\ee
\end{prop}

\begin{rem}\label{uiremiftcon}
Since $ [u_i]_{C^{0,\al}(\ol{B}_r(x))} $ is uniformly bounded, if we replace the assumption \eqref{uiL2bounded} with $ \{u_i=0\}\cap B_r(x)\neq\emptyset $ for any $ i\in\Z_+ $, results in the proposition still hold.
\end{rem}

\begin{rem}
One can view Proposition \ref{propConv} as a generalized form of properties in \S 4 of \cite{DWW16} on the study of \eqref{MEMSeq} with $ f\equiv 0 $. The proof is similar, with adjustments made to account for the influence of $ f $. On the other hand, for the sequence $ \{u_i\} $ itself in Proposition \ref{propConv2}, any subsequence does not have a limit. Indeed, by simple analysis, we will see that $ u_i\to+\ift $ in $ B_r(x) $ uniformly. As a result, considering limits of the sequence $ \{v_i\}=\{u_i-u_i(x_i)\} $ is necessary.
\end{rem}

\begin{rem}
There are other ways to construct a precompact sequence from $ u_i $. Here, ``precompact" is to say that there exists a subsequence converging in some specific regimes. For instance, consider the sequence $ \{w_i\} $ defined by
$$
w_i(y):=\f{r^nu_i(y)}{\|u_i\|_{L^1(B_r(x))}},\quad y\in B_r(x).
$$
It can be easily seen from Lemma \ref{InESLem1} and \ref{InESLem2} that for any $ 0<s<r $, $ \|w_i\|_{H^1(B_s)} $ is uniformly bounded. As a result, up to a subsequence, it has a limit, at least in a weak sense. We believe that this limit shares some connections with those presented in Proposition \ref{propConv} and \ref{propConv2} but will not go further since it is not highly related to the main theme of this paper. We refer to readers \cite{Alp18,Alp20,ST19} for similar methods on some other models.
\end{rem}

\subsubsection{Proof of Proposition \ref{propConv}}

Without loss of generality, we let $ r=4 $, $ R=8 $, and $ x=0 $. Given \eqref{uiAssori} and \eqref{uiL2bounded}, we assume that for some $ \Lda>0 $,
\be
\sup_{i\in\Z_+}\(\|u_i\|_{C^{0,\al}(\ol{B}_4)}+\|f_i\|_{L^2(B_4)}\)\leq\Lda.\label{uiAss}
\ee
It follows from Lemma \ref{InESLem2} that for any $ s\in(0,4) $,
\be
\sup_{i\in\Z_+}\|\na u_i\|_{L^2(B_s)}\leq C(\Lda,n,p,s).\label{uiAssnablsu}
\ee
As a result, there exist $ u_{\ift}\in C^{0,\al}(\ol{B}_4)\cap H_{\loc}^1(B_4) $ and $ f_{\ift}\in L^2(B_4) $ such that up to a subsequence, there are
\be
\begin{aligned}
u_i&\to u_{\ift}\text{ strongly in }L^{\ift}(B_4),\\
u_i&\wc u_{\ift}\text{ weakly in }H_{\loc}^1(B_4),\\
f_i&\wc f_{\ift}\text{ weakly in }L^2(B_4).
\end{aligned}\label{ConvPre}
\ee
Then $ f_{\ift}\in M^{2\al+n-4,2}(B_4) $, due to Lemma \ref{MorreyL2}. Thus, Proposition \ref{propConv} is the direct consequence of the following lemma.

\begin{lem}\label{ConvLem1}
For $ \{u_i\},u_{\ift} $, and $ f_{\ift} $ in \eqref{ConvPre}, we have
\begin{gather}
u_i^{-p}\to u_{\ift}^{-p},\,\,u_i^{1-p}\to u_{\ift}^{1-p}\text{ strongly in }L_{\loc}^1(B_4),\label{ConPro1}\\
u_i\to u_{\ift}\text{ strongly in }H_{\loc}^1(B_4).\label{ConPro2}
\end{gather}
Additionally, $ u_{\ift} $ is a weak solution of \eqref{MEMSeq} with respect to $ f_{\ift} $. If for any $ i\in\Z_+ $, $ u_i $ is a stationary solution with respect to $ f_i $, then $ u_{\ift} $ is a stationary solution with respect to $ f_{\ift} $.
\end{lem}

Before we show this lemma, we need the result below, which gives the estimate of $ \HH^{\al+n-2} $-measure for $ \{u_{\ift}=0\} $.

\begin{lem}\label{uiftzero}
$ \HH^{\al+n-2}(\{u_{\ift}=0\}\cap B_2)=0 $.
\end{lem}

\begin{rem}\label{aln2zero}
Indeed, for a fixed $ u\in (H_{\loc}^1\cap L_{\loc}^{-p})(B_2) $ being a weak solution of \eqref{MEMSeq}, Theorem 12 in \cite{DPP06} implies that $ \HH^{\al+n-2}(\{u=0\}\cap B_2)=0 $, under the assumption that $ f\in L_{\loc}^1(B_2) $. However, such a result cannot cover Lemma \ref{uiftzero} since here $ u_{\ift} $ is considered as a limit of $ u_i $ in the sense of \eqref{ConvPre} and has not been proved as a weak solution yet.
\end{rem}

\begin{proof}[Proof of Lemma \ref{uiftzero}]
We assume that $ \{u_{\ift}=0\}\cap B_2\neq\emptyset $, and for otherwise, there is nothing to prove. We first fix $ x\in\{u_{\ift}=0\}\cap B_2 $. By \eqref{uiAss}, it follows from Lemma \ref{InESLem1} that
$$
\int_{B_{\rho}(x)}u_i\geq C(\Lda,n,p)^{-1}\rho^{\al+n}
$$
for any $ i\in\Z_+ $ and $ \rho\in(0,1) $. This, together with \eqref{uiAss} and \eqref{ConvPre}, implies that
$$
\sup_{B_{\rho}(x)}u_{\ift}\geq C(\Lda,n,p)^{-1}\rho^{\al}
$$
for any $ \rho\in(0,1) $, and
\be
\|u_{\ift}\|_{C^{0,\al}(\ol{B}_4)}+\|f_{\ift}\|_{L^2(B_4)}\leq\Lda.\label{uiftLda}
\ee
As a result, there exist $ y\in B_{\rho}(x) $ and a sufficiently small $ \delta=\delta(\Lda,n,p)>0 $ such that $ B_{\delta\rho}(y)\subset B_{2\rho}(x) $, and
$$
\inf_{B_{\delta\rho}(y)}u_{\ift}\geq C(\Lda,n,p)^{-1}\rho^{\al}.
$$
Thus, we have $ B_{\delta\rho}(y)\subset\{u_{\ift}>0\} $. If $ 0<\rho<\f{2-|x|}{2} $, then
$$
B_{\delta\rho}(y)\subset B_{2\rho}(x)\subset\{u_{\ift}=0\}\cap B_2.
$$
Consequently,
$$
\f{\cL^n((\{u_{\ift}=0\}\cap B_2)\cap B_{2\rho}(x))}{\cL^n(B_{2\rho}(x))}\leq\f{\cL^n(B_{2\rho}(x)\backslash B_{\delta\rho}(y))}{\cL^n(B_{2\rho}(x))}\leq 1-\(\f{\delta}{2}\)^n.
$$
By the arbitrariness for the choice of $ x $, it follows from the Lebesgue differentiation theorem that
\be
\cL^n(\{u_{\ift}=0\}\cap B_2)=0.\label{Leb0}
\ee
Using \eqref{ConvPre}, we see that $ u_i^{-p}\to u^{-p} $ uniformly in any compact subset of $ \{u_{\ift}>0\}\cap B_2 $. In view of \eqref{Leb0}, Lemma \ref{InESLem1}, and Fatou's lemma, it follows that $ u_i^{-p}\to u_{\ift}^{-p} $ a.e. in $ B_1 $, and
\be
\int_{B_2}u_{\ift}^{-p}\leq\liminf_{i\to+\ift}\int_{B_2}u_i^{-p}\leq C(\Lda,n,p).\label{L1uift}
\ee
Fix $ s\in(0,2) $ and let $ 0<\va<2-s $. By Vitali's covering lemma, we can choose a countable covering of $ \{u_{\ift}=0\}\cap B_s $ by $ \{B_{r_j}(y_j)\} $ such that $ \sup_{j\in\Z_+}r_j\leq\va $, $ \{y_j\}\subset\{u_{\ift}=0\}\cap B_s $, and balls in the collection $ \{B_{\f{r_j}{5}}(y_j)\} $ are pairwise disjoint. For any $ j\in\Z_+ $, since $ u_{\ift}(y_j)=0 $, the estimate \eqref{uiftLda} yields that
$$
\sup_{B_{\f{r_j}{5}}(y_j)}u_{\ift}\leq Cr_j^{\al}[u_{\ift}]_{C^{0,\al}(B_{r_j}(y_j))}\leq C(\Lda,n,p)r_j^{\al},
$$
and then
$$
\int_{B_{\f{r_j}{5}}(y_j)}u_{\ift}^{-p}\geq C(\Lda,n,p)^{-1}r_j^{\al+n-2}.
$$
As a result, we arrive at
$$
\sum_{j=1}^{+\ift}r_j^{\al+n-2}\leq C\(\sum_{j=1}^N\int_{B_{\f{r_j}{5}}(y_j)}u_{\ift}^{-p}\)\leq C(\Lda,n,p)\(\int_{B_{\va}(\{u_{\ift}=0\})\cap B_2}u_{\ift}^{-p}\).
$$
By \eqref{Leb0} and \eqref{L1uift}, we can take $ \va\to 0^+ $ to obtain $
\HH^{\al+n-2}(\{u_{\ift}=0\}\cap B_s)=0 $. Letting $ s\to 2^{-} $, we complete the proof.
\end{proof}

\begin{proof}[Proof of Lemma \ref{ConvLem1}] First, assume that the convergence results presented in \eqref{ConPro1} hold. Since $ u_i $ is a weak solution of \eqref{MEMSeq} with respect to $ f_i $ for any $ i\in\Z_+ $, it follows from \eqref{ConvPre}, \eqref{ConPro1} that $ u_{\ift} $ is a weak solution of \eqref{MEMSeq} with respect to $ f_{\ift} $. Moreover, if \eqref{ConPro2} is true, $ u_{\ift} $ will inherit the stationary property of $ \{u_i\} $. Consequently, it remains to show \eqref{ConPro1} and \eqref{ConPro2}. We divide the proof of them into three steps.

\vspace{2mm}
\emph{Step 1. $ u_i^{-p}\to u_{\ift}^{-p} $ strongly in $ L_{\loc}^1(B_4) $.} For simplicity, we prove that $ u_i^{-p}\to u_{\ift}^{-p} $ strongly in $ L^1(B_1) $, and for general cases, it follows from standard covering arguments. Given \eqref{L1uift} and its proof, we only need to prove that
\be
\int_{B_1}u_{\ift}^{-p}\geq\limsup_{i\to+\ift}\int_{B_1}u_i^{-p}.\label{uiftgeq}
\ee
Applying Lemma \ref{uiftzero} for any $ \va>0 $, there is a countable covering of $ \{u_{\ift}=0\}\cap B_2 $, denoted by $ \{B_{r_j}(y_j)\} $ such that $ \{y_j\}\subset\{u_{\ift}=0\}\cap B_2 $, $ \sup_{i\in\Z_+}r_j\leq\va $, and
\be
\sum_{i=1}^{+\ift}r_j^{\al+n-2}<\va.\label{rjaln2}
\ee
Let $ U:=\cup_{i=1}^{+\ift}B_{r_j}(y_j) $. We see that $ U $ is open and $ \{u_{\ift}=0\}\cap B_2\subset U $. Using \eqref{ConvPre}, it follows that in $ (\{u_{\ift}>0\}\cap B_1)\backslash U $, $ u_{\ift} $ has a positive lower bound and then $ u_i^{-p}\to u_{\ift}^{-p} $ uniformly. Consequently,
\be
\lim_{i\to+\ift}\int_{(\{u_{\ift}>0\}\cap B_1)\backslash U}u_i^{-p}=\int_{(\{u_{\ift}>0\}\cap B_1)\backslash U}u_{\ift}^{-p}.\label{oneside1}
\ee
On the other hand, by using \eqref{uiAss} and Lemma \ref{InESLem1}, we deduce that
$$
\int_{U}u_i^{-p}\leq\sum_{j=1}^{+\ift}\int_{B_{r_j}(y_j)}u_i^{-p}\leq C\(\sum_{j=1}^{+\ift}r_j^{\al+n-2}\)\leq C(\Lda,n,p)\va.
$$
Here, for the last inequality, we have used \eqref{rjaln2}. This, together with \eqref{oneside1}, implies that
$$
\int_{B_1}u_{\ift}^{-p}\geq\limsup_{i\to+\ift}\int_{B_1}u_i^{-p}-C(\Lda,n,p)\va.
$$
Letting $ \va\to 0^+ $, we obtain \eqref{uiftgeq}.

\vspace{2mm}
\emph{Step 2. $ u_i^{1-p}\to u_{\ift}^{1-p} $ strongly in $ L_{\loc}^1(B_4) $.} Analogous to Step 1, we still only show the convergence in $ L^1(B_1) $. By Lemma \ref{uiftzero} and Remark \ref{aln2zero}, we obtain that $ \HH^{\al+n-2}(\{u_i=0\}\cap B_2)=0 $. Given the inequality
$$
|t^{1-p}-s^{1-p}|\leq C(p)|t-s|(s^{-p}+t^{-p}),\quad s,t>0,
$$
and Lemma \ref{InESLem1}, we have
$$
\int_{B_1}|u_i^{1-p}-u_{\ift}^{1-p}|\leq C(\Lda,n,p)\sup_{B_1}|u_i-u_{\ift}|.
$$
As a result, $ u_i^{1-p}\to u^{1-p} $ strongly in $ L^1(B_1) $, due to \eqref{ConvPre}.

\vspace{2mm}
\emph{Step 3. $ u_i\to u_{\ift} $ strongly in $ H_{\loc}^1(B_4) $.} Testing \eqref{MEMSeq} for $ u_i $ with $ u_i\vp^2 $, where $ \vp\in C_0^{\ift}(B_4) $, we have
$$
\int_{B_4}|\na u_i|^2\vp^2+2\int_{B_4}(\na u_i\cdot\na\vp)u\vp+\int_{B_4}u_i^{1-p}\vp^2+\int_{B_4}f_iu_i\vp^2=0.
$$
In the previous steps, we have obtained \eqref{ConPro1}. As a result, it follows from \eqref{ConvPre} that
\be
\lim_{i\to+\ift}\(\int_{B_4}|\na u_i|^2\vp^2\)+2\int_{B_4}(\na u_{\ift}\cdot\na\vp)u\vp+\int_{B_4}u_{\ift}^{1-p}\vp^2+\int_{B_4}f_{\ift}u_{\ift}\vp^2=0.\label{Testui}
\ee
Moreover, since \eqref{ConPro1} holds, by the analysis at the beginning of the proof for this lemma, $ u_{\ift} $ is a weak solution with respect to $ f_{\ift} $. Then by testing \eqref{MEMSeq} for $ u_{\ift} $ with $ u_{\ift}\vp^2 $, we have
$$
\int_{B_4}|\na u_{\ift}|^2\vp^2+2\int_{B_4}(\na u_{\ift}\cdot\na\vp)u\vp+\int_{B_4}u_{\ift}^{1-p}\vp^2+\int_{B_4}f_{\ift}u_{\ift}\vp^2=0.
$$
This, together with \eqref{Testui}, gives
\be
\lim_{i\to+\ift}\int_{B_4}|\na u_i|^2\vp^2=\int_{B_4}|\na u_{\ift}|^2\vp^2.\label{etanau}
\ee
Combining with the arbitrariness of $ \vp $, we conclude that $ u_i\to u_{\ift} $ strongly in $ H_{\loc}^1(B_4) $.
\end{proof}

\subsubsection{Proof of Proposition \ref{propConv2}} Just as the proof of Proposition \ref{propConv}, we still let $ r=4 $, $ R=8 $, and $ x=0 $. By \eqref{uiAssori} and Lemma \ref{InESLem1}, we can get similar results like \eqref{uiAss} and \eqref{uiAssnablsu}. Precisely, there exists $ \Lda>0 $ such that
\be
\sup_{i\in\Z_+}\(\|v_i\|_{C^{0,\al}(\ol{B}_4)}+\|f_i\|_{L^2(B_4)}\)\leq\Lda,\label{viassuse}
\ee
and for any $ s\in(0,4) $, there holds that 
$$
\sup_{i\in\Z_+}\|\na v_i\|_{L^2(B_s)}\leq C(\Lda,n,p,s).
$$
Incorporating with Lemma \ref{MorreyL2}, we can obtain $ v_{\ift}\in C^{0,\al}(\ol{B}_4)\cap H_{\loc}^1(B_4) $ and $ f_{\ift}\in M^{2\al+n-4,2}(B_4) $ such that up to a subsequence, there hold
\be
\begin{aligned}
v_i&\to v_{\ift}\text{ strongly in }L^{\ift}(B_4),\\
v_i&\wc v_{\ift}\text{ weakly in }H_{\loc}^1(B_4),\\
v_i&\wc f_{\ift}\text{ weakly in }L^2(B_4).
\end{aligned}\label{ConvPrevi}
\ee

In the same spirit of Lemma \ref{ConvLem1}, Proposition \ref{propConv2} is a direct consequence of the following lemma.

\begin{lem}\label{ConvLem2}
For $ \{v_i\},v_{\ift} $, and $ f_{\ift} $ in \eqref{ConvPrevi}, we have
\be
v_i\to v_{\ift}\text{ strongly in }H_{\loc}^1(B_4).\label{ConPro2vi}
\ee
Moreover, $ \Delta v_{\ift}=f_{\ift} $ in the sense of \eqref{weakuift}.
\end{lem}
\begin{proof}
Since for any $ i\in\Z_+ $, $ u_i $ is a weak solution of \eqref{MEMSeq} with respect to $ f_i $, we obtain that
\be
\int_{B_4}(\na v_i\cdot\na\vp+(u_i^{-p}+f_i)\vp)=0\label{testuxri1}
\ee
for any $ \vp\in C_0^{\ift}(B_4) $. By \eqref{uiAssori} and \eqref{uiL2bounded2}, we have
\be
\lim_{i\to+\ift}\(\inf_{B_4}u_i\)=+\ift.\label{uiiftcon}
\ee
Combining with \eqref{ConvPrevi}, it follows from letting $ i\to+\ift $ in \eqref{testuxri1} that
\be
\int_{B_4}(\na v_{\ift}\cdot\na\vp+f_{\ift}\vp)=0.\label{viftwanttest}
\ee
Now, it remains to show \eqref{ConPro2vi}. For $ \vp\in C_0^{\ift}(B_4) $, we test \eqref{testuxri1} by $ v_i\vp^2 $ and obtain
$$
\int_{B_4}|\na v_i|^2\vp^2+2\int_{B_4}(\na v_i\cdot\na\vp)v_i\vp+\int_{B_4}u_i^{-p}v_i\vp^2+\int_{B_4}f_iv_i\vp^2=0.
$$
Taking $ i\to+\ift $, we deduce from \eqref{viassuse}, \eqref{ConvPrevi}, and \eqref{uiiftcon} that
\be
\lim_{i\to+\ift}\(\int_{B_4}|\na v_i|^2\vp^2\)+2\int_{B_4}(\na v_{\ift}\cdot\na\vp)v_{\ift}\vp+\int_{B_4}f_{\ift}v_{\ift}\vp^2=0.\label{naviift}
\ee
Testing \eqref{viftwanttest} with $ v_{\ift}\vp^2 $ and using the formula \eqref{naviift}, we arrive at
$$
\lim_{i\to+\ift}\int_{B_4}|\na v_i|^2\vp^2=\int_{B_4}|\na v_{\ift}|^2\vp^2.
$$
By the arbitrariness for the choice of $ \vp $, \eqref{ConPro2vi} follows directly.
\end{proof}

\subsection{Blow-up limits of stationary solutions}

In this subsection, letting $ \ga>0 $ and $ \om\subset\R^n $ be a bounded domain, we assume that $ u\in(C_{\loc}^{0,\al}\cap H_{\loc}^1\cap L_{\loc}^{-p})(\om) $ is a stationary solution of \eqref{MEMSeq} with respect to $ f\in M_{\loc}^{2\al+n-4+\ga,2}(\om) $. 

\begin{defn}[Blow-ups]\label{defblowup}
Let $ x\in \om $ and $ 0<r<\dist(x,\pa\om) $. Assume that $ v:\om\to\R $ is a measurable function. We define
$$
T_{x,r}v(y):=r^{-\al}v(x+ry)\quad\text{and}\quad T_{x,r}^*v(y):=r^{2-\al}v(x+ry),
$$
where $ y\in\eta_{x,r}(\om) $.
\end{defn}

\begin{prop}\label{ScalingProp}
Let $ x\in \om $ and $ 0<r<\dist(x,\om) $. We have the following properties on scaling.
\begin{enumerate}[label=$(\theenumi)$]
\item $ T_{x,r}u $ is a stationary solution of \eqref{MEMSeq} with respect to $ T_{x,r}^*f $.
\item For any subdomain $ \om'\subset\subset \om $, 
\begin{align*}
[T_{x,r}u]_{C^{0,\al}(\eta_{x,r}(\ol{\om'}))}&=[u]_{C^{0,\al}(\ol{\om'})},\\
r^{-\ga}[T_{x,r}^*f]_{M^{2\al+n-4+\ga,2}(\eta_{x,r}(\om'))}&=[f]_{M^{2\al+n-4+\ga,2}(\om')}.
\end{align*}
\item For any $ B_R(y)\subset \eta_{x,r}(\om) $, there holds
\begin{align*}
\vt(T_{x,r}u;y,R)&=\vt(u;x+ry,rR),\\
\vt_{T_{x,r}^*f}(T_{x,r}u;y,R)&=\vt_f(u;x+ry,rR).
\end{align*}
\end{enumerate}
\end{prop}
\begin{proof}
The properties follow from simple calculations, so we omit details.
\end{proof}

\begin{prop}\label{Blowprop}
Assume that $ x\in\om $. We have results as follows.
\begin{enumerate}[label=$(\theenumi)$]
\item If $ u(x)=0 $, then there exist $ r_i\to 0^+ $ and $ 0\not\equiv u_{\ift}\in(C_{\loc}^{0,\al}\cap H_{\loc}^1\cap L_{\loc}^{-p})(\R^n) $ such that the following properties hold.
\begin{enumerate}[label=$(\mathrm{\alph*})$]
\item $ u_i:=T_{x,r_i}u\to u_{\ift} $ strongly in $ (H_{\loc}^1\cap L_{\loc}^{\ift})(\R^n) $.
\item $ f_i:=T_{x,r_i}^*f\to 0 $ strongly in $ L_{\loc}^2(\R^n) $.
\item $ u_i^{-p}\to u_{\ift}^{-p} $ and $ u_i^{1-p}\to u_{\ift}^{1-p} $ strongly in $ L_{\loc}^1(\R^n) $.
\item $ u_{\ift}(0)=0 $ and $ u_{\ift} $ is a stationary solution \eqref{MEMSeq} with respect to $ f\equiv 0 $.
\end{enumerate}
\item If $ u(x)>0 $, then exists $ r_i\to 0^+ $ such that $ u_i:=T_{x,r_i}(u-u(x))\to u_{\ift}\equiv 0 $ strongly in $ (H_{\loc}^1\cap L_{\loc}^{\ift})(\R^n) $.
\end{enumerate}
We call $ u_{\ift} $ a tangent function of $ u $ at the point $ x $.
\end{prop}

\begin{rem}
In the above proposition, the tangent function can be uniformly denoted by the limit of subsequence of the blow-up sequence $ \{T_{x,r}(u-u(x))\}_{r>0} $. It is noteworthy that for $ T_{x,r}u $, if $ u(x)>0 $, one will immediately see that it does not have limits for any subsequences.
\end{rem}

\begin{proof}[Proof of Proposition \ref{Blowprop}]
For $ x\in\om $, we choose $ r_0>0 $ such that $ B_{r_0}(x)\subset\subset\om $ and assume that
$$
[u]_{C^{0,\al}(\ol{B}_{r_0}(x))}+[f]_{M^{2\al+n-4+\ga,2}(B_{r_0}(x))}\leq\Lda.
$$
This, together with Proposition \ref{ScalingProp}, implies that for any $ R>0 $, if $ r>0 $ is sufficiently small such that $ rR<r_0 $, then
\be
[T_{x,r}u]_{C^{0,\al}(\ol{B}_R)}+
r^{-\ga}[T_{x,r}^*f]_{M^{2\al+n-4+\ga,2}(B_R)}\leq\Lda,\label{TxrTxstarf}
\ee
and $ T_{x,r}u $ is a stationary solution of \eqref{MEMSeq} with respect to $ T_{x,r}^*f $ in $ B_R $. Additionally, it follows from \eqref{TxrTxstarf} that
\be
T_{x,r}^*f\to 0\text{ strongly in }L_{\loc}^2(\R^n).\label{Txrfstrong0}
\ee

\emph{Case 1. $ u(x)=0 $.} As a result, $ T_{x,r}u(0)=0 $. Moreover, for any $ R>0 $, we have 
$$
\sup_{i\in\Z_+}\|T_{x,r}u\|_{L^2(B_R)}\leq C(\Lda,n,p,R). 
$$ 
Applying Proposition \ref{propConv} and diagonal arguments to $ T_{x,r}u $, we get the first point of this proposition.

\emph{Case 1. $ u(x)>0 $.} For $ 0<r<r_0 $, we let $ u_{x,r}:=T_{x,r}(u-u(x)) $. Since $ u(x)>0 $, we deduce from \eqref{TxrTxstarf} that for any $ R>0 $,
\be
\lim_{r\to 0^+}\(\inf_{B_R}T_{x,r}u\)=+\ift.\label{weakviforift}
\ee
By applying Proposition \ref{propConv2} and diagonal arguments, there is $ u_{\ift}\in (C_{\loc}^{0,\al}\cap H_{\loc}^1)(\R^n) $ such that
\be
\begin{aligned}
u_i:=u_{x,r_i}&\to u_{\ift}\text{ strongly in }(H_{\loc}^1\cap L_{\loc}^{\ift})(\R^n).
\end{aligned}\label{conuxri22}
\ee
By \eqref{Txrfstrong0}, \eqref{weakviforift}, and Weyl's lemma, $ u_{\ift} $ is a harmonic function. Given \eqref{TxrTxstarf}, \eqref{conuxri22}, and the property that for any $ i\in\Z_+ $, $ u_i(0)=0 $, we have $ u_{\ift}(0)=0 $. In addition, it follows from \eqref{TxrTxstarf} and \eqref{conuxri22} that $ [u_{\ift}]_{C^{0,\al}(\ol{B}_R)}\leq\Lda $ for any $ R>0 $. Then Corollary \ref{Liouvillecla} yields that $ u_{\ift}\equiv 0 $, which completes the proof.
\end{proof}

The following lemma gives the $ \al $-homogeneity of tangent functions. Since the tangent function is $ 0 $ for $ x\in\{u>0\} $, we assume $ u(x)=0 $ in this lemma.

\begin{lem}\label{tangent0}
Let $ x\in\om $ with $ u(x)=0 $. If $ u_{\ift} $ is a tangent function of $ u $ at the point $ x $, then $ u_{\ift} $ is $ \al $-homogeneous at $ 0 $. Precisely, for any $ y\in\R^n $ and $ \lda>0 $, $
u_{\ift}(\lda y)=\lda^{\al}u_{\ift}(y) $. In particular, for any $ r>0 $,
\be
\vt(u_{\ift};0,r)=\vt(u_{\ift};0)=\vt(u;x).\label{uifttheta}
\ee
\end{lem}

\begin{proof}
Recall that in Proposition \ref{Blowprop}, $ u_i=T_{x,r_i}u $ and $ f_i=T_{x,r_i}^*f $. Using Proposition \ref{propupture}, \ref{ScalingProp}, and Corollary \ref{coruse}, we have that for any $ R>0 $, 
\begin{align*}
0\leq\int_{B_{4R}}|y\cdot\na u_i-\al u_i|^2\ud y&\leq C(n,p)\[\vt_f(u;x,r_iR)-\vt_f\(u;x,\f{r_iR}{2}\)\]\to 0,
\end{align*}
as $ r_i\to 0^+ $. Since $ u_i\to u_{\ift} $ strongly in $ H_{\loc}^1(\R^n) $. Then for a.e. $ y\in\R^n $, $
y\cdot\na u_{\ift}-\al u_{\ift}=0 $. As a result, $ u_{\ift} $ is $ \al $-homogeneous, and the first equality of \eqref{uifttheta} holds. Now, it remains to show the second equality in \eqref{uifttheta}. Firstly, Proposition \ref{ScalingProp} implies that $
\vt_{f_i}(u_i;0,1)=\vt_f(u;x,r_i) $. Letting $ i\to+\ift $, the right-hand side of above is $ \vt(u;x) $, so we only need to verify
\be
\lim_{i\to+\ift}\vt_{f_i}(u_i;0,1)=\vt(u_{\ift};0,1)=\vt(u_{\ift};0).\label{verifyuse}
\ee
Assume that $ B_{r_0}(x)\subset\subset \om $, and for some $ \Lda>0 $,
$$
[u]_{C^{0,\al}(\ol{B}_{r_0}(x))}+[f]_{M^{2\al+n-4+\ga,2}(B_{r_0}(x))}\leq\Lda.
$$
Applying \eqref{thetaftheta}, Lemma \ref{InESLem1}, \ref{InESLem2}, and the property that $ u(x)=0 $, we arrive at
\begin{align*}
|\vt_{f_i}(u_i;0,1)-\vt(u_i;0,1)|\leq C(\ga,\Lda,n,p)r_i^{\ga}
\end{align*}
for any $ 0<r_i<\f{r_0}{100} $. Taking $ r_i\to 0^+ $, it follows from the convergence results of $ u_i $ given in Proposition \ref{Blowprop} that $ \vt(u_i;0,1)\to\vt(u_{\ift};0,1) $, and then \eqref{verifyuse} holds.
\end{proof}

We finally close this subsection with the following property on the upper semicontinuity of $ \vt(u;\cdot) $ defined by \eqref{densitylimit}.

\begin{lem}\label{semicon}
$ \vt(u;\cdot) $ is upper semicontinuous. Precisely, if $ x_j\to x_{\ift}\in\om $, then
\be
\vt(u;x_{\ift})\geq\limsup_{j\to+\ift}\vt(u;x_j).\label{semiconfor}
\ee
\end{lem}
\begin{proof}
If the right-hand side of \eqref{semiconfor} is $ -\ift $, then there is nothing to prove. Without loss of generality, we assume that $ \{x_j\}\subset\{u=0\} $ and then $ u(x_{\ift})=0 $. According to Proposition \ref{MonFor}, for sufficiently large $ j\in\Z_+ $, we have 
$$
\vt_f(u;x_j,r)\geq\vt(u;x_j)>-C(f,\om,u,x_{\ift}).
$$
It follows from taking $ j\to+\ift $ that for any $ r>0 $,
$$
\vt_f(u;x_{\ift},r)=\lim_{j\to+\ift}\vt_f(u;x_j,r)\geq\limsup_{j\to+\ift}\vt(u;x_j).
$$
Letting $ r\to 0^+ $, we obtain \eqref{semiconfor}.
\end{proof}

\subsection{Some applications of compactness results} As applications of compactness results presented in the previous subsections, we prove two properties on the solutions of \eqref{MEMSeq}. We first show the lower bound for some specific class of positive solutions and then establish the a priori estimates on $ C^{0,\al} $ norm for stationary solutions.

\subsubsection{Lower bounds of positive solutions} We consider positive and convex weak solutions and focus on the lower bound estimate on a convex bounded domain.

\begin{prop}\label{proplowerbound}
Suppose that the bounded domain $ \om\subset\R^n $ is convex. Let $ \ga>0 $ and assume that $ u\in C^{0,\al}(\ol{\om})\cap (H^1\cap L^{-p})(\om) $ is a weak solution of \eqref{MEMSeq} with respect to $ f\in M^{2\al+n-4+\ga,2}(\om) $, satisfying
\be
[u]_{C^{0,\al}(\ol{\om})}+[f]_{M^{2\al+n-4+\ga,2}(\om)}\leq\Lda.\label{gapC0al}
\ee
If $ u>0 $ and convex in $ \om $, then there exists $ C>0 $, depending only on $ \ga,\Lda,n,\om $, and $ p $ such that 
\be
\inf_{\om}u\geq C.\label{omugeqC}
\ee
\end{prop}

\begin{rem}\label{constantrem}
Up to a translation, in the estimate \eqref{omugeqC}, the constant does not depend on the position of $ \om $. Precisely, if $ \om'=\om+x $ for some $ x\in\R^n $, the constant $ C $ will not change, that is $ C(\om)=C(\om') $.
\end{rem}

\begin{rem}
Due to the example \eqref{ralsolu} for $ n=2 $, the assumption that $ u $ is convex is necessary since in the ball $ \om=B_{1-\va}((1,0)) $ with $ \va>0 $, the function $ u(x)=\al^{-\al}|x|^{\al} $ satisfies $ \inf_{\om}u=\al^{-\al}\va^{\al} $. Letting $ \va\to 0^+ $, it contradicts to Remark \ref{constantrem}.
\end{rem}

\begin{rem}
We also cannot remove the assumption \eqref{gapC0al} in Proposition \ref{proplowerbound}. In particular, if \eqref{gapC0al} no longer holds, then we can construct a class of positive and convex solutions $ \{u_{\va}\}_{\va>0} $ such that $ \inf_{\om}u_{\va}=\va $. We conclude such a result into the following lemma.
\end{rem}

\begin{lem}
For any $ \va>0 $, there exists $ u\in C^{\ift}(B_1) $ such that $ u $ is convex, positive solution of $ \Delta u=u^{-p} $ and $ \inf_{B_1}u=\va $. 
\end{lem}

Here, since we only need to construct one positive solution satisfying the desired properties, it is a straightforward to consider the case that $ u(x)=u(x_1) $, where $ x=(x_1,x_2,...,x_n) $. The result is a direct consequence of the lemma below.

\begin{lem}
For any $ \va>0 $, there exists a positive function function $ u=u(r)\in C^{\ift}(\R) $ such that $ \inf_{\R}u=\va $, and
\be
\f{\ud^2u}{\ud r^2}=u^{-p}.\label{ODEpositive}
\ee
\end{lem}

\begin{proof}
For any $ s\geq\va $, we define
\be
v(s):=\int_{\va}^s\f{\ud t}{\sqrt{\lda(\va^{1-p}-t^{1-p})}},\label{defvuse}
\ee
where $ \lda:=\f{2}{p-1} $. Thus, $ v:[\va,+\ift)\to\R $ is a continuous function, strictly increasing from $ 0 $ to $ +\ift $ in $ [\va,+\ift) $. Moreover, $ v $ is smooth in $ (\va,\ift) $, and $ v'(s)\to+\ift $ as $ s\to\va^+ $. For $ r\in\R $, let 
\be
u(r):=v^{-1}(|r|).\label{udeffromv}
\ee
We can complete the proof if we show that $ u $ is a smooth solution of \eqref{ODEpositive} with 
\be
\inf_{\R}u=u(0)=\va.\label{uinfva}
\ee
Indeed, $ u $ is a solution of \eqref{ODEpositive} satisfying 
$$
u(0)=\va\quad\text{and}\quad\left.\f{\ud}{\ud r}\right|_{r=0}u=0.
$$
According to \eqref{udeffromv}, $ u $ is strictly increasing in $ [0,+\ift) $ and strictly decreasing in $ (-\ift,0] $. Thus, \eqref{uinfva} follows directly. Using \eqref{defvuse} and \eqref{udeffromv}, we have
$$
\f{\ud u}{\ud r}=\sqrt{\lda(\va^{1-p}-u^{1-p})}
$$
for any $ r\in[0,+\ift) $, and then
$$
\f{\ud^2u}{\ud r^2}=-\f{\lda(1-p)u^{-p}}{2\sqrt{\lda(\va^{1-p}-u^{1-p})}}\f{\ud u}{\ud r}=u^{-p}.
$$
As a result, $ u\in C^2(\R) $ and \eqref{ODEpositive} holds when $ r\in[0,+\ift) $. On the other hand, by almost the same calculations, it also holds for $ r\in(-\ift,0] $. Using bootstrap arguments, we deduce that $ u\in C^{\ift}(\R) $, satisfying \eqref{uinfva}, as desired.
\end{proof}

\begin{proof}[Proof of Proposition \ref{proplowerbound}]
Assume that the result is not true. Then we can choose a sequence of positive, convex, weak solutions $ \{u_i\}\subset C^{0,\al}(\ol{\om})\cap(H^1\cap L^{-p})(\om) $ with respect to $ \{f_i\}\subset M^{2\al+n-4+\ga,2}(\om) $, satisfying
\begin{gather}
\sup_{i\in\Z_+}\([u_i]_{C^{0,\al}(\ol{\om})}+[f_i]_{M^{2\al+n-4+\ga,2}(\om)}\)\leq\Lda,\label{gapC0ali}\\
0<u(x_i)=\va_i^{\al}<2\inf_{\om}u_i<i^{-1},\label{xiuiass}
\end{gather}
where $ x_i\in\om $ for any $ i\in\Z_+ $. Letting $ v_i:=T_{x_i,\va_i}u_i $ and $ g_i:=T_{x_i,\va_i}^*f_i $, it follows from \eqref{gapC0ali}, \eqref{xiuiass}, and Proposition \ref{ScalingProp} that
\be
\begin{gathered}
\sup_{i\in\Z_+}\([v_i]_{C^{0,\al}(\eta_{x_i,\va_i}(\ol{\om}))}+\va_i^{-\ga}[g_i]_{M^{2\al+n-4+\ga,2}(\eta_{x_i,\va_i}(\om))}\)\leq\Lda,\\
v_i(0)=\va_i^{-\al}u(x_i)=1<2\inf_{\eta_{x_i,\va_i}(\om)}v_i.
\end{gathered}\label{vivaassu}
\ee
Since $ \om $ is convex, up to rotations, we can extract a subsequence of $ \om_i $ without changing the notation such that $
\om_i:=\eta_{x_i,\va_i}(\om)\to\R^n $ or $ \R_+^n:=\R^n\cap\{x_n>0\} $. We only consider the case $ \om_i\to\R_+^n $. The other case is from similar arguments. In view of \eqref{vivaassu} and Proposition \ref{propConv}, there exists $ v_{\ift}\in(C_{\loc}^{0,\al}\cap H_{\loc}^1\cap L_{\loc}^{-p})(\R_+^n) $ such that up to a subsequence,
\begin{align*}
v_i&\to v_{\ift}\text{ strongly in }(H_{\loc}^1\cap L_{\loc}^{\ift})(\R_+^n),\\
g_i&\to 0\text{ strongly in }L_{\loc}^2(\R_+^n).
\end{align*}
Moreover, $ v_{\ift} $ is a convex weak solution of $ \Delta v_{\ift}=v_{\ift}^{-p} $, satisfying
\be
\inf_{\R_+^n}v_{\ift}\geq\f{1}{2}\quad
\text{and}\quad\sup_{K\subset\subset\R_+^n}[v_{\ift}]_{C^{0,\al}(\ol{K})}\leq\Lda.\label{infRplus}
\ee
By \eqref{infRplus}, the standard regularity theory for elliptic equations yields that $ v_{\ift}\in C^{\ift}(\R_+^n) $. According to the convexity of $ v_{\ift} $, there is $ x\in\R_+^n $ such that $ D^2v_{\ift}(x) $ is semi-positive definite and $ D^2v_{\ift}(x)\neq 0 $. Consequently, $ v_{\ift}(y)-v_{\ift}(x)\geq D^2v_{\ift}(x)(y-x) $. Given \eqref{infRplus}, we arrive at
\be
|D^2v_{\ift}(x)(y-x)|\leq\Lda|x-y|^{\al}.\label{Dviftx}
\ee
Since $ D^2v_{\ift}(x)\neq 0 $, without loss of generality, we let $ D^2v_{\ift}(x)=\diag\{\lda_1,...\lda_k,0,...,0\} $, where $ \lda_1\geq\lda_2\geq...\geq\lda_k>0 $ with $ k\in\Z\cap[1,n] $. Letting $ y=x+(x_1,0,...,0) $ in \eqref{Dviftx}, it implies that $
\lda_1|x_1|\leq\Lda|x_1|^{\al} $, which is impossible if $ |x_1|>0 $ is sufficiently large. 
\end{proof}

\subsubsection{Interior H\"{o}lder estimates} Now, let us study the a priori $ C^{0,\al} $ estimate for stationary solutions of \eqref{MEMSeq}. Such results have been obtained in \cite{DWW16} for smooth solutions of $ \Delta u=u^{-p} $. For the model in our paper, we concentrate on a more general case by a compactness argument. Though we do not directly apply Proposition \ref{propConv}, \ref{propConv2}, and \ref{Blowprop}, the spirit and arguments are analogous to those.

\begin{prop}\label{propHolder}
Let $ q>0 $ be such that $ \f{1}{2}+\f{1}{2p}<\f{q}{n} $. Assume that $ u\in(C_{\loc}^{0,\al}\cap H_{\loc}^1\cap L_{\loc}^{-p})(B_5) $ is a stationary solution of \eqref{MEMSeq} with respect to $ f\in L_{\loc}^q(B_5) $, satisfying
\be
\|u\|_{L^1(B_4)}+\|f\|_{L^q(B_4)}\leq\Lda.\label{L1bound}
\ee
Then there exists $ C>0 $ depending only on $ \Lda,n,p $, and $ q $ such that
\be
\|u\|_{C^{0,\al}(B_1)}\leq C.\label{Holderbound}
\ee
\end{prop}

\begin{proof}[Proof of Proposition \ref{propHolder}]
Let $ \vp\in C_0^{\ift}(\R^n) $ satisfy the following properties
\begin{itemize}
\item $ \vp\equiv 1 $ in $ B_1 $ in $ \R^n $, and $ B_2=\{\vp>0\} $.
\item $ 0\leq\vp\leq 1 $ and $ |\na\vp|\leq C(n) $ in $ \R^n $.
\end{itemize}
We will show that
\be
[u\vp]_{C^{0,\al}(\ol{B}_2)}\leq C(\Lda,n,p,q).\label{uetabound}
\ee
This, together with \eqref{L1bound}, implies \eqref{Holderbound}. Assume that the estimate \eqref{uetabound} is not true, there exists a sequence of stationary solutions of \eqref{MEMSeq}, denoted by $ \{u_i\}\subset(C_{\loc}^{0,\al}\cap H_{\loc}^1\cap L^{-p})(B_4) $ with respect to $ f_i\in L_{\loc}^q(B_5) $ such that for any $ i\in\Z_+ $, the following properties hold.
\begin{itemize}
\item $ u_i $ and $ f_i $ satisfy 
\be
\|u_i\|_{L^1(B_4)}+\|f_i\|_{L^q(B_4)}\leq\Lda.\label{uificon}
\ee
\item There exists $ x_i,y_i\in\ol{B}_2 $ such that
\be
L_i:=\f{|(u_i\vp)(x_i)-(u_i\vp)(y_i)|}{|x_i-y_i|^{\al}}\geq\f{1}{2}[u_i\vp]_{C^{0,\al}(\ol{B}_2)},\label{Lidef}
\ee
and 
\be
\lim_{i\to+\ift}L_i=+\ift.\label{Lilimit}
\ee
\end{itemize}
Using \eqref{uificon}, Lemma \ref{LemHL}, and the fact that $ u_i\geq 0 $ for any $ i\in\Z_+ $, we have
\be
\sup_{i\in\Z_+}\|u_i\|_{L^{\ift}(B_2)}\leq C(\Lda,n,p,q).\label{Liftboundui}
\ee
Due to \eqref{Lilimit}, $ |x_i-y_i|\to 0 $. Let $ r_i:=|x_i-y_i| $ and $ z_i:=\f{y_i-x_i}{r_i} $. We see that $ |z_i|=1 $. Up to a subsequence, there holds
\be
z_i\to z_{\ift}\in\Ss^{n-1}.\label{Ziconv}
\ee
Recalling the notation in Definition \ref{defblowup}, for any $ i\in\Z_+ $, we set
\begin{align*}
v_i&:=L_i^{-1}T_{x_i,r_i}(u_i\vp),\\
w_i&:=L_i^{-1}\vp(x_i)T_{x_i,r_i}u_i=\va_iT_{x_i,r_i}u_i,\\
g_i&:=L_i^{-1}\vp(x_i)T_{x_i,r_i}^*f_i=\va_iT_{x_i,r_i}^*f_i,
\end{align*}
defined in $ \om_i:=\eta_{x_i,r_i}(\om) $, where $ \va_i:=L_i^{-1}\vp(x_i) $. We note that \eqref{Lilimit} implies 
\be
\lim_{i\to+\ift}\va_i=0.\label{vailimit}
\ee
The estimate \eqref{Liftboundui} and the choice of $ \vp $ yield that for any $ x\in\om_i $,
\be
|v_i(x)-w_i(x)|\leq L_i^{-1}r_i^{-\al}|(\vp(x_i+r_ix)-\vp(x_i))(u_i(x_i+r_ix))|\leq CL_i^{-1}r_i^{1-\al}|x|.\label{Lipdiff}
\ee
Because of \eqref{Lidef}, we obtain
\be
1=|v_i(0)-v_i(z_i)|\geq\f{1}{2}[v_i]_{C^{0,\al}(\ol{\om_i})}.\label{normalzi}
\ee
For any $ x\in\om_i $, there exists $ y_x\in\pa\om_i $ such that $ \dist(x,\pa\om_i)=|x-y_x| $. Since $ \vp\equiv 0 $ in $ \R^n\backslash B_2 $, it follows that
\be
\begin{aligned}
0&\leq v_i(x)=|v_i(x)-v_i(y_x)|=L_i^{-1}T_{x_i,r_i}(u_i\vp)(x)\\
&\leq CL_i^{-1}r_i^{1-\al}|\vp(x_i+r_ix)-\vp(x_i+r_iy_x)|\\
&\leq CL_i^{-1}r_i^{1-\al}|x-y_x|=C(\Lda,n,p,q)L_i^{-1}r_i^{1-\al}\dist(x,\pa\om_i).
\end{aligned}\label{vidist}
\ee
For the second inequality above, we have used \eqref{Liftboundui}. Since $ u_i $ is a stationary solution of \eqref{MEMSeq} with respect to $ f_i $, by Proposition \ref{ScalingProp}, it can be easily checked that
\be
\int_{\om_i}(\na w_i\cdot\na\vp+(\va_i^{p+1}w_i^{-p}+g_i)\vp)=0\label{weakvifor}
\ee
for any $ \vp\in C_0^{\ift}(\om_i) $, and
\be
\int_{\om}\left[\(\f{|\na w_i|^2}{2}-\f{\va_i^{p+1}w_i^{1-p}}{p-1}\)\op{div}Y-DY(\na w_i,\na w_i)-g_i(Y\cdot\na w_i)\right]=0\label{stavifor}
\ee
for any $ Y\in C_0^{\ift}(\om,\R^n) $. Let $ A_i:=v_i(0) $. Up to a subsequence, we assume that $ \lim_{i\to+\ift}A_i $ exists (possibly be $ +\ift $). We now divide the proof into two cases, depending on the value of this limit.

\vspace{2mm}
\emph{Case 1. $ A_i\to+\ift $.} The estimate \eqref{vidist} implies that 
$$
\dist(0,\pa\om_i)\geq C(\Lda,n,p,q)L_ir_i^{\al-1}A_i\to+\ift.
$$
Consequently, $ \om_i\to\R^n $. According to \eqref{normalzi}, up to a subsequence, there exists $ v_{\ift}\in C_{\loc}^{0,\al}(\R^n) $, such that
\be
v_i-A_i\to v_{\ift}\text{ strongly in }L_{\loc}^{\ift}(\R^n).\label{wiAilimit}
\ee
Using \eqref{Lipdiff}, we also have
\be
w_i-A_i\to w_{\ift}\equiv v_{\ift}\text{ strongly in }L_{\loc}^{\ift}(\R^n).\label{wilimitwift}
\ee
In particular, for any $ R>0 $, 
$$
\lim_{i\to+\ift}\|(w_i-A_i)-w_{\ift}\|_{L^{\ift}(B_R)}=0.
$$
Moreover, since $ L_i\to+\ift $ and $ r_i\to 0^+ $, \eqref{Lipdiff} and \eqref{normalzi} also give that if $ i\in\Z_+ $ is sufficiently large, then for any $ R>0 $,
\be
\inf_{B_{2R}}w_i\geq\inf_{B_{2R}}v_i-CL_i^{-1}r_i^{1-\al}R\geq A_i-2R^{\al}-CL_i^{-1}r_i^{1-\al}R\geq\f{A_i}{2}.\label{Ailower}
\ee
Using \eqref{uificon}, H\"{o}lder's inequality, and the assumption that $ \f{1}{2}+\f{1}{2p}<\f{q}{n} $, it follows that
\be
\int_{B_{2R}}|g_i|^2\leq \va_ir_i^{4-2\al-n}(r_iR)^{n(1-\f{2}{q})}\(\int_{B_{2r_iR}(x)}|f_i|^q\)^{\f{2}{q}}\leq C(\Lda,n,p,q,R)\va_i.\label{giestim}
\ee
This, together with \eqref{vailimit} and \eqref{Ailower}, implies that
\be
\lim_{i\to+\ift}\|\va_i^{p+1}w_i^{-p}+g_i\|_{L^2(B_{2R})}=0.\label{vap1gi}
\ee
In view of \eqref{weakvifor} and the fact that $ w_i(0)=A_i $, we can apply Caccioppoli's inequality to get that 
$$
\sup_{i\in\Z_+}\|w_i-A_i\|_{H^1(B_{R})}<+\ift. 
$$
By the arbitrariness of $ R>0 $, we further assume that 
\be
w_i-A_i\wc w_{\ift}\text{ weakly in } H_{\loc}^1(\R^n).\label{wiH1limit}
\ee
Combining with \eqref{weakvifor} and \eqref{vap1gi}, it follows that $ \Delta w_{\ift}=0 $ in the weak sense. By Weyl's lemma, $ w_{\ift} $ is a harmonic function. Taking \eqref{Ziconv}, \eqref{normalzi}, \eqref{wiAilimit}, and \eqref{wilimitwift} into account, we arrive at
\be
1=|w_{\ift}(0)-w_{\ift}(z_{\ift})|\geq\f{1}{2}[w_{\ift}]_{C^{0,\al}(B_R)}\label{conuse}
\ee
for any $ R>0 $. As a result, Corollary \ref{Liouvillecla} yields that $ w_{\ift} $ is a constant function, which is a contradiction to \eqref{conuse}.

\vspace{2mm}
\emph{Case 2. $ A_i\to A_{\ift}\in[0,+\ift) $.} Using \eqref{normalzi}, we have
\be
1\leq v_{\va}(0)+v_{\va}(z_i).\label{vass}
\ee
It follows from \eqref{vidist} and the fact $ |z_i|=1 $ that
\begin{align*}
0<C(\Lda,n,p,q)L_ir_i^{\al-1}&\leq\dist(0,\pa\om_i)+\dist(z_i,\pa\om_i)\leq 2\dist(0,\pa\om_i)+1.
\end{align*}
As a result, since $ L_i\to+\ift $ and $ r_i\to 0^+ $, we have $ \dist(0,\pa\om_i)\to+\ift $, and then $ \om_i\to\R^n $. By similar arguments as in the derivations of \eqref{wiAilimit} and \eqref{wilimitwift},  we can use \eqref{Lipdiff}, \eqref{normalzi}, and \eqref{vass} to obtain $ w_{\ift}\in C_{\loc}^{0,\al}(\R^n) $ such that
\be
w_i,v_i\to w_{\ift}\text{ strongly in }L_{\loc}^{\ift}(\R^n),\label{wvawift}
\ee
and then $ 1\leq w_{\ift}(0)+w_{\ift}(z_{\ift}) $.
Consequently, $ \{w_{\ift}>0\}\neq\emptyset $. Let $ K\subset\{w_{\ift}>0\} $ be a compact set. Define $
\delta:=\f{1}{2}\inf_Kw_{\ift} $. We now apply the convergence results in \eqref{wvawift} to obtain that for any sufficiently large $ i\in\Z_+ $, $
\inf_{K}w_i\geq\delta $. Incorporated with \eqref{giestim}, this yields
$$
\lim_{i\to+\ift}\|\va_i^{p+1}w_i^{-p}+g_i\|_{L^2(K)}=0.
$$
By almost the same arguments in Case 1, we see that $ \Delta w_{\ift}=0 $ in $ \{w_{\ift}>0\} $. If $ \{w_{\ift}=0\}=\emptyset $, we still have \eqref{conuse}, which is a contradiction, due to Corollary \ref{Liouvillecla}.

Now, let us assume that $ \{w_{\ift}=0\}\neq\emptyset $. Up to a translation, we let $ 0\in\{w_{\ift}=0\} $. It can be easily checked that \eqref{conuse} is also true for this case. As a result, for any $ R>0 $,
\be
\sup_{i\in\Z_+}\|w_i\|_{L^{\ift}(B_R)}<+\ift.\label{wiunifo}
\ee
We claim that
\begin{align}
w_i&\to w_{\ift}\text{ strongly in }H_{\loc}^1(\R^n),\label{H1strongift}\\
\va_i^{p+1}w_{\va}^{1-p}&\to 0\text{ strongly in }L_{\loc}^1(\R^n)\label{L1strongift}.
\end{align}
Let $ \psi\in C_0^{\ift}(\R^n) $. For $ i\in\Z_+ $ sufficiently large, the application of \eqref{weakvifor} with $ \vp=w_i\psi^2 $ implies that
\be
\int_{\R^n}(|\na w_i|^2\psi^2+\va_i^{p+1}w_i^{1-p}\psi^2+2(\na w_i\cdot\na\psi)w_i\psi+g_iw_i\psi^2)=0.\label{witest}
\ee
By \eqref{abdelta}, we have
$$
\int_{\R^n}(|\na w_i|^2\psi^2+\va_i^{p+1}w_i^{1-p}\psi^2)\leq C(n)\int_{\R^n}(w_i^2|\na\psi|^2+|g_i|w_i\psi^2).
$$
This, together with \eqref{giestim} and \eqref{wiunifo}, gives that for any $ R>0 $, $ \|w_i\|_{H^1(B_R)} $ is uniformly bounded. Consequently, up to a subsequence, we obtain
\be
w_i\to w_{\ift}\text{ weakly in }H_{\loc}^1(\R^n).\label{wiftwilimiH1}
\ee
Taking $ i\to+\ift $ in \eqref{witest}, it follows that
\be
\begin{aligned}
&\lim_{i\to+\ift}\(\int_{\R^n}|\na w_i|^2\psi^2\)-\int_{\R^n}|\na w_{\ift}|^2+\int_{\R^n}\va_i^{p+1}w_i^{1-p}\psi^2\\
&\quad\quad\quad\quad=-\int_{\R^n}(|\na w_{\ift}|^2+2(\na w_{\ift}\cdot\na\psi)w_{\ift}\psi)
\end{aligned}\label{nalimitwi}
\ee
Since $ w_{\ift} $ is smooth in $ \{w_{\ift}>0\} $, it follows from Sard's theorem that $ \{w_{\ift}=t\} $ is a smooth hypersurface for a.e. $ t>0 $. According to integration by parts and the fact that $ \Delta w_{\ift}=0 $ in $ \{w_{\ift}>0\} $, we have
\begin{align*}
&\int_{\{w_{\ift}>t\}}(|\na w_{\ift}|^2+2(\na w_{\ift}\cdot\na\psi)w_{\ift}\psi)=\int_{\{w_{\ift}=t\}}\pa_{\nu_t}w_{\ift}w_{\ift}\psi^2\\
&\quad\quad\quad\quad=t\(\int_{\{w_{\ift}=t\}}\pa_{\nu_t}w_{\ift}\psi^2\)=t\(\int_{\{w_{\ift}>t\}}\na w_{\ift}\cdot\na(\psi^2)\),
\end{align*}
where $
\pa_{\nu_t}w_{\ift}=\nu_t\cdot\na w_{\ift} $, and $ \nu_t $ is the outward unit normal vector of $ \{w_{\ift}=t\} $. Letting $ t\to 0^+ $, we see that
$$
\int_{\R^n}(|\na w_{\ift}|^2\psi^2+2(\na w_{\ift}\cdot\na\psi)w_{\ift}\psi)=0.
$$
This, together with \eqref{nalimitwi}, shows \eqref{H1strongift} and \eqref{L1strongift}. Combining with \eqref{stavifor}, $ w_{\ift} $ satisfies
$$
\int_{\R^n}(|\na w_{\ift}|^2-2DY(\na w_{\ift},\na w_{\ift}))=0
$$
for any $ Y\in C_0^{\ift}(\R^n,\R^n) $. This leads to the property that $ w_{\ift} $ is a stationary solution of $ w_{\ift}\Delta w_{\ift}=0 $ in $ \R^n $ and satisfies \eqref{conuse}. Consequently, Lemma \ref{ulapulem} implies that $ w_{\ift} $ is a constant function, which is a contradiction to \eqref{conuse}.
\end{proof}

\section{Classical stratification theory}\label{StratificationSection}

Using Proposition \ref{Blowprop} and Lemma \ref{tangent0}, for stationary solutions of \eqref{MEMSeq}, at any point in the rupture set, the tangent function exists and is $ \al $-homogeneous. Based on this property, we can establish the stratification results. 

\subsection{Symmetry property of functions} Generally speaking, stratification is a classification of points in the definite domain, depending on the symmetry of tangent functions. Such an idea leads us to define the characterization of the concept on $ k $-symmetric functions related to the model in this paper.

\begin{defn}[$ k $-symmetric functions]\label{ksymmetryf}
Let $ k\in\Z\cap[0,n] $ and $ x\in\R^n $. A function $ h\in C_{\loc}^{0,\al}(\R^n) $) is called $ k $-symmetric at $ x $ with respect to $ V\in\bG(n,k) $ if it satisfies the following properties.
\begin{enumerate}[label=$(\theenumi)$]
\item $ h $ is $ \al $-homogeneous at $ x $, namely, for any $ \lda>0 $ and $ y\in\R^n $,
$ h(x+\lda y)=\lda^{\al}h(x+y) $.
\item $ h $ is invariant with respect to $ V $, namely, for any $ v\in V $ and $ y\in\R^n $, $ h(y+v)=h(y) $. 
\end{enumerate}
For simplicity, if $ x=0 $, we say that $ h $ is $ k $-symmetric.
\end{defn}

\begin{rem}\label{remh0}
By the definition, if $ h $ is $ 0 $-symmetric, then $ h(0)=0 $. If $ h $ is $ n $-symmetric, then $ h\equiv 0 $ in $ \R^n $. 
\end{rem}

Concerning $ k $-symmetric functions, we first present some lemmas on the extension and convergence results.

\begin{lem}\label{extend}
Let $ k\in\Z\cap[0,n] $, $ r>0 $, and $ x\in\R^n $. Assume that $ h\in C^{0,\al}(\ol{B}_r(x)) $ and $ V\in\bG(n,k) $ satisfy the following properties.
\begin{enumerate}[label=$(\theenumi)$]
\item $ h $ is $ \al $-homogeneous at $ x $ in $ \ol{B}_r(x) $, namely, for any $ 0<\lda\leq r $ and $ y\in\ol{B}_1 $, $
h(x+\lda y)=\lda^{\al}h(x+y) $.
\item $ h $ is invariant with respect to $ V $, namely, for any $ v\in V $ and $ y\in\ol{B}_r(x) $ with $ y+v\in\ol{B}_r(x) $, $
h(y+v)=h(y) $. 
\end{enumerate}
Here, we call $ h $ a $ k $-symmetric in $ B_r(x) $. Then there exists an extension of $ h $ to $ \R^n $, denoted by $ \wt{h}\in C_{\loc}^{0,\al}(\R^n) $ such that $ \wt{h}\equiv h $ in $ \ol{B}_r(x) $, $ [\wt{h}]_{C^{0,\al}(\R^n)}=[h]_{C^{0,\al}(\ol{B}_r(x))} $, and $ \wt{h} $ is $ k $-symmetric at $ x $ with respect to $ V $. Moreover, this extension is unique. In particular, if $ h_1 $ and $ h_2 $ are two extensions satisfying the above properties, then $ h_1\equiv h_2 $.
\end{lem}
\begin{proof}
We first note that the uniqueness follows directly from the $ \al $-homogeneity. Define
$$
\wt{h}(y)=\left\{\begin{aligned}
&h(y)&\text{ if }&y\in\ol{B}_r(x),\\
&\(\f{|y-x|}{r}\)^{\al}h\(x+\f{r(y-x)}{|y-x|}\)&\text{ if }&y\in\R^n\backslash\ol{B}_r(x).
\end{aligned}\right.
$$
Then $ \wt{h} $ satisfies desired properties.
\end{proof}

\begin{lem}\label{ConSym}
Let $ k\in\Z\cap[0,n] $. Assume that $ \{h_i\}\subset C_{\loc}^{0,\al}(\R^n) $ is a sequence of $ k $-symmetric functions at $ \{x_i\}\subset\R^n $ with respect to $ \{V_i\}\subset\bG(n,k) $. If $ h_i\to h_{\ift}\in C_{\loc}^{0,\al}(\R^n) $ strongly in $ L_{\loc}^{\ift}(\R^n) $, $ V_i\to V_{\ift}\in\bG(n,k) $, and $ x_i\to x_{\ift} $, then $ h_{\ift} $ is $ k $-symmetric at $ x_{\ift} $ with respect to $ V_{\ift} $.
\end{lem}
\begin{proof}
The $ \al $-homogeneity and invariance with respect to $ k $-dimensional subspace is defined pointwise, so they are still valid under the strong convergence in $ L_{\loc}^{\ift}(\R^n) $. 
\end{proof}

\subsection{Stratification and Hausdorff dimensions of strata} Let $ \ga>0 $ and $ \om\subset\R^n $ be a bounded domain. Assume that $ u\in(C_{\loc}^{0,\al}\cap H_{\loc}^1\cap L_{\loc}^{-p})(\om) $ is a stationary solution of \eqref{MEMSeq} with respect to $ f\in M_{\loc}^{2\al+n-4+\ga,2}(\om) $.

\begin{defn}
For any $ k\in\Z\cap[0,n-1] $, define the $ k $-stratum of $ u $ by
$$
S^k(u):=\{x\in\om:\text{no tangent function }v\text{ of }u\text{ at }x\text{ is }(k+1)\text{-symmetric}\}.
$$
As a result,
$$
S^0(u)\subset S^1(u)\subset S^2(u)\subset...\subset S^{n-1}(u)\subset\om.
$$
\end{defn}

\begin{rem}\label{remSn1u0}
Using Proposition \ref{Blowprop}, we see that the tangent function at points in $ \{u>0\} $ is $ 0 $. Thus, we have $ S^{n-1}(u)\subset\{u=0\} $.
\end{rem}

The following proposition establishes the estimate for the Hausdorff dimension of $ k $-stratum defined above.

\begin{prop}\label{Hausdorffdim}
Assume that $ \{S^k(u)\}_{k=0}^{n-1} $ are given above. We have
$$ 
S^0(u)\subset S^1(u)\subset S^2(u)\subset...\subset S^{n-2}(u)=S^{n-1}(u)=\{u=0\},
$$
and for any $ k\in\Z\cap[0,n-2] $,
\be
\dim_{\HH}(S^k(u))\leq k\label{Whitees}
\ee
If $ n=2 $, then $ S^0(u) $ is discrete. 
\end{prop}

\begin{rem}
This proposition is a generalization for Theorem \ref{thmDWW} since $ f $ is not identical to $ 0 $ here. Additionally, we consider the stratification for the rupture set.
\end{rem}

To show Proposition \ref{Hausdorffdim}, we will adopt the standard arguments developed in \cite{Whi97}. One can also see Chapter 10 of \cite{GM05} and Chapter 2 of \cite{LW08} for references on similar methods applied in the study of harmonic maps. Before proving this result, we first recall some fundamental properties and concepts.

\begin{lem}[\cite{DWW16}, Lemma 5.8 and 5.10]\label{DWW58}
Assume that $ h\in(C_{\loc}^{0,\al}\cap H_{\loc}^1\cap L_{\loc}^{-p})(\R^n) $ is a stationary solution of \eqref{MEMSeq} with respect to $ f\equiv 0 $. Suppose that $ h $ is $ 0 $-symmetric. Then for any $ x\neq 0 $, $ \vt(h;x)\leq\vt(h;0) $ and the set
$$
\Sg(h):=\{x\in\R^n:\vt(h;x)=\vt(h;0)\}
$$
is a subspace of $ \R^n $. Moreover, $ h $ is invariant with respect to $ \Sg(h) $. If $ n=2 $, then $ \{h=0\}=\{(0,0)\} $.
\end{lem}

\begin{lem}[\cite{DWW16}, Lemma 5.9]\label{DWW59}
Let $ k\in\Z\cap[2,n-1] $. Assume that $ u=u(x_1,x_2,...,x_k)\in(H_{\loc}^1\cap L_{\loc}^{-p})(\R^k) $ is a weak solution of \eqref{MEMSeq} in $ \R^k $ with respect to $ f\equiv 0 $. Take $ \wt{u} $ to be the trivial extension of $ u $ to $ \R^n $, given by
$$
\wt{u}(x_1,x_2,...,x_n):=u(x_1,x_2,...,x_k).
$$
Then $ \wt{u} $ is stationary if and only if $ u $ is stationary.
\end{lem}

\begin{defn}\label{deltaj}
Let $ \delta>0 $ and $ k\in\Z\cap[0,n-1] $. We say that the subset $ S\subset\R^n $ satisfies the $ (\delta,k) $-approximation property if there is $ \rho_0>0 $ such that for any $ y\in S $ and $ \rho\in(0,\rho_0] $, there exists $ V\in\bG(n,k) $, satisfying $
\eta_{y,\rho}(S)\cap B_1\subset B_{\delta}(V) $.
\end{defn}

\begin{lem}[\cite{GM05}, Lemma 10.38]\label{lemhau}
Let $ k\in\Z\cap[0,n-1] $. There exists $ \beta:\R_+\to\R_+ $ satisfying the following properties.
\begin{enumerate}[label=$(\theenumi)$]
\item $ \lim_{t\to 0^+}\beta(t)=0 $.
\item If $ S\subset\R^n $ satisfies the $ (\delta,k) $-approximation property with $ \delta>0 $, then $
\HH^{k+\beta(\delta)}(S)=0 $.
\end{enumerate}
\end{lem}

The proof of Proposition \ref{Hausdorffdim} is divided into several secondary lemmas. We will prove them in order.

\begin{lem}\label{Sn2}
$ \{u=0\}=S^{n-2}(u) $.
\end{lem}
\begin{proof}
Let $ x\in\{u=0\}\backslash S^{n-2}(u) $. By Proposition \ref{Blowprop}, there is a tangent function $ h $ of $ u $ at $ x $ such that $ h $ is invariant with respect to $ V\in\bG(n,n-1) $. Up to a rotation, we assume that $ V=\R^{n-1}\times(0) $. As a result, one can regard the function $ h $ as a function with the last two variables. Precisely, we write $
h(x_{n-1},x_n)=h(x_1,x_2,...,x_n) $. Moreover, $ h(x_{n-1},0)=0 $ for any $ x_{n-1}\in\R $ since $ h $ is invariant with respect to $ V $. Using Proposition \ref{Blowprop} and Lemma \ref{DWW59}, it can be seen that $ h $ is a stationary solution of $ \Delta h=h^{-p} $ in $ \R^2 $. However, it follows from Lemma \ref{DWW58} that $ \{h=0\}=\{(0,0)\} $, a contradiction.    
\end{proof}

\begin{lem}\label{SSgeq}
For any $ k\in\Z\cap[0,n-2] $, $ S^k(u)=\Sg^k(u) $, where
$$
\Sg^k(u):=\{x\in\{u=0\}:\dim(\Sg(h))\leq k\text{ for any tangent function }h\text{ of }u\text{ at }x\}.
$$
\end{lem}
\begin{proof}
If $ x\notin S^k(u) $, there is a tangent function $ h $ of $ u $ at $ x $ such that $ h $ is $ (k+1) $-symmetric with respect to $ V\in\bG(n,k+1) $. Simple calculations imply that for any $ y\in V $, $ h $ is $ 0 $-symmetric at $ y $. Thus, we obtain that for any $ y\in V $,
$$
\vt(h;y)=\vt(h;y,1)=\vt(h;0,1)=\vt(h;0),
$$
where for the second inequality, we have used the invariance of $ h $ with respect to $ V $. As a result, $ V\subset\Sg(h) $, $ \dim(\Sg(h))\geq k+1 $, and $ x\notin\Sg^k(u) $. 

On the other hand, if $ x\notin\Sg^k(u) $, then Lemma \ref{DWW58} yields that there is a tangent function $ h $, which is invariant with respect to $ \Sg(h) $ such that $ \dim(\Sg(h))\geq k+1 $. Incorporating with Lemma \ref{tangent0}, $ h $ is $ (k+1) $-symmetric and then $ x\notin S^k(u) $. 
\end{proof}

\begin{lem}\label{deltaaplem}
Let $ k\in\Z\cap[0,n-2] $. For any $ x\in\Sg^k(u) $ and $ \delta>0 $, there exists $ \va>0 $, depending only on $ \delta,f,u $, and $ x $ such that if $ \rho\in(0,\min\{\delta,\dist(x,\pa\om)\}] $, then for some $ V\in\bG(n,k) $,
\be
\eta_{x,\rho}(\{y\in B_{\rho}(x):\vt(u;y)\geq\vt(u;x)-\va\})\subset B_{\delta}(V).\label{deltaaplemf}
\ee    
\end{lem}

\begin{rem}
Since $ x\in\Sg^k(u)\subset\{u=0\} $, it follows from Lemma \ref{propupture} that $ \vt(u;x)>-\ift $. Consequently, the left-hand side of \eqref{deltaaplemf} is a subset of $ \{u=0\} $.
\end{rem}

\begin{proof}
If such a result is not true, there exist $ \delta_0>0 $, $ x_0\in\Sg^k(u) $, $ \va_i\to 0^+ $, and $ \rho_i\to 0^+ $ such that for any $ V\in\bG(n,k) $,
\be
\{y\in B_1:\vt(T_{x_0,\rho_i}u,y)\geq\vt(u,x_0)-\va_i\}\not\subset B_{\delta_0}(V).\label{anynot}
\ee
Using Proposition \ref{Blowprop}, there exists $ h\in(C_{\loc}^{0,\al}\cap H_{\loc}^1\cap L_{\loc}^{-p})(\R^n) $ such that up to a subsequence
\be
T_{x_0,\rho_i}u\to h\text{ strongly in }(H_{\loc}^1\cap L_{\loc}^{\ift})(\R^n),\label{Tx0rhoilimit}
\ee
where $ h $ is also a stationary solution of \eqref{MEMSeq} with respect to $ f\equiv 0 $. Since $ x_0\in\Sg^k(u) $, we have $ \dim(\Sg(h))\leq k $. Thus, there exists $ V_0\in\bG(n,k) $ such that $ \Sg(h)\subset V_0 $. By Lemma \ref{semicon}, $ \vt(h;\cdot) $ is upper semicontinuous. According to the property that an upper semicontinuous function achieves its maximums in all compact sets, there exists $ \xi_0>0 $ such that 
\be
\sup_{y\in\ol{B}_1\backslash B_{\delta_0}(V_0)}\vt(h;y)<\vt(h;0)-\xi_0.\label{vtmuxvt}
\ee
We claim that for sufficiently large $ i\in\Z_+ $,
\be
\{y\in B_1:\vt(T_{x_0,\rho_i}u,y)\geq\vt(h;0)-\xi_0\}\subset B_{\delta_0}(V_0).\label{claimvt}
\ee
Since Lemma \ref{tangent0} yields that $ \vt(h;0)=\vt(u;x_0) $, if this claim \eqref{claimvt} holds, then
$$
\{y\in B_1:\vt(T_{x_0,\rho_i}u,y)\geq\vt(u;x_0)-\va_i\}\subset B_{\delta_0}(V_0),
$$
which is a contradiction to \eqref{anynot}. Let us now turn to the proof of \eqref{claimvt}. If the result is not true, then there exists a sequence of points $ \{y_i\}\subset B_1\backslash B_{\delta_0}(V_0) $ such that for any $ i\in\Z_+ $,
\be
\vt(T_{x_0,\rho_i}u;y_i)\geq\vt(h;0)-\xi_0.\label{thetax0T}
\ee
Up to a subsequence, we assume that 
\be
y_i\to y_{\ift}\in\ol{B}_1\backslash B_{\delta_0}(V_0).\label{yinftyin}
\ee
It follows from \eqref{thetax0T} and Proposition \ref{MonFor} that 
$$
\vt_{T_{x_0,\rho_i}^*f}(T_{x_0,\rho_i}u;y_i,r)\geq\vt(T_{x_0,\rho_i}u;y_i)\geq\vt(h;0)-\xi_0
$$
for any $ r>0 $. Consequently, by \eqref{thetaftheta} and \eqref{Tx0rhoilimit}, we have that for any $ r>0 $,
\be
\vt(h;0)-\xi_0\leq\lim_{i\to+\ift}\vt_{T_{x_0,\rho_i}^*f}(T_{x_0,\rho_i}u;y_i,r)=\vt(h;y_{\ift},r).\label{lettr0}
\ee
Letting $ r\to 0^+ $ in \eqref{lettr0}, it yields that $
\vt(h;0)-\xi_0\leq\vt(h,y_{\ift}) $. Given \eqref{yinftyin}, it contradicts \eqref{vtmuxvt}, and the claim \eqref{claimvt} is proved.    
\end{proof}

\begin{proof}[Proof of Proposition \ref{Hausdorffdim}]
In view of Lemma \ref{Sn2} and \ref{SSgeq}, we only need to show that for any $ k\in\Z\cap[0,n-2] $,
\be
\dim_{\HH}(\Sg^k(u))\leq k,\label{Sgkleq}
\ee
and $ \Sg^0(u) $ is discrete for $ n=2 $. Fix $ \delta>0 $ and $ k\in\Z\cap[0,n-2] $. Assume that $ \Sg^{k,i}(u) $ is the set of points $ x\in\Sg^k(u) $ such that for any $ \rho\in(0,i^{-1}] $, there exists $ V\in\bG(n,k) $, satisfying
\be
\eta_{x,\rho}(\{y\in B_\rho(x):\vt(u;y)\geq\vt(u;x)-i^{-1}\})\subset B_{\delta}(V).\label{BdeltaVhol}
\ee
By Lemma \ref{deltaaplem}, we have
$$
\Sg^k(u)=\bigcup_{i\in\Z_+}\Sg^{k,i}(u).
$$
For $ \ell\in\Z $, define
\be
\Sg^{k,i,\ell}(u):=\{x\in\Sg^{k,i}(u):\vt(u;x)\in((\ell-1)i^{-1},\ell i^{-1}]\}.\label{defSgkiq}
\ee
Thus, we have 
$$
\Sg^k(u)=\bigcup_{i,\ell\in\Z_+}\Sg^{k,i,\ell}(u).
$$
For any $ x\in\Sg^{k,i,\ell}(u) $ and $ 0<\rho\leq i^{-1} $, we choose $ V(x,\rho)\in\bG(n,k) $ such that \eqref{BdeltaVhol} holds. As a result, \eqref{defSgkiq} yields that
$$
\Sg^{k,i,\ell}(u)\subset\{y:\vt(u;y)>\vt(u;x)-i^{-1}\}.
$$
In particular,
\begin{align*}
\eta_{x,\rho}(\Sg^{k,i,\ell}(u))\cap B_1&\subset\eta_{x,\rho}(\{y:\vt(u;y)\geq\vt(u;x)>i^{-1}\})\cap B_1\\
&=\eta_{x,\rho}(\{y\in B_{\rho}(x):\vt(u;y)\geq\vt(u;x)>i^{-1}\})\subset B_{\delta}(V(x,\rho)).
\end{align*}
Then $ \Sg^{k,i,\ell}(u) $ has the $ (\delta,k) $-approximation property for any $ \delta>0 $ with $ \rho_0=i^{-1} $. It follows from Lemma \ref{lemhau} that for any $ i,\ell\in\Z_+ $, $
\dim_{\HH}(\Sg^{k,i,\ell}(u))\leq k $, which implies \eqref{Sgkleq}.

Finally, we show that for $ n=2 $, $ \Sg^0(u) $ is discrete. If the statement is false, without loss of generality, we assume that $ x_i\to 0\in\om $, $ u(0)=0 $, and for any $ i\in\Z_+ $, $ u(x_i)=0 $. Define 
$ u_i=T_{0,|x_i|}u $. Up to a subsequence, there is a tangent function $ h $ of $ u $ at $ 0 $ such that $ u_i $ converges to $ h $ in the sense of Proposition \ref{Blowprop} and $ h $ is a stationary solution of \eqref{MEMSeq} in $ \R^2 $ with respect to $ f\equiv 0 $. Moreover, we have $ \f{x_i}{|x_i|}\to x_{\ift}\in\Ss^1 $. As a result, $ h(x_{\ift})=0 $. Since by Lemma \ref{tangent0}, $ h $ is $ 0 $-symmetric at $ 0 $, we have $ \R x_{\ift}\subset\{h=0\} $, which is a contradiction to Lemma \ref{DWW58}.
\end{proof}

\part{Quantitative stratification} This part is dedicated to developing the quantitative stratification theory based on the concepts and frameworks established in \cite{NV17} and \cite{NV18} for harmonic maps. Finally, we will prove the main theorems presented in our paper by the conclusion of this section.

\section{Introduction and mains theorems} \label{SettingsQuantitative}
Quantitative stratification was first introduced by Cheeger and Naber in \cite{CN13a} and \cite{CN13b} in their studies of Gromov-Hausdorff limits, harmonic maps, and minimal currents. In the subsequent work \cite{NV17}, Naber and Valtorta expanded upon these techniques by employing Reifenberg-type results from geometric measure theory, thereby enhancing the conclusions drawn in \cite{CN13b} specifically for harmonic maps. Building on this foundation, in \cite{NV18}, the same authors simplified the arguments presented in \cite{NV17} and demonstrated parallel results for approximate harmonic maps. The literature on this topic is extensive, showcasing various applications of these methods. For interested readers, we highlight some notable studies, including \cite{Alp18,Alp20,DMSV18,EE19,FWZ24,HSV19,Sin18,Ved21,Wan21}. The central idea of quantitative stratification is that since directly analyzing the $ k $-stratum $ S^k(u) $ is difficult, it is natural to conduct an approximation and examine such ``approximating stratum".

\subsection{Settings and definitions} Given Definition \ref{ksymmetryf} of $ k $-symmetric functions, we can define quantitative symmetry for functions. Since we mainly consider the interior-type results, for simplicity, in the rest of this paper, we let $ R_0\in(100,200) $ and primarily focus on stationary solutions of \eqref{MEMSeq} with $ \om=B_{4R_0} $.

\begin{defn}[Quantitative symmetry]\label{qunsybypairMEMS}
Let $ \va>0 $, $ k\in\Z\cap[0,n] $, and $ u\in C_{\loc}^{0,\al}(B_{4R_0}) $. We say that $ u $ is $ (k,\va) $-symmetric in $ B_r(x)\subset\subset B_{4R_0} $, if there exists a $ k $-symmetric function $ h\in C_{\loc}^{0,\al}(\R^n) $ (or simply $ h\in C^{0,\al}(\ol{B}_1) $, which is $ k $-symmetric in $ B_1 $) such that
\be
\|T_{x,r}(u-u(x))-h\|_{L^{\ift}(B_1)}<\va.\label{Holdernormsmall}
\ee
\end{defn}

\begin{rem}
Quantitative symmetry implies that the blow-up of $ u-u(x) $ at $ x $ with scale $ r $, is in the $ \va $-neighborhood of a $ k $-symmetric function $ h $.
\end{rem}

\begin{rem}\label{scalerem}
In the above definition, using a change of variables, we see that the function $ u $ is $ (k,\va) $-symmetric in $ B_r(x) $, if and only if $ T_{x,r}u $ is $ (k,\va) $-symmetric in $ B_1 $.
\end{rem}

\begin{rem}
In \eqref{Holdernormsmall}, we use the norm $ \|\cdot\|_{L^{\ift}(B_1)} $ to characterize the difference between the blow-up and symmetric functions. Additionally, it is flexible to choose such a norm. One can also use the $ L^2 $-norm, and all the results still hold. The criterion is to use the norm corresponding to the convergence results of the blow-ups for solutions given in Proposition \ref{Blowprop}.
\end{rem}

Through the notion of quantitative symmetry, analogous to the definition of $ k $-stratum $ S^k(u) $, we can give our stratification in the quantitative form.

\begin{defn}[Quantitative stratification]\label{defSkvarMEMS}
Let $ \ga>0 $. Assume that $ u\in(C_{\loc}^{0,\al}\cap H_{\loc}^1\cap L_{\loc}^{-p})(B_{4R_0}) $ is a stationary solution of \eqref{MEMSeq} with respect to $ f\in M_{\loc}^{2\al+n-4+\ga,2}(B_{4R_0}) $. For any $ \va>0 $, $ k\in\Z\cap[0,n-1] $, and $ 0<r<1 $, the $ k $-th $ (\va,r) $-stratification of $ u $, denoted by $ S_{\va,r}^k(u) $, is given by
$$
S_{\va,r}^k(u):=\{x\in B_{R_0}:u\text{ is not }(k+1,\va)\text{-symmetric in }B_s(x)\text{ for any }r\leq s<1\}.
$$
We also define
\be
S_{\va}^k(u):=\bigcap_{0<r<1}S_{\va,r}^k(u).\label{definSva1}
\ee
In other words,
$$
S_{\va}^k(u)=\{x\in B_{R_0}:u\text{ is not }(k+1,\va)\text{ symmetric in }B_r(x)\text{ for any }0<r<1\}.
$$
\end{defn}

\begin{rem}\label{inclusionSvak}
For $ \va,\va'>0 $, $ k,k'\in\Z\cap[0,n-1] $, and $ r,r'\in(0,1) $. If $ \va\geq\va' $, $ k\leq k' $, and $ r\leq r' $, then $
S_{\va,r}^k(u)\subset S_{\va',r'}^{k'}(u) $ and $ S_{\va}^k(u)\subset S_{\va'}^{k'}(u) $.
\end{rem}

A direct consequence of the above definition is that we can use $ S_{\va,r}^k(u) $ to characterize the $ k $-stratum $ S^k(u) $. Precisely, we have the following lemma.

\begin{lem}\label{decomSkuseSva}
Let $ k\in\Z\cap[0,n-1] $. Suppose that $ u $ and $ f $ are the same as in Definition \ref{defSkvarMEMS}. Then
\be
S^k(u)\cap B_{R_0}=\bigcup_{\va>0}\bigcap_{0<r<1}S_{\va,r}^k(u).\label{SMdecomposition}
\ee
\end{lem}
\begin{proof}
We define the right-hand side of \eqref{SMdecomposition} as $ A^k(u) $. Assume that $ x\notin S^k(u)\cap B_{R_0} $. If $ u(x)=0 $, we can use the first point of Proposition \ref{Blowprop} to obtain a sequence $ r_i\to 0^+ $ and a $ (k+1) $-symmetric tangent function of $ u $ at $ x $, denoted by $ h\in (C_{\loc}^{0,\al}\cap H_{\loc}^1\cap L_{\loc}^{-p})(\R^n) $ such that
\be
T_{x,r_i}u\to h\text{ strongly in }(H_{\loc}^1\cap L_{\loc}^{\ift})(\R^n).\label{Txricon}
\ee
For any $ \va>0 $, as long as $ r_i>0 $ is sufficiently small, by \eqref{Txricon}, we have
$$
\|T_{x,r_i}(u-u(x))-h\|_{L^{\ift}(B_1)}=\|T_{x,r_i}u-h\|_{L^{\ift}(B_1)}<\va,
$$
which implies that $ x\notin A^k(u) $. Thus, $ A^k(u)\subset S^k(u)\cap B_{R_0} $. If $ u(x)>0 $, then it follows from the second point of Proposition \ref{Blowprop} that there exists $ r_i\to 0^+ $ such that
$$
T_{x,r_i}(u-u(x))\to 0\text{ strongly in }(H_{\loc}^1\cap L_{\loc}^{\ift})(\R^n).
$$
The zero function is $ (k+1) $-symmetric. For any $ \va>0 $, choosing sufficiently small $ r_i>0 $, it yields that $
\|T_{x,r_i}(u-u(x))\|_{L^{\ift}(B_1)}<\va $, and consequently, $ x\notin A^k(u) $. Then we have $ A^k(u)\subset S^k(u)\cap B_{R_0} $.

On the other hand, suppose that $ x\notin A^k(u) $. Without loss of generality, we let $ u(x)=0 $, since if $ u(x)>0 $, then $ x\notin S^k(u)\cap B_{R_0} $, due to Remark \ref{remSn1u0}. As a result, we can choose $ \{r_i\}\subset(0,1) $ and a sequence of $ (k+1) $-symmetric functions $ \{h_i\}\subset C_{\loc}^{0,\al}(\R^n) $ such that
\be
\|T_{x,r_i}u-h_i\|_{L^{\ift}(B_1)}<i^{-1}.\label{hiiminus1}
\ee
If $ r:=\inf r_i>0 $, then we can assume that $ r_i\to r $. As a result, $ T_{x,r_i}u\to T_{x,r}u $ strongly in $ L^{\ift}(B_1) $. It follows from \eqref{hiiminus1} that $ h_i\to T_{x,r}u $ strongly in $ L^{\ift}(B_1) $. Applying Lemma \ref{ConSym}, $ T_{x,r}u $ is $ (k+1) $-symmetric. This implies that $ u $ is $ (k+1) $-symmetric at $ x $. Thus, any tangent function of $ u $ at $ x $ is $ (k+1) $-symmetric and $ x\notin S^k(u)\cap B_{R_0} $. Now, we suppose that $ \inf r_i=0 $. Given the first property of Proposition \ref{Blowprop}, up to a subsequence, the convergence \eqref{Txricon} holds, and $ h\in(C_{\loc}^{0,\al}\cap H_{\loc}^1\cap L_{\loc}^{-p})(\R^n) $ is a tangent function of $ u $ at $ x $. Incorporating with \eqref{hiiminus1}, we have $ h_i\to h $ strongly in $ L^{\ift}(B_1) $. Consequently, Lemma \ref{ConSym} implies that $ h $ is $ (k+1) $-symmetric, and then $ x\notin S^k(u)\cap B_{R_0} $, hence $ S^k(u)\cap B_{R_0}\subset A^k(u) $.
\end{proof}

\subsection{Estimates on quantitative stratification}

We now present the main theorem for quantitative stratification as follows.

\begin{thm}\label{quantitativethm}
Let $ \ga>0 $, $ \va>0 $ and $ k\in\Z\cap[0,n-2] $. Assume that $ u\in(C_{\loc}^{0,\al}\cap H_{\loc}^1\cap L_{\loc}^{-p})(B_{4R_0}) $ is a stationary solution of \eqref{MEMSeq} with respect to $ f\in M_{\loc}^{2\al+n-4+\ga,2}(B_{4R_0}) $, satisfying
\be
[u]_{C^{0,\al}(\ol{B}_{2R_0})}+[f]_{M^{2\al+n-4+\ga,2}(B_{2R_0})}\leq\Lda.\label{quantitativethmass}
\ee
There is a constant $ C>0 $, depending only on $ \va,\ga,\Lda,n $, and $ p $ such that the following properties hold.
\begin{enumerate}[label=$(\theenumi)$]
\item If $ 0<r<1 $, then
\be
\cL^n(B_r(S_{\va,r}^k(u)))\leq Cr^{n-k}.\label{quanti1}
\ee
In particular, it follows from \eqref{definSva1} that
\be
\cL^n(B_r(S_{\va}^k(u)))\leq Cr^{n-k}.\label{quanti11}
\ee
\item For any $ x\in B_{R_0} $ and $ 0<r<1 $, we have 
\be
\HH^k(S_{\va}^k(u)\cap B_r(x))\leq Cr^k,\label{quanti2}
\ee
which is equivalent to say that $ S_{\va}^k(u) $ is upper Ahlfors $ k $-regular.
\item Moreover, for $ \HH^k $-a.e. $ x\in S_{\va}^k(u) $ or $ S^k(u) $, there exists $ V\in\bG(n,k) $ such that any tangent function of $ u $ at $ x $ is $ k $-symmetric with respect to $ V $.
\end{enumerate}
\end{thm}

\begin{rem}
The proof of this theorem depends on the Reifenberg-type theorems developed in \cite{NV17} and references therein.
\end{rem}

\begin{rem}\label{remSvau0sub}
Using Lemma \ref{decomSkuseSva}, we see that for any $ k\in\Z\cap[0,n-1] $, $ S_{\va}^k(u)\subset\{u=0\} $. Notably, $ S_{\va,r}^k(u) $ is not necessarily a subset of the rupture set. The essential point of Theorem \ref{quantitativethm} is that \eqref{quanti1} gives the estimate of $ S_{\va,r}^k(u) $, instead of $ \{u=0\}\cap S_{\va,r}^k(u) $. For $ \{u=0\}\cap S_{\va,r}^k(u) $, the proof is much simpler, but there is a loss of important information about solutions.
\end{rem}

\section{Properties on quantitative stratification}\label{quantitatvePro}

\subsection{Rupture sets and quantitative stratification} 

Given Proposition \ref{Hausdorffdim}, the rupture set is actually $ S^{n-2}(u) $. Moreover, using properties of $ S_{\va,r}^{n-2}(u) $, we can give a quantitative form of such a result.

\begin{prop}\label{corcomMEMS}
Let $ \ga>0 $. Assume that $ u\in(C_{\loc}^{0,\al}\cap H_{\loc}^1\cap L_{\loc}^{-p})(B_{4R_0}) $ is a stationary solution of \eqref{MEMSeq} with respect to $ f\in M_{\loc}^{2\al+n-4+\ga,2}(B_{4R_0}) $, satisfying 
$$
[u]_{C^{0,\al}(\ol{B}_{2R_0})}+[f]_{M_{\loc}^{2\al+n-4+\ga,2}(B_{2R_0})}\leq\Lda.
$$
Then there exists $ \va>0 $, depending only on $ \ga,\Lda,n $, and $ p $ such that for any $ 0<r<1 $,
\be
\{x\in B_{R_0}:u(x)<\va r^{\al}\}\subset S_{\va,r}^{n-2}(u).\label{varales}
\ee
\end{prop}

\begin{rem}\label{remn2recti}
By \eqref{definSva1}, \eqref{varales}, and Remark \ref{remSvau0sub}, we see that
$$
\{x\in B_{R_0}:u(x)=0\}=S_{\va}^{n-2}(u)
$$
for some $ \va>0 $ depending only on $ \ga,\Lda,n $, and $ p $.
\end{rem}

Proposition \ref{corcomMEMS} is a direct consequence of the lemma below.

\begin{lem}\label{comMEMS}
Let $ \ga>0 $, $ 0<s\leq 1 $, and $ x\in\R^n $. Assume that $ u\in (C_{\loc}^{0,\al}\cap H_{\loc}^1\cap L_{\loc}^{-p})(B_{4s}(x)) $ is a stationary solution of \eqref{MEMSeq} with respect to $ f\in M_{\loc}^{2\al+n-4+\ga,2}(B_{4s}(x)) $, satisfying
$$
[u]_{C^{0,\al}(\ol{B}_{2s}(x))}+[f]_{M^{2\al+n-4+\ga,2}(B_{2s}(x))}\leq\Lda.
$$
There exists $ \va>0 $, depending only on $ \ga,\Lda,n $, and $ p $ such that if $ u $ is $ (n-1,\va) $-symmetric in $ B_s(x) $, then $
u(x)\geq\va s^{\al} $.
\end{lem}
\begin{proof}
By Proposition \ref{ScalingProp} and Remark \ref{scalerem}, we assume that $ s=1 $ and $ x=0 $. If the result is not true, then there exists a sequence of stationary solutions $ \{u_i\}\subset (C_{\loc}^{0,\al}\cap H_{\loc}^1\cap L_{\loc}^{-p})(B_4) $ of \eqref{MEMSeq} with respect to $ \{f_i\}\subset M_{\loc}^{2\al+n-4+\ga,2}(B_4) $ such that for any $ i\in\Z_+ $, the following properties hold.
\begin{itemize}
\item $ u_i $ and $ f_i $ satisfy  
\begin{gather}
[u_i]_{C^{0,\al}(\ol{B}_2)}+[f_i]_{M^{2\al+n-4+\ga,2}(B_2)}\leq\Lda,\label{uifiass1}\\
0\leq u_i(0)<i^{-1}.\label{n2uifiass12}
\end{gather}
\item $ u_i $ is $ (n-1,i^{-1}) $-symmetric in $ B_1 $. In particular, there exists $ h_i\in C_{\loc}^{0,\al}(\R^n) $, being $ (n-1) $-symmetric with respect to $ V_i\in\bG(n,n-1) $ such that
\be
\|(u_i-u_i(0))-h_i\|_{L^{\ift}(B_1)}<i^{-1}.\label{n1uvii}
\ee
\end{itemize}
Estimates \eqref{uifiass1} and \eqref{n2uifiass12} yield that 
$$
\sup_{i\in\Z_+}\|u_i\|_{L^2(B_2)}\leq C(\Lda,n,p).
$$
By Proposition \ref{propConv}, there exist $ u_{\ift}\in C^{0,\al}(\ol{B}_2)\cap (H_{\loc}^1\cap L_{\loc}^{-p})(B_2) $ and $ f_{\ift}\in M^{2\al+n-4+\ga,2}(B_2) $ such that up to a subsequence,
\begin{align*}
&u_i\to u_{\ift}\text{ strongly in }(H_{\loc}^1\cap L^{\ift})(B_2),\\
&f_i\to f_{\ift}\text{ weakly in }L^2(B_2).
\end{align*}
Moreover, $ u_{\ift} $ is a stationary solution of \eqref{MEMSeq} with respect to $ f_{\ift} $ in $ B_2 $. In particular, due to \eqref{n2uifiass12}, we have $ u_{\ift}(0)=\lim_{i\to+\ift}u_i(0)=0 $. Given \eqref{uifiass1} and \eqref{n1uvii}, by further extracting subsequences, we have $ V_i\to V_{\ift} $ and
$$
h_i\to u_{\ift}\text{ strongly in }L^{\ift}(B_1).
$$
As a result, Lemma \ref{ConSym} implies that $ u_{\ift} $ is $ (n-1) $-symmetric with respect to $ V_{\ift} $. According to the property that $ u_{\ift}(0)=0 $, it can be seen that for any $ y\in V_{\ift}\cap B_1 $, $ u_{\ift}(y)=0 $. Consequently, 
$$
\dim_{\HH}(\{u_{\ift}=0\}\cap B_1)\geq n-1,
$$
which is a contradiction to Proposition \ref{Hausdorffdim}. 
\end{proof}

\begin{proof}[Proof of Proposition \ref{corcomMEMS}]
Let $ 0<r<1 $. Assume that $ \va>0 $ is to be determined. For any $ x\notin S_{\va,r}^{n-2}(u) $, by the definition of $ S_{\va,r}^{n-2}(u) $, there exists some $ r\leq s<1 $ such that $ u $ is $ (n-1,\va) $-symmetric in $ B_s(x) $. Applying Lemma \ref{comMEMS}, if $ \va=\va(\ga,\Lda,n,p)>0 $ is sufficiently small, then $ u(x)>\va s^{\al}\geq\va r^{\al} $, and $ x\notin\{x\in B_{R_0}:u(x)<\va r^{\al}\} $, which implies \eqref{varales}.
\end{proof}

\subsection{Characterization of quantitative symmetry} In Definition \ref{qunsybypairMEMS}, we need to use the $ k $-symmetric function $ h $ as an approximation. Here, we intend to find another one that is not dependent on such approximating functions to describe quantitative symmetry. They are in the similar spirits of \cite{HSV19} and \cite{Ved21}.

\begin{lem}\label{SmaHomMEMS}
Let $ \ga>0 $, $ 0<s\leq 1 $, $ k\in\Z\cap[0,n-1] $, and $ x\in\R^n $. Assume that $ u\in (C_{\loc}^{0,\al}\cap H_{\loc}^1\cap L_{\loc}^{-p})(B_{20s}(x)) $ is a stationary solution of \eqref{MEMSeq} with respect to $ f\in M_{\loc}^{2\al+n-4+\ga,2}(B_{20s}(x)) $, satisfying
$$
[u]_{C^{0,\al}(\ol{B}_{15s}(x))}+[f]_{M^{2\al+n-4+\ga,2}(B_{15s}(x))}\leq\Lda.
$$
For any $ \va>0 $, there exists $ \delta>0 $ depending only on $ \va,\ga,\Lda,n $, and $ p $ such that if
\be
\vt_f(u;x,s)-\vt_f\(u;x,\f{s}{2}\)<\delta,\label{Vinaucon0}
\ee
and
\be
\inf_{V\in\bG(n,k)}\(s^{2-2\al-n}\int_{B_s(x)}|V\cdot\na u|^2\)<\delta,\label{Vinaucon}
\ee
then $ u $ is $ (k,\va) $-symmetric in $ B_s(x) $. Here, we have the convention that for $ k=0 $, the right-hand side of \eqref{Vinaucon} is $ 0 $, and this assumption is trivially true. 
\end{lem}

\begin{proof}
Using Proposition \ref{ScalingProp} and Remark \ref{scalerem}, we assume that $ s=1 $ and $ x=0 $. For simplicity, we let $ k\geq 1 $, and for $ k=0 $, the result follows directly from the proof with $ k\geq 1 $. Suppose that the statement is not true, there exist $ \va_0>0 $ and $ \{u_i\}\subset (C_{\loc}^{0,\al}\cap H_{\loc}^1\cap L_{\loc}^{-p})(B_{20}) $, being a sequence of stationary solutions of \eqref{MEMSeq} with respect to $ \{f_i\}\subset M_{\loc}^{2\al+n-4+\ga,2}(B_{20}) $ such that for any $ i\in\Z_+ $, the following properties hold.
\begin{itemize}
\item $ u_i $ and $ f_i $ satisfy  
\begin{gather}
[u_i]_{C^{0,\al}(\ol{B}_{15})}+[f_i]_{M^{2\al+n-4+\ga,2}(B_{15})}\leq\Lda,\label{uifiass2}\\
\vt_{f_i}(u_i;0,1)-\vt_{f_i}\(u_i;0,\f{1}{2}\)<i^{-1},\label{thetatwo1}\\
\int_{B_1}|V_i\cdot\na u_i|^2<i^{-1},\quad V_i\in\bG(n,k).\label{Vnaui1}
\end{gather}
\item $ u_i $ is not $ (k,\va_0) $-symmetric in $ B_1 $.
\end{itemize}
Using \eqref{thetatwo1} and Corollary \ref{coruse}, we have
\be
\int_{B_4}|y\cdot\na u_i-\al u_i|^2\ud y\leq C(n,p)i^{-1}\label{useupperui}
\ee
for any $ i\in\Z_+ $. This, together with \eqref{uifiass2} and Lemma \ref{InESLem2}, implies that
$$
\sup_{i\in\Z_+}\|u_i\|_{L^2(B_{15})}\leq C(\ga,\Lda,n,p).
$$
As a result, by \eqref{uifiass2} and Proposition \ref{propConv}, there exist $ u_{\ift}\in C^{0,\al}(\ol{B}_{15})\cap H_{\loc}^1(B_{15}) $ and $ V_{\ift}\in\bG(n,k) $ such that up to a subsequence, $ V_i\to V_{\ift} $ and
\be
u_i\to u_{\ift}\text{ strongly in }(H_{\loc}^1\cap L^{\ift})(B_{15}).\label{conuvift}
\ee
The inequality \eqref{useupperui} yields
$$
0\leq\int_{B_4}|y\cdot\na u_{\ift}-\al u_{\ift}|^2\ud y=\lim_{i\to+\ift}\int_{B_4}|y\cdot\na u_i-\al u_i|^2\ud y=0.
$$
Thus, $ u_{\ift} $ is $ 0 $-symmetric in $ B_4 $. Remark \ref{remh0} shows that $ u_{\ift}(0)=0 $. Moreover, \eqref{Vnaui1} and \eqref{conuvift} imply that
$$
\int_{B_1}|V_{\ift}\cdot\na u_{\ift}|^2=\lim_{i\to+\ift}\int_{B_1}|V_i\cdot\na u_i|^2=0,
$$
and consequently, $ u_{\ift} $ is invariant with respect to $ V_{\ift} $ in $ B_1 $. Applying Lemma \ref{extend}, we regard $ u_{\ift} $ as a $ k $-symmetric function with respect to $ V_{\ift} $. For $ i\in\Z_+ $ sufficiently large, it follows from \eqref{conuvift} that $
\|u_i-u_{\ift}\|_{L^{\ift}(B_1)}<\f{\va_0}{2} $. Since $ u_{\ift}(0)=0 $, we obtain $ 0\leq u_i(0)<\f{\va_0}{2} $. It gives that
$$
\|(u_i-u_i(0))-u_{\ift}\|_{L^{\ift}(B_1)}\leq|u_i(0)|+\|u_i-u_{\ift}\|_{L^{\ift}(B_1)}<\va_0,
$$
and then $ u_i $ is $ (k,\va_0) $-symmetric in $ B_1 $, which is a contradiction to the assumption.
\end{proof}

We next present some results on cone-splitting properties and the concept of effective spanned subspaces. These notions and facts are fundamental in the preceding analysis.

\begin{lem}[Cone-splitting]\label{ConSpl}
If $ u\in C_{\loc}^{0,\al}(\R^n) $ is $ 0 $-symmetric at $ x_1,x_2\in\R^n $ with $ x_1\neq x_2 $, then $ u $ is $ 1 $-symmetric at $ x_1 $ with respect to $ \op{span}\{x_1-x_2\} $. 
\end{lem}

\begin{cor}\label{CorSpl}
Let $ k\in\Z\cap[0,n-1] $. If $ u\in C_{\loc}^{0,\al}(\R^n) $ is $ k $-symmetric with respect to $ V\in\bG(n,k) $, and is $ 0 $-symmetric at $ x\notin V $, then $ u $ is $ (k+1) $-symmetric with respect to $ \op{span}\{x,V\} $.
\end{cor}

\begin{rem}
The proof of Lemma \ref{ConSpl} and Corollary \ref{CorSpl} follows from straightforward calculations, and we omit it for simplicity. For similar results, readers can also refer to \S 4 of \cite{CN13b}. 
\end{rem}

\begin{rem}
The above cone-splitting results imply that the homogeneity (or $ 0 $-symmetry) of functions leads to the improvements of invariance with respect to subspaces of $ \R^n $.   
\end{rem}

\begin{defn}[Effectively spanned subspace]
In $ \R^n $, for $ k\in\Z\cap[1,n] $, let $ \{x_i\}_{i=0}^k\subset\R^n $ and $ s>0 $. We say that these points $ s$-effectively span $ L=x_0+\op{span}\{x_i-x_0\}_{i=1}^k\in\bA(n,k) $ if for all $ i\in\Z\cap[2,k] $,
$$
\dist(x_i,x_0+\op{span}\{x_1-x_0,...,x_{i-1}-x_0\})\geq 2s.
$$
We also say such points $ s $-independent. For a set $ F\subset\R^n $, we say that it $ s $-effectively spans a $ k $-dimensional affine subspace if there exist $ \{x_i\}_{i=0}^k\subset F $, which are $ s $-independent.
\end{defn}

\begin{lem}\label{rhoind}
Let $ k\in\Z\cap[1,n] $. We have the following properties.
\begin{enumerate}[label=$(\theenumi)$]
\item If $ \{x_i\}_{i=0}^k $ $ s $-effectively span $ L\in\bA(n,k) $, then for any $ x\in L $, there exists a unique set of numbers $ \{\al_i\}_{i=1}^k $ such that
$$
x=x_0+\sum_{i=1}^k\al_i(x_i-x_0)\quad\text{and}\quad|\al_i|\leq \f{C|x-x_0|}{s},
$$
where $ C>0 $ depends only on $ n $.
\item If $ \{x_{i,j}\}_{i=1}^k $ are $ s $-independent for any $ j\in\Z_+ $, and $ x_{i,j}\to x_{i,\ift} $ for any $ i\in\Z\cap[1,k] $, then $ \{x_{i,\ift}\}_{i=1}^k $ are also $ s $-independent.
\end{enumerate}
\end{lem}
\begin{proof}
This lemma is a scaled form of Lemma 4.6 in \cite{NV17}.
\end{proof}

By compactness arguments similar to those in the proof of Lemma \ref{SmaHomMEMS}, we can obtain the following cone-splitting results in the quantitative form, which we can regard as a generalization of Lemma \ref{ConSpl} and Corollary \ref{CorSpl}.

\begin{prop}\label{QuaConSplMEMS}
Let $ 0<\beta<\f{1}{2} $, $ \ga>0 $, $ 0<s\leq 1 $, and $ k\in\Z\cap[0,n-1] $. Assume that $ u\in (C_{\loc}^{0,\al}\cap H_{\loc}^1\cap L_{\loc}^{-p})(B_{20s}) $ is a stationary solution of \eqref{MEMSeq} with respect to $ f\in M_{\loc}^{2\al+n-4+\ga,2}(B_{20s}) $, satisfying
$$
[u]_{C^{0,\al}(\ol{B}_{15s})}+[f]_{M^{2\al+n-4+\ga,2}(B_{15s})}\leq\Lda.
$$
Suppose that $ \{x_i\}_{i=0}^k\subset B_s $ with $ x_0=0 $. For any $ \va>0 $, there exists $ \delta>0 $, depending only on $ \beta,\va,\ga,\Lda,n $, and $ p $ such that if for any $ i\in\Z\cap[0,k] $,
$$
\vt_f(u;x_i,s)-\vt_f\(u;x_i,\f{s}{2}\)<\delta, 
$$
and $ \{x_i\}_{i=0}^k $ are $ \beta s $-independent, then $ u $ is $ (k,\va) $-symmetric in $ B_s $.
\end{prop}

\begin{proof}
Using proposition \ref{ScalingProp} and Remark \ref{scalerem}, we let $ s=1 $. If the result is not true, then there exist $ \va_0>0 $, a sequence of stationary solutions of \eqref{MEMSeq}, denoted by $ \{u_j\}\subset (C_{\loc}^{0,\al}\cap H_{\loc}^1\cap L_{\loc}^{-p})(B_{20}) $ with respect to $ \{f_i\}\subset M_{\loc}^{2\al+n-4+\ga,2}(B_{20}) $, and $ \{\{x_{i,j}\}_{i=0}^k\}\subset B_1 $ with $ x_{0,j}=0 $ such that for any $ j\in\Z_+ $, the following properties hold.
\begin{itemize}
\item $ u_j $ and $ f_j $ satisfy
\begin{gather}
[u_j]_{C^{0,\al}(\ol{B}_{15})}+[f_j]_{M^{2\al+n-4+\ga,2}(B_{15})}\leq\Lda,\label{uifiass3}\\
\sup_{i\in\Z\cap[0,k]}\[\vt_{f_j}(u_j;x_{i,j},1)-\vt_{f_j}\(u_j;x_{i,j},\f{1}{2}\)\]<j^{-1}.\label{0ksmall}
\end{gather}
\item $ \{x_{i,j}\}_{i=0}^k $ are $ \beta $-independent points.
\item $ u_j $ is not $ (k,\va_0) $-symmetric in $ B_1 $.
\end{itemize}
Using \eqref{0ksmall}, it follows that
\be
\sup_{0\leq i\leq k}\(\int_{B_4(x_{i,j})}|(y-x_{i,j})\cdot\na u_j-\al u_j|^2\ud y\)\leq C(n,p)j^{-1}.\label{supikCnpj1}
\ee
By \eqref{uifiass3} and Corollary \ref{coruse}, we obtain
$$
\sup_{j\in\Z_+}\|u_j\|_{L^2(B_{15})}\leq C(\ga,\Lda,n,p).
$$
The estimate \eqref{uifiass3} and Proposition \ref{propConv} implies that there exist $ u_{\ift}\in C^{0,\al}(\ol{B}_{15})\cap H_{\loc}^1(B_{15}) $ and $ \{x_{i,\ift}\}_{i=0}^k\subset\ol{B}_1 $ such that
\be
\begin{aligned}
u_j&\to u_{\ift}\text{ strongly in }(H_{\loc}^1\cap L^{\ift})(B_{15}),\\
x_{i,j}&\to x_{i,\ift}\text{ for any }i\in\Z\cap[0,k].
\end{aligned}\label{conn1compa22}
\ee
Thus, by \eqref{supikCnpj1}, we have
$$
\int_{B_4(x_{i,\ift})}|(y-x_{i,\ift})\cdot\na u_{\ift}-\al u_{\ift}|^2\ud y=0
$$
for any $ i\in\Z\cap[0,k] $. Consequently for any $ i\in\Z\cap[0,k] $, $ u_{\ift} $ is $ 0 $-symmetric at $ x_{i,\ift} $ in $ B_4(x_{i,\ift}) $. Moreover, the second property of \eqref{rhoind} gives that $ \{x_{i,\ift}\}_{i=0}^k $ are $ \beta $-independent. As a result, with the help of Lemma \ref{extend}, \ref{ConSpl}, and Corollary \ref{CorSpl}, without changing the notation, we deduce that $ u_{\ift} $ can be extended to a $ k $-symmetric function in $ C_{\loc}^{0,\al}(\R^n) $ with respect to $ V=\op{span}\{x_i\}_{i=1}^k $. For sufficiently large $ j\in\Z_+ $, \eqref{conn1compa22} shows that 
\begin{align*}
\|(u_j-u_j(0))-u_{\ift}\|_{L^{\ift}(B_1)}&\leq |u_j(0)|+\|u_j-u_{\ift}\|_{L^{\ift}(B_1)}\\
&=|u_j(0)-u_{\ift}(0)|+\|u_j-u_{\ift}\|_{L^{\ift}(B_1)}\\
&\leq 2\|u_j-u_{\ift}\|_{L^{\ift}(B_1)}<\va_0,
\end{align*}
where we have also used Remark \ref{remh0} to get $ u_{\ift}(0)=0 $. This is a contradiction to the assumption that $ u_j $ is not $ (k,\va_0) $-symmetric in $ B_1 $.
\end{proof}

We close this subsection by proving the following result using similar compactness arguments in the proofs of Lemma \ref{SmaHomMEMS} and \ref{QuaConSplMEMS}. Intuitively, we show that if the set of points with large density $ \vt_f(u;\cdot,\cdot) $ effectively spans an affine subspace, then the density is also large at any point in such a subspace.

\begin{lem}\label{lempinchMEMS}
Let $ 0<\beta<\f{1}{20} $, $ \ga>0 $, $ k\in\Z\cap[0,n-1] $, $ 0<s\leq 1 $, and $ x\in\R^n $. Assume that $ u\in(C_{\loc}^{0,\al}\cap H_{\loc}^1\cap L_{\loc}^{-p})(B_{25s}(x)) $ is a stationary solution of \eqref{MEMSeq} with respect to $ f\in M_{\loc}^{2\al+n-4+\ga,2}(B_{25s}(x)) $, satisfying
$$
[u]_{C^{0,\al}(\ol{B}_{20s}(x))}+[f]_{M^{2\al+n-4+\ga,2}(B_{20s}(x))}\leq\Lda.
$$
Let $ E>0 $ be such that
$$
\sup_{y\in B_{2s}(x)}\vt_f(u;y,s)\leq E.
$$
For any $ \xi>0 $, there exists $ 0<\delta<1 $, depending only on $ \beta,\ga,\Lda,n,p $, and $ \xi $ such that if 
$$
[f]_{M^{2\al+n-4+\ga,2}(B_{20s}(x))}<\delta,
$$
and the set
$$
F=\{y\in B_{2s}(x):\vt_f(u;y,\beta s)>E-\delta\}
$$
$ \beta s $-effectively spans $ L\in\bA(n,k) $, then
$$
\inf_{y\in L\cap B_{2s}(x)}\vt_f(u;y,\beta s)\geq E-\xi.
$$
\end{lem}
\begin{proof}
By Proposition \ref{ScalingProp} and Remark \ref{scalerem}, we assume that $ s=1 $ and $ x=0 $. Suppose that the statement is false. There are $ \xi_0>0 $ and a sequence of stationary solutions of \eqref{MEMSeq}, denoted by $ \{u_j\}\subset (C_{\loc}^{0,\al}\cap H_{\loc}^1\cap L_{\loc}^{-p})(B_{25}) $ with respect to $ \{f_j\}\subset M_{\loc}^{2\al+n-4+\ga,2}(B_{25}) $ such that for any $ j\in\Z_+ $, the following properties hold.
\begin{itemize}
\item $ u_j $ and $ f_j $ satisfy
\begin{align}
[u_j]_{C^{0,\al}(\ol{B}_{20})}&\leq\Lda,\label{fjsmall1}\\
[f_j]_{M^{2\al+n-4+\ga,2}(B_{20})}&<j^{-1}.\label{fjsmall}
\end{align}
\item For any $ y\in B_2 $, $
\vt_{f_j}(u_j;y,1)\leq E_j $.
\item The set 
$$
F_j:=\{y\in B_2:\vt_{f_j}(u_j;y,\beta)>E_j-j^{-1}\} 
$$
contains $ \{x_{i,j}\}_{i=0}^k $, which $ \beta $-effectively spans $ L_j\in\bA(n,k) $.
\item There exists $ y_j\in L_j\cap B_2 $ such that 
\be
\vt_{f_j}(u_j;y_j,\beta)<E_j-\xi_0.\label{ujyjEjxi0}
\ee
\end{itemize}
By the definition of $ F_j $ and $ E_j $ for any $ i\in\Z\cap[0,k] $ and $ j\in\Z_+ $, we have
\be
\begin{gathered}
\vt_{f_j}(u_j;x_{i,j},1)-\vt_{f_j}(u_j;x_{i,j},\beta)<j^{-1},\\
E_j-j^{-1}\leq\vt_{f_j}(u_j;x_{i,j},\beta)\leq E_j.
\end{gathered}\label{Fjjminus1}
\ee
Using Proposition \ref{MonFor}, Corollary \ref{coruse}, and the fact that $ 0<\beta<\f{1}{20} $, it follows that
\begin{align*}
\sup_{0\leq i\leq k,\,\,j\in\Z_+}\(\int_{B_4(x_{i,j})}|(y-x_{i,j})\cdot\na u_j-\al u_j|^2\)\leq C(n,p).
\end{align*}
This, together with \eqref{thetaftheta}, \eqref{fjsmall1}, \eqref{fjsmall}, Lemma \ref{InESLem1}, and Lemma \ref{InESLem2}, implies that 
\be
\sup_{j\in\Z_+}\|u_j\|_{L^2(B_{20})}\leq C(\ga,\Lda,n,p),\label{ujL2usebou}
\ee
and then
$$ 
\sup_{j\in\Z_+}|\vt_{f_j}(u;y,1)|\leq C(\ga,\Lda,n,p). 
$$
Thus, we have $ E_j\geq-C(\ga,\Lda,n,p) $.
This, together with the fact that for any $ j\in\Z_+ $, $ F_j\neq\emptyset $, implies that
\be
\sup_{j\in\Z_+}|E_j|\leq C(\ga,\Lda,n,p).\label{EjMleq}
\ee
Up to a subsequence, we assume that 
\be
x_{i,j}\to x_{i,\ift}\in\ol{B}_2\quad\text{and}\quad y_j\to y_{\ift}\in\ol{B}_2.\label{xiftyift}
\ee
Consequently, there is $ L_{\ift}\in\bA(n,k) $ such that $ L_j\to L_{\ift} $ with $ y_{\ift}\in L_{\ift} $. Indeed, by Lemma \ref{rhoind}, it yields that $
L_{\ift}:=x_{0,\ift}+\{x_{i,\ift}-x_{0,\ift}\}_{i=1}^k $. Moreover, according to \eqref{fjsmall1}, \eqref{fjsmall}, \eqref{ujL2usebou}, and \eqref{EjMleq}, we can apply Proposition \ref{propConv} and further extract a subsequence of $ j\in\Z_+ $ such that
\be
\begin{aligned}
E_j\to E_{\ift}\in\R,\quad\text{and}\quad u_j\to u_{\ift}\text{ strongly in }(H_{\loc}^1\cap L^{\ift})(B_{20}),
\end{aligned}\label{ujEjconv}
\ee
where $ u_{\ift}\in C^{0,\al}(\ol{B}_{20})\cap H_{\loc}^1(B_{20}) $. Combining with \eqref{Fjjminus1}, \eqref{xiftyift}, Definition \ref{defnofphi}, and Proposition \ref{MonFor}, we have that for any $ i\in\Z\cap[0,k] $,
\begin{align*}
0&\leq-\int_{\beta}^1\(\rho^{-2\al-n-1}\int_{\R^n}|(y-x_{i,\ift})\cdot\na u_{\ift}-\al u_{\ift}|^2\dot{\phi}_{x_{i,\ift},\rho}\ud y\)\ud\rho\\
&=\lim_{j\to+\ift}\[-\int_{\beta}^1\(\rho^{-2\al-n-1}\int_{\R^n}|(y-x_{i,j})\cdot\na u_j-\al u_j|^2\dot{\phi}_{x_{i,j},\rho}\ud y\)\ud\rho\]\\
&\leq\lim_{j\to+\ift}(\vt_{f_j}(u_j;x_{i,j},1)-\vt_{f_j}(u_j;x_{i,j},\beta))=0,
\end{align*}
and hence, $ u_{\ift} $ is $ 0 $-symmetric at $ x_{i,\ift} $ in $ B_8(x_i) $. We deduce from Lemma \ref{ConSpl} and Corollary \ref{CorSpl} that $ u_{\ift} $ is invariant with respect to $ L_{\ift} $ in $ B_4 $. Precisely, for any $ v\in V_{\ift}:=\op{span}\{x_i-x_0\}_{i=1}^k $ and $ y\in L_{\ift}\cap B_4 $, if $ y+v\in B_4 $, then
\be
u_{\ift}(y+v)=u_{\ift}(y).\label{invauift}
\ee
Using \eqref{thetaftheta}, \eqref{fjsmall}, \eqref{xiftyift}, \eqref{ujEjconv}, Lemma \ref{InESLem1}, and Lemma \ref{InESLem2}, we arrive at
\begin{align*}
|\vt_{f_j}(u_j;x_{i,j},\beta)-\vt(u_j;x_{i,j},\beta)|&\leq C(\beta,\ga,\Lda,n,p)j^{-\f{1}{2}},\\
|\vt_{f_j}(u_j;y_j,\beta)-\vt(u_j;y_j,\beta)|&\leq C(\beta,\ga,\Lda,n,p)j^{-\f{1}{2}}.
\end{align*}
Taking $ j\to+\ift $, it follows that for any $ i\in\Z\cap[0,k] $,
\begin{align*}
\lim_{j\to+\ift}\vt_{f_j}(u_j;x_{i,j},\beta)&=\lim_{j\to+\ift}\vt(u_j;x_{i,j},\beta)=\vt(u_{\ift},x_{i,\ift},\beta),\\
\lim_{j\to+\ift}\vt_{f_j}(u_j;y_j,\beta)&=\lim_{j\to+\ift}\vt(u_j;y_j,\beta)=\vt(u_{\ift},y_{\ift},\beta).
\end{align*}
Estimates in \eqref{Fjjminus1} and the property \eqref{invauift} show that
\be
\vt(u_{\ift};y_{\ift},\beta)=\vt(u_{\ift};x_{0,\ift},\beta)=...=\vt(u_{\ift},x_{k,\ift},\beta)=E_{\ift}.\label{thetaeq2}
\ee
We let $ j\to+\ift $ in \eqref{ujyjEjxi0} and obtain $
\vt(u_{\ift};y_{\ift},\beta)\leq E_{\ift}-\xi_0 $, which is a contradiction to \eqref{thetaeq2}.
\end{proof}

\subsection{Further properties on quantitative stratification} 

In Lemma \ref{SmaHomMEMS}, to show the $ (k,\va) $-symmetry of $ u $ in the ball $ B_s(x) $, we combine the assumptions \eqref{Vinaucon0} and \eqref{Vinaucon} to ensure the approximation of both $ \al $-homogeneity and invariance with respect to a $ k $-dimensional subspace. In this subsection, the main focus is on solutions that only satisfy conditions like \eqref{Vinaucon0}. Precisely, we will establish the alternative lemma as follows.

\begin{lem}\label{kplus1MEMStwoone}
Let $ \ga>0 $, $ k\in\Z\cap[0,n-2] $, $ 0<s\leq 1 $, and $ x\in\R^n $. Assume that $ u\in(C_{\loc}^{0,\al}\cap H_{\loc}^1\cap L_{\loc}^{-p})(B_{20s}(x)) $ is a stationary solution of \eqref{MEMSeq} with respect to $ f\in M_{\loc}^{2\al+n-4+\ga,2}(B_{20s}(x)) $, satisfying
$$
[u]_{C^{0,\al}(\ol{B}_{15s}(x))}+[f]_{M^{2\al+n-4+\ga,2}(B_{15s}(x))}\leq\Lda.
$$
For any $ \va>0 $, there exist $ \delta,\delta'>0 $, depending only on $ \va,\ga,\Lda,n $, and $ p $ such that either
\be
\inf_{V\in\bG(n,k+1)}\(s^{2-2\al-n}\int_{B_s(x)}|V\cdot\na u|^2\)<\delta,\label{useassp11}
\ee
or there exists $ s_x\in[\delta's,s] $ such that $ u $ is $ (k+1,\va) $-symmetric in $ B_{s_x}(x) $.
\end{lem}

\begin{rem}
As discussed in \S\ref{subdiffi}, the result in the above lemma shares remarkable differences with other models, such as harmonic maps. Indeed, for stationary harmonic maps $ \Phi\in H^1(B_2,\cN) $ with $ \cN\subset\R^d $ being a smooth manifold and $ \Theta(\Phi;0,2)\leq\Lda $, where $ \Theta(\Phi;\cdot,\cdot) $ is given by \eqref{Harmonicden}, the proof of such a similar result is quite straightforward. Since $ \Theta(u;0,s)\geq 0 $ for any $ 0\leq s\leq 2 $, it is easy to perform \eqref{harmonoton} with a dyadic decomposition and find some $ s_0\in[\delta',1] $ such that
\be
\Theta(\Phi;0,s_0)-\Theta\(\Phi;0,\f{s_0}{2}\)<-\f{C(\Lda,n,\cN)}{\log\delta'}.\label{dyadics0har}
\ee
This, together with paralleled assumptions like \eqref{useassp11} and similar form of Lemma \ref{SmaHomMEMS}, implies the result. For more details, see the proof of Lemma 32 in \cite{NV18}. However, for our model concerning stationary solutions of \eqref{MEMSeq}, the density $ \vt_f(u;x,s) $ diverges to $ -\ift $ as $ s\to 0^+ $ when $ u(x)>0 $, due to Lemma \ref{propupture}. As a result, it requires necessary modifications.
\end{rem}

\begin{rem}\label{remkplus}
A direct consequence of Lemma \ref{kplus1MEMStwoone} is that if $ u $ is not $ (k+1,\va) $-symmetric in $ B_t(x) $ for any $ t\in[\delta's,s] $, then
$$
\inf_{V\in\bG(n,k+1)}\(s^{2-2\al-n}\int_{B_s(x)}|V\cdot\na u|^2\)\geq\delta.
$$
\end{rem}

Before we give the proof of Lemma \ref{kplus1MEMStwoone}, we first use it to obtain a crucial proposition, which we will apply in the rest of this paper.

\begin{prop}\label{FpropMEMS}
Let $ 0<\beta<\f{1}{2} $, $ \ga>0 $, $ k\in\Z\cap[0,n-2] $, $ 0<s\leq 1 $, and $ x\in B_2 $. Assume that $ u\in(C_{\loc}^{0,\al}\cap H_{\loc}^1\cap L_{\loc}^{-p})(B_{4R_0}) $ is a stationary solution of \eqref{MEMSeq} with respect to $ f\in M_{\loc}^{2\al+n-4+\ga,2}(B_{4R_0}) $, satisfying
\be
[u]_{C^{0,\al}(\ol{B}_{2R_0})}+[f]_{M^{2\al+n-4+\ga,2}(B_{2R_0})}\leq\Lda.\label{FpropMEMSuholder}
\ee
For any $ \va>0 $, there exist $ \delta,\delta'>0 $, depending only on $ \beta,\va,\ga,\Lda,n $, and $ p $ such that if the set
$$
F=\{y\in B_{2s}(x):\vt_f(u;y,s)-\vt_f(u;y,\beta s)<\delta\}
$$
$ \beta s $-effectively spans $ L\in\bA(n,k) $, then $ S_{\va,\delta's}^k(u)\cap B_s(x)\subset B_{2\beta s}(L) $.
\end{prop}

\begin{proof}
Assume that $ \{x_i\}_{i=0}^k\subset F $ are $ \beta s $-independent, spanning an affine subspace $ L=x_0+\op{span}\{x_i-x_0\}_{i=1}^k $ such that
\be
\sup_{0\leq i\leq k}(\vt_f(u;x_i,s)-\vt_f(u;x_i,\beta s))<\delta.\label{uxidelMEMS}
\ee
Fix 
\be 
y_0\in B_s(x)\backslash B_{2\beta s}(L).\label{y0contain}
\ee
We will choose appropriate $ 0<\delta,\delta'<1 $ such that $ y_0\notin S_{\va,\delta's}^k(u) $. By \eqref{y0contain}, there exists $ 0<\sg<\min\{\f{1}{1000},\f{\beta}{2}\} $, satisfying 
\be
B_{\sg s}(y_0)\subset\(\bigcap_{i=0}^kB_{4s}(x_i)\)\cap(B_{2s}(x)\backslash B_{\beta s}(L)).\label{Btauscontain}
\ee
Since $ 0<\beta<\f{1}{2} $, by \eqref{uxidelMEMS}, Proposition \ref{MonFor} and Corollary \ref{coruse} yield that
\be
\begin{aligned}
\int_{B_{4s}(x_i)}|(y-x_i)\cdot\na u-\al u|^2\ud y&\leq C\[\vt_f(u;x_i,s)-\vt_f\(u;x_i,\f{s}{2}\)\]s^{2\al+n}\\
&\leq C(\vt_f(u;x_i,s)-\vt_f(u;x_i,\beta s))s^{2\al+n}\\
&\leq C(n,p)\delta s^{2\al+n}
\end{aligned}\label{BgasMEMS0}
\ee
for any $ i\in\Z\cap[0,k] $. In particular,
$$
\int_{B_{4s}(x_0)}|(y-x_0)\cdot\na u-\al u|^2\leq C(n,p)s^{2\al+n}.
$$
Given \eqref{FpropMEMSuholder}, it follows from Lemma \ref{InESLem1}, \ref{InESLem2}, and average arguments that there is $ x'\in B_{4s}(x_0) $ such that
$$
0\leq u(x')\leq C(\beta,\Lda,n,p)s^{\al}.
$$
Moreover, by \eqref{FpropMEMSuholder} and \eqref{Btauscontain}, we obtain
\be 
0<\sup_{B_{\sg s}(y_0)}u\leq C(\beta,\Lda,n,p)s^{\al}.\label{uy0upper}
\ee
According to \eqref{Btauscontain} and \eqref{BgasMEMS0}, there holds
\be
\int_{B_{\sg s}(y_0)}|(y-x_i)\cdot\na u-\al u|^2\ud y\leq C(n,p)\delta s^{2\al+n}\label{BgasMEMS}
\ee
for any $ i\in\Z\cap[0,k] $, and then
\be
\sup_{1\leq i\leq k}\(\int_{B_{\sg s}(y_0)}|(x_i-x_0)\cdot\na u|^2\)\leq C(n,p)s^{2\al+n}.\label{xix0MEMS}
\ee
Consequently,
\be
\int_{B_{\sg s}(y_0)}|V\cdot\na u|^2\leq C(\beta,n,p)\delta s^{2\al+n-2},\label{vhatMEMS}
\ee
where $ V=\op{span}\{x_i-x_0\}_{i=1}^k $. For any $ y\in B_{\sg s}(y_0) $, let 
$$
\pi_L(y)=x_0+\sum_{i=1}^k\al_i(y)(x_i-x_0)\in L
$$
be the point such that 
\be
|\pi_L(y)-y|=\dist(y,L)\geq\beta s.\label{inebetas}
\ee
Here, for the inequality of \eqref{inebetas}, we have used \eqref{Btauscontain}. Applying Lemma \ref{rhoind}, if $ y\in B_{\sg s}(y_0) $, then $ \sup_{1\leq i\leq k}|\al_i(y)|\leq C(\beta,n) $. Now, we can deduce from \eqref{BgasMEMS} and \eqref{xix0MEMS} that
\be
\int_{B_{\sg s}(y_0)}|(y-\pi_L(y))\cdot\na u-\al u|^2\ud y\leq C(\beta,n,p)\delta s^{2\al+n}.\label{Bgasy0MEMS}
\ee
The definition of $ \pi_L $ implies that
\be
|\pi_L(y)-\pi_L(y_0)|\leq|y-y_0|\label{yy0}
\ee
for any $ y\in\R^n $. As a result,
\begin{align*}
\int_{B_{\sg s}(y_0)}|(y_0-\pi_L(y_0))\cdot\na u|^2&\leq C\int_{B_{\sg s}(y_0)}|(y-\pi_L(y))\cdot\na u-\al u|^2\ud y\\
&\quad\quad+C\int_{B_{\sg s}(y_0)}|((y-y_0)-(\pi_L(y)-\pi_L(y_0))\cdot\na u-\al u|^2\ud y\\
&\leq C\delta s^{2\al+n}+C\sg^2s^2\int_{B_{\sg s}(y_0)}|\na u|^2+C\int_{B_{\sg s}(y_0)}|u|^2\\
&\leq C(\beta,\ga,\Lda,n,p)(\delta s^{2\al+n}+(\sg s)^{2\al+n}+(\sg s)^ns^{2\al}),
\end{align*}
where for the second inequality, we have used \eqref{Bgasy0MEMS} and \eqref{yy0}, for the third inequality, we have used \eqref{FpropMEMSuholder}, \eqref{uy0upper}, and Lemma \ref{InESLem2}. Combining with \eqref{vhatMEMS}, we can see that
$$
(\sg s)^{2-2\al-n}\int_{B_{\sg s}(y_0)}|V'\cdot\na u|^2\leq C(\beta,\ga,\Lda,n,p)(\sg^2+\delta\sg^{2-2\al-n}+\sg^{2-2\al}),
$$
where
$$
V'=V\oplus\op{span}\left\{\f{y_0-\pi_L(y_0)}{|y_0-\pi_L(y_0)|}\right\}.
$$
Choosing $ \sg=\sg(\beta,\va,\Lda,n,p)>0 $ and $ \delta=\delta(\beta,\va,\ga,\Lda,n,p)>0 $ sufficiently small, we apply Lemma \ref{kplus1MEMStwoone} to get $ \delta'=\delta'(\beta,\va,\ga,\Lda,n,p)>0 $ such that $ u $ is $ (k+1,\va) $-symmetric in $ B_{s_0}(y_0) $ with some $ s_0\in[\delta's,s] $. Thus, we complete the proof.
\end{proof}

We now turn to show Lemma \ref{kplus1MEMStwoone}. Since there is no similar estimate like \eqref{dyadics0har}, we first consider points with a small value of $ u $ and present the following lemma.

\begin{lem}\label{rxlogrMEMS}
Let $ \ga>0 $, $ 0<\sg<\f{1}{1000} $, $ 0<s\leq 1 $, and $ x\in\R^n $. Assume that $ u\in (C_{\loc}^{0,\al}\cap H_{\loc}^1\cap L_{\loc}^{-p})(B_{20s}(x)) $ is a stationary solution of \eqref{MEMSeq} with respect to $ f\in M_{\loc}^{2\al+n-4+\ga,2}(B_{20s}(x)) $, satisfying 
\begin{gather}
0\leq u(x)<(\sg s)^{\al}\label{sgsuse}\\
[u]_{C^{0,\al}(\ol{B}_{15s}(x))}+[f]_{M^{2\al+n-4+\ga,2}(B_{15s}(x))}\leq\Lda.\label{rxlogrMEMSass}
\end{gather}
There exists $ s_x\in[\sg s,s] $ such that
\be
\vt_f(u;x,s_x)-\vt_f\(u;x,\f{s_x}{2}\)\leq-\f{C}{\log\sg},\label{vtrys}
\ee
where $ C>0 $ depends only on $ \ga,\Lda,n $, and $ p $. 
\end{lem}

\begin{proof}
We claim that if $ t\in[\sg s,s] $, then
\be
|\vt_f(u;x,t)|\leq C(\ga,\Lda,n,p).\label{claimvtf}
\ee
Analogous to the calculations in \eqref{u10rhobales}, we obtain from \eqref{sgsuse} that
$$
t^{-2\al-n}\int_{B_{10t}(x)}u^2\leq C(t^{-2\al}u(x)^2+1)\leq C(t^{-2\al}(\sg s)^{2\al}+1)\leq C(\Lda,n,p).
$$
Given \eqref{thetaftheta} and \eqref{rxlogrMEMSass}, the claim \eqref{claimvtf} follows directly from Lemma \ref{InESLem1} and \ref{InESLem2}. Choosing $ \ell\in\Z_+ $ such that $ 2^{-\ell-1}\leq\sg<2^{-\ell} $ with $ \ell\sim-\log\sg $, we can apply \eqref{claimvtf} and Proposition \ref{MonFor} to deduce that
\begin{align*}
0&\leq\sum_{i=0}^{\ell-1}\[\vt_f\(u;x,\f{s}{2^i}\)-\vt_f\(u;x,\f{s}{2^{i+1}}\)\]\\
&\leq\vt_f(u;x,s)-\vt_f\(u;x,\f{s}{2^{\ell}}\)\\
&\leq\vt_f(u;x,s)-\vt_f(u;x,\sg s)\\
&\leq C(\ga,\Lda,n,p).
\end{align*}
Thus, there must be some $ i_0\in\Z\cap[0,\ell-1] $ such that
$$
\vt_f\(u;x,\f{s}{2^{i_0}}\)-\vt_f\(u;x,\f{s}{2^{i_0+1}}\)\leq\f{C}{\ell}\leq-\f{C(\ga,\Lda,n,p)}{\log\sg}.
$$
Letting $ s_x=\f{s}{2^{i_0}}\in[\sg s,s] $, we complete the proof.
\end{proof}

Using Lemma \ref{rxlogrMEMS}, we have the result below. 

\begin{lem}\label{kplus1MEMS}
Let $ \ga>0 $, $ k\in\Z\cap[0,n-2] $, $ 0<s\leq 1 $, and $ x\in\R^n $. Assume that $ u\in(C_{\loc}^{0,\al}\cap H_{\loc}^1\cap L_{\loc}^{-p})(B_{20s}(x)) $ is a stationary solution of \eqref{MEMSeq} with respect to $ f\in M_{\loc}^{2\al+n-4+\ga,2}(B_{20s}(x)) $, satisfying
$$
[u]_{C^{0,\al}(\ol{B}_{15s}(x))}+[f]_{M^{2\al+n-4+\ga,2}(B_{15s}(x))}\leq\Lda.
$$
For any $ \va>0 $, there exists $ \delta>0 $, depending only on $ \va,\ga,\Lda,n $, and $ p $ such that if
\be
\inf_{V\in\bG(n,k+1)}\(s^{2-2\al-n}\int_{B_s(x)}|V\cdot\na u|^2\)<\delta,\label{useassp}
\ee
then either
\be
u(x)\geq\delta^{\f{\al}{2(n+2\al-2)}}s^{\al},\label{ussalusecor}
\ee
or there is $ s_x\in[\delta^{\f{1}{2(n+2\al-2)}}s,s] $ such that $ u $ is $ (k+1,\va) $-symmetric in $ B_{s_x}(x) $.
\end{lem}

\begin{proof}
Given that \eqref{ussalusecor} is false, we obtain from Lemma \ref{rxlogrMEMS} that there exists $ s_x\in[\delta^{\f{1}{2(n+2\al-2)}}s,s] $ such that
$$
\vt_f(u;x,s_x)-\vt_f\(u;x,\f{s_x}{2}\)\leq-\f{C(\ga,\Lda,n,p)}{\log\delta}.
$$
By the assumption \eqref{useassp}, we have
\begin{align*}
\inf_{V\in\bG(n,k+1)}\(s_x^{2-2\al-n}\int_{B_{s_x}(x)}|V\cdot\na u|^2\)&\leq\inf_{V\in\bG(n,k+1)}\(s_x^{2-2\al-n}\int_{B_s(x)}|V\cdot\na u|^2\)\\
&\leq C(\ga,\Lda,n,p)\delta^{\f{1}{2}},
\end{align*}
where for the last inequality, we have used $ s_x\geq\delta^{\f{1}{2(n+2\al-2)}}s $. Applying Proposition \ref{SmaHomMEMS}, we can choose $ \delta=\delta(\va,\ga,\Lda,n,p)>0 $ sufficiently small such that $ u $ is $ (k+1,\va) $-symmetric in $ B_{s_x}(x) $ and the result follows directly. 
\end{proof}

Indeed, for $ k=n-2 $, the results in lemma \ref{kplus1MEMS} can be improved with a simple application of compactness arguments.

\begin{lem}\label{kplus1MEMSn2}
Let $ \ga>0 $, $ 0<s\leq 1 $, and $ x\in\R^n $. Assume that $ u\in(C_{\loc}^{0,\al}\cap H_{\loc}^1\cap L_{\loc}^{-p})(B_{4s}(x)) $ is a stationary solution of \eqref{MEMSeq} with respect to $ f\in M_{\loc}^{2\al+n-4+\ga,2}(B_{4s}(x)) $, satisfying
$$
[u]_{C^{0,\al}(\ol{B}_{2s}(x))}+[f]_{M^{2\al+n-4+\ga,2}(B_{2s}(x))}\leq\Lda.
$$
There exists $ \delta>0 $, depending only on $ \ga,\Lda,n $, and $ p $ such that if
$$
\inf_{V\in\bG(n,n-1)}\(s^{2-2\al-n}\int_{B_s(x)}|V\cdot\na u|^2\)<\delta,
$$
then
\be
\inf_{B_{\f{s}{2}}(x)}u\geq\delta s^{\al}.\label{ussalusecor2}
\ee
If we further assume that $ f\equiv 0 $, then the conclusion \eqref{ussalusecor2} can be improved to 
$$
\inf_{B_{\f{s}{2}}(x)}u(x)\geq Cs^{\al},
$$
where $ C>0 $ depends only on $ \Lda,n $, and $ p $.
\end{lem}
\begin{proof}
By using Proposition \ref{ScalingProp} and Remark \ref{scalerem}, we assume that $ s=1 $ and $ x=0 $. Suppose that the result does not hold. There exists a sequence of stationary solutions $ \{u_i\}\subset (C_{\loc}^{0,\al}\cap H_{\loc}^1\cap L_{\loc}^{-p})(B_4) $ of \eqref{MEMSeq} with respect to $ \{f_i\}\subset M_{\loc}^{2\al+n-4+\ga,2}(B_4) $ such that for any $ i\in\Z_+ $, the following properties hold.
\begin{itemize}
\item $ u_i $ and $ f_i $ satisfy  
\begin{gather}
[u_i]_{C^{0,\al}(\ol{B}_2)}+[f_i]_{M^{2\al+n-4+\ga,2}(B_2)}\leq\Lda,\label{uifiass121}\\
0\leq\inf_{B_\f{1}{2}}u_i<i^{-1}.\label{uifiass12}
\end{gather}
\item There exists $ V_i\in\bG(n,n-1) $ such that
\be
\int_{B_1}|V_i\cdot\na u_i|^2<i^{-1}.\label{uiViesi1}
\ee
\end{itemize}
Using \eqref{uifiass121} and \eqref{uifiass12}, we obtain 
$$
\sup_{i\in\Z_+}\|u_i\|_{L^2(B_2)}\leq C(\ga,\Lda,n,p).
$$
It follows from Proposition \ref{propConv} that up to a subsequence, $ V_i\to V_{\ift} $, and
\be
\begin{aligned}
u_i&\to u_{\ift}\text{ strongly in }(H_{\loc}^1\cap L^{\ift})(B_2),\\
f_i&\to f_{\ift}\text{ weakly in }L^2(B_2),
\end{aligned}\label{assuifiuift}
\ee
where $ u_{\ift}\in C^{0,\al}(\ol{B}_2)\cap (H_{\loc}^1\cap L^{-p})(B_2) $. By Lemma \ref{MorreyL2}, $ f\in M^{2\al+n-4+\ga,2}(B_2) $. Given \eqref{uiViesi1}, we deduce that $ u_{\ift} $ is invariant with respect to $ V_{\ift} $. Moreover, $ u_{\ift} $ is a stationary solution of \eqref{MEMSeq} with respect to $ f_{\ift} $. On the other hand, $ \inf_{B_{\f{1}{2}}}u_{\ift}=0 $, due to \eqref{uifiass12} and \eqref{assuifiuift}. Consequently,
\be
\dim_{\HH}(\{u_{\ift}=0\}\cap B_1)\geq n-1,\label{n1up}
\ee
which is contradictory to Proposition \ref{Hausdorffdim}. For the case that $ f\equiv 0 $, assuming that the result is false, we change the assumption \eqref{uifiass12} to $ 0\leq\inf_{B_{\f{1}{2}}}u_i<C' $, where $ C'>0 $ is to be determined. With the help of almost the same arguments above, we can still obtain convergence results in \eqref{assuifiuift}. Additionally,  
\be
0\leq\inf_{B_{\f{1}{2}}}u_{\ift}\leq C'.\label{uiftgeq2}
\ee
Thus, $ u_{\ift} $ is invariant with respect to $ V_{\ift} $. This yields that $ u_{\ift}>0 $ in $ B_1 $, since if not, we have \eqref{n1up}, a contradiction. In particular, $ u_{\ift} $ is a positive and convex solution of $ \Delta u_{\ift}=u_{\ift}^{-p} $. Using Lemma \ref{proplowerbound}, there holds $
\inf_{B_{\f{1}{2}}}u\geq C(\Lda,n,p)>0 $. If we choose sufficiently small $ C'=C'(\Lda,n,p)>0 $, it is contradictory to \eqref{uiftgeq2}.
\end{proof}

Next, let us turn to the points where $ u $ has given lower bounds. Such a result is the consequence of standard regularity estimates for elliptic equations.

\begin{lem}\label{lowernva}
Let $ \ga>0 $, $ k\in\Z\cap[0,n-2] $, $ 0<s\leq 1 $, and $ x\in\R^n $. Assume that $ u\in(C_{\loc}^{0,\al}\cap H_{\loc}^1\cap L_{\loc}^{-p})(B_{2s}(x)) $ is a stationary solution of \eqref{MEMSeq} with respect to $ f\in M_{\loc}^{2\al+n-4+\ga,2}(B_{2s}(x)) $, satisfying
\be
[u]_{C^{0,\al}(\ol{B}_{2s}(x))}+[f]_{M^{2\al+n-4+\ga,2}(B_{2s}(x))}\leq\Lda.\label{assuhol}
\ee
For any $ \va,\sg>0 $, there exists $ \delta>0 $, depending only on $ \va,\ga,\Lda,n,p $, and $ \sg $ such that if 
\be
u(x)\geq(\sg s)^{\al},\label{uxsgsal}
\ee
then $ u $ is $ (n,\va) $-symmetric in $ B_{\delta s}(x) $. 
\end{lem}
\begin{proof}
By Proposition \ref{ScalingProp} and Remark \ref{scalerem}, we let $ s=1 $ and $ x=0 $. Using \eqref{assuhol}, we can choose sufficiently small $ \sg'=\sg'(\Lda,n,p,\sg)\in(0,1) $ such that $
\inf_{B_{\sg'}}u\geq C(\Lda,n,p,\sg)^{-1} $. Given Lemma \ref{Inclusion1}, without loss of generality, we assume that $ 0<2\al+\ga<2 $. Indeed, since $ 2\al=\f{4}{p+1}<2 $, there exists $ 0<\ga'<\ga $ such that $ 2\al+\ga'<2 $. Lemma \ref{Inclusion1} implies
$$
[f]_{M^{2\al+n-4+\ga',2}(B_2)}\leq C(\ga,n,p)[f]_{M^{2\al+n-4+\ga,2}(B_2)}\leq C(\ga,\Lda,n,p).
$$
Consequently, we can conduct our arguments to $ \ga' $, and all results will not change. Let
$ g(u,f):=u^{-p}+f $. Inequalities \eqref{assuhol} and \eqref{uxsgsal} yield that
$$
[g(u,f)]_{M^{2\al+n-4+\ga,2}(B_{2\sg'})}\leq C(\ga,\Lda,n,p,\sg).
$$
Applying \eqref{assuhol} again with the property $ 0<\sg'<\f{1}{2} $, it follows from Lemma \ref{InESLem2} and \ref{Liouvillecla} that
\begin{align*}
[u]_{C^{0,\al+\f{\ga}{2}}\(B_{\sg'}\)}&\leq C\[\((\sg')^{2-2\al-n-\ga}\int_{B_{2\sg'}}|\na u|^2\)^{\f{1}{2}}+[g(u,f)]_{M^{2\al+n-4+\ga,2}(B_{2\sg'})}\]\\
&\leq C(1+(\sg')^{-\f{\ga}{2}})\leq C(\ga,\Lda,n,p,\sg)(\sg')^{-\f{\ga}{2}}.
\end{align*}
Consequently,
\begin{align*}
\|T_{0,\sg''\sg'}(u-u(0))\|_{L^{\ift}(B_1)}\leq C(\sg''\sg')^{-\al}(\sg''\sg')^{\al+\f{\ga}{2}}[u]_{C^{0,\al+\f{\ga}{2}}(B_{\sg'})}\leq C(\ga,\Lda,n,p,\sg)(\sg'')^{\f{\ga}{2}}
\end{align*}
for any $ 0<\sg''<1 $. We now choose sufficiently small $ \sg''=\sg''(\va,\ga,\Lda,n,p,\sg)>0 $ such that
$$
\|T_{0,\sg''\sg'}(u-u(0))\|_{L^{\ift}(B_1)}<\va.
$$
Thus, $ u $ is $ (n,\va) $-symmetric in $ B_{\sg''\sg'}(x) $, and $ \delta=\sg''\sg'>0 $ is what we want.
\end{proof}

\begin{proof}[Proof of Lemma \ref{kplus1MEMStwoone}]
Using Lemma \ref{kplus1MEMS}, we choose $ \delta=\delta(\va,\ga,\Lda,n,p)>0 $ sufficiently small such that either $
u(x)\geq\delta^{\f{\al}{2(n+2\al-2)}}s^{\al} $ or there exists $ s_x\in[\delta^{\f{1}{2(n+2\al-2)}}s,s] $ such that $ u $ is $ (k+1,\va) $-symmetric in $ B_{s_x}(x) $. We apply Lemma \ref{lowernva} and obtain $ \tau=\tau(\va,\ga,\Lda,n,p)>0 $ such that $ u $ is $ (n,\va) $-symmetric in $ B_{\tau s_x}(x) $. Letting $ \delta'=\min\{\tau\delta^{\f{1}{2(n+2\al-2)}},1\} $, we can complete the proof.
\end{proof}

\section{Reifenberg-type theorems} \label{Reifenbergtype}

In this section, we will present Reifenberg-type results, which serve as powerful tools for solving various problems related to geometric measure theory. The foundational concepts were first introduced by Reifenberg in \cite{Rei60}. Interested readers can refer to the lecture notes \cite{Nab20} by Naber for a more comprehensive overview of these techniques. Our focus here will be on the specific Reifenberg-type results that are particularly relevant to our subsequent proofs.

\begin{defn}\label{displacementk}
Let $ k\in\Z\cap[0,n] $, $ 0<r\leq 1 $, and $ \om\subset\R^n $ is a bounded open set. Assume that $ \mu $ is a finite Radon measure on $ \om $, namely, $ \mu(\om)<+\ift $. For $ x\in \om $ and $ 0<r<\dist(x,\om) $, we define the $ k $-dimensional displacement
$$
D_{\mu}^k(x,r):=\min_{L\in\bA(n,k)}\(r^{-k-2}\int_{B_r(x)}\dist^2(y,L)\ud\mu(y)\).
$$
\end{defn}

The theorem below addresses the Reifenberg-type estimates for discrete Radon measures, which one can understand as the summation of a finite number of Dirac measures, each associated with different weights. Intuitively, this theorem posits that if the $ k $-displacement $ D_{\mu}^k(\cdot,\cdot) $ is sufficiently small in some sense, then the measure $ \mu $ is $ k $-Ahlfors regular. Such estimates are crucial for understanding the fine structure of measures in geometric measure theory, as they provide insights into the local behavior and regularity of the underlying sets associated with the measure.

\begin{thm}[\cite{NV17}, Theorem 3.4]\label{Rei1}
Let $ k\in\Z\cap[0,n] $, $ 0<r\leq 1 $, and $ x_0\in\R^n $. Assume that $ \{B_{r_y}(y)\}_{y\in\cD}\subset B_{2r}(x_0) $ is a collection of pairwise disjoint balls with $ \cD\subset B_r(x_0) $ and 
$$
\mu:=\sum_{y\in\cD}\w_kr_y^k\delta_y.
$$
There exist $ \delta_{\op{R}}>0 $ and $ C_{\op{R}}>0 $, depending only on $ n $ such that if
\be
\int_{B_t(x)}\(\int_0^tD_{\mu}^k(y,s)\f{\ud s}{s}\)\ud\mu(y)<\delta_{\op{R}}t^k\label{Reicon}
\ee
for any $ B_t(x)\subset B_{2r}(x_0) $ with $ t>0 $, then
\be
\mu(B_r(x_0))=\sum_{y\in\cD}\w_kr_y^k\leq C_{\op{R}}r^k.\label{contentmuBr}
\ee
\end{thm}

\begin{rem}
Here, the essential point of Theorem \ref{Rei1} is that the number $ C_{\op{R}}>0 $ only depends on $ n $. Indeed, further results in \cite{Mis18} imply that if the right-hand side of \eqref{Reicon} is replaced by $ C_0t^k $ for some fixed constant $ C_0>0 $, then in \eqref{contentmuBr}, $ C_{\op{R}} $ depends on $ C_0 $ and $ n $. Such a result is not applicable in the proceeding reasoning since under iterations $ C_{\op{R}} $ may increase. 
\end{rem}

Indeed, for later use, we will apply the following variant of the above theorem.

\begin{cor}\label{Rei1cor}
The result of Theorem \ref{Rei1} is still true if the assumption \eqref{Reicon} is valid for any $ x\in B_r(x_0) $ and $ 0<t<\f{r}{10} $.
\end{cor}
\begin{proof}
We show that there exists $ \delta=\delta(n)>0 $ such that if for any $ x\in B_r(x_0) $ and $ 0<t<\f{r}{10} $, there holds
\be
\int_{B_t(x)}\(\int_0^tD_{\mu}^k(y,s)\f{\ud s}{s}\)\ud\mu(y)<\delta t^k,\label{Reicon2}
\ee
then 
\be
\mu(B_r(x_0))\leq C(n)r^k.\label{muBrxCn}
\ee
Choose a covering of $ B_r(x_0) $, given by $ \{B_{\f{r}{100}}(x_i)\}_{i=1}^N $ such that $ \{x_i\}_{i=1}^N\subset B_r(x_0) $ and balls in $ \{B_{\f{r}{200}}(x_i)\}_{i=1}^N $ are disjoint. For any $ i\in\Z\cap[1,N] $, we define 
$$
\cD_i:=\cD\cap B_{\f{r}{100}}(x_i)\subset B_{\f{r}{100}}(x_i) 
$$
and choose $ \beta=\beta(n)\in(0,1) $ such that
$$
\bigcup_{y\in\cD_i}B_{\beta r_y}(y)\subset B_{\f{r}{50}}(x_i).
$$
Let
$$
\mu_i:=\sum_{y\in\cD_i}\w_k(\beta r_y)^k\delta_y.
$$
Fix $ i\in\Z\cap[1,N] $. For any $ B_t(x)\subset B_{\f{r}{50}}(x_i) $, we claim that if $ \delta=\delta(n)>0 $ is chosen sufficiently small, then
\be
\int_{B_t(x)}\(\int_0^tD_{\mu_i}^k(y,s)\f{\ud s}{s}\)\ud\mu_i(y)<\delta_{\op{R}}t^k.\label{claimRei}
\ee
Suppose that we have shown this claim. Theorem \ref{Rei1} implies that
$$
\mu_i(B_{\f{r}{100}}(x_i))=\sum_{y\in\cD_i}\w_k(\beta r_y)^k\leq C_{\op{R}}(n)\(\f{r}{100}\)^k,
$$
and then 
$$
\mu(B_r(x_0))\leq C\(\sum_{i=1}^N\mu_i\(B_{\f{r}{100}}(x_i)\)\)\leq C(n)C_{\op{R}}(n)\(\f{r}{100}\)^k\leq C(n)r^k.
$$
Thus, we have \eqref{muBrxCn}. Let us turn to the proof of \eqref{claimRei}. If $ B_t(x)\cap B_r(x_0)=\emptyset $, by the definition of $ \mu_i $, there is nothing to prove. Assume that there exists $ x'\in B_t(x)\cap B_r(x_0) $. Consequently, $
B_t(x)\subset B_{2t}(x') $ and $ 2t<\f{r}{25} $. Now \eqref{Reicon2} gives 
$$
\int_{B_t(x)}\(\int_0^tD_{\mu_i}^k(y,s)\f{\ud s}{s}\)\ud\mu_i(y)\leq\int_{B_{2t}(x')}\(\int_0^{2t}D_{\mu}^k(y,s)\f{\ud s}{s}\)\ud\mu(y)\leq\delta(2t)^k.
$$
Choosing $ \delta=\delta(n)>0 $ sufficiently small such that $ 2^n\delta<\delta_{\op{R}} $, we obtain \eqref{claimRei}.
\end{proof}

The theorem below characterizes the rectifiability of sets based on the displacements defined in Definition \ref{displacementk}.

\begin{thm}[\cite{AT15}, Corollary 1.3]\label{Rei2}
Let $ S\subset\R^n $ be a $ \HH^k $-measurable set. $ S $ is rectifiable if and only if for $ \HH^k $-a.e. $ x\in S $,
$$
\int_0^1D_{\HH^k\llcorner S}^k(x,s)\f{\ud s}{s}<+\ift.
$$
\end{thm}

\begin{rem}
Theorem 3.3 of \cite{NV17} gives a more subtle improvement of this theorem. However, Theorem \ref{Rei2} is already enough in our later proofs.
\end{rem}

\section{\texorpdfstring{$ L^2 $}{}-best approximation estimates}\label{L2best}

\subsection{Introduction and results}

To apply theorems in \S\ref{Reifenbergtype}, we need to establish the connections between the density \eqref{thetafuxr} and the displacements given in Definition \ref{displacementk}.

\begin{thm}\label{beta2MEMS}
Let $ \ga>0 $, $ k\in\Z\cap[0,n-2] $, $ 0<s\leq 1 $, and $ x\in B_{R_0} $. Assume that $ u\in(C_{\loc}^{0,\al}\cap H_{\loc}^1\cap L_{\loc}^{-p})(B_{4R_0}) $ is a stationary solution of \eqref{MEMSeq} with respect to $ f\in M_{\loc}^{2\al+n-4+\ga,2}(B_{4R_0}) $. There exist $ C>0 $, depending only on $ n $, and $ p $ such that if there is $ \tau>0 $, satisfying
\be
\inf_{V\in\bG(n,k+1)}\(s^{2-2\al-n}\int_{B_{5s}(x)}|V\cdot\na u|^2\)>\tau,\label{lowerV}
\ee
then for any $ \mu\in\M(B_s(x)) $ with $ \mu(B_s(x))<+\ift $, there holds
$$
D_{\mu}^k(x,s)\leq C\tau^{-1}s^{-k}\int_{B_s(x)}W_f(u;y,s)\ud\mu(y),
$$
where $ D_{\mu}^k(\cdot,\cdot) $ is the $ k $-dimensional displacement given in Definition \ref{displacementk} and for any $ y\in B_{2s}(x) $,
$$
W_f(u;y,s):=\vt_f(u;y,2s)-\vt_f(u;y,s).
$$
\end{thm}

We first give the following direct consequence of this theorem.

\begin{cor}\label{beta22MEMS}
Let $ \ga>0 $, $ k\in\Z\cap[0,n-2] $, $ 0<s\leq 1 $, and $ x\in B_{R_0} $. Assume that $ u\in(C_{\loc}^{0,\al}\cap H_{\loc}^1\cap L_{\loc}^{-p})(B_{4R_0}) $ is a stationary solution of \eqref{MEMSeq} with respect to $ f\in M_{\loc}^{2\al+n-4+\ga,2}(B_{4R_0}) $, satisfying
$$
[u]_{C^{0,\al}(\ol{B}_{2R_0})}+[f]_{M^{2\al+n-4+\ga,2}(B_{2R_0})}\leq\Lda.
$$
For any $ \va>0 $, there exist $ \delta,C>0 $, depending only on $ \va,\ga,\Lda,n $, and $ p $ such that if $ u $ is $ (0,\delta) $-symmetric but not $ (k+1,\va) $-symmetric in $ B_{5s}(x) $, then for any $ \mu\in\M(B_s(x)) $ with $ \mu(B_s(x))<+\ift $,
$$
D_{\mu}^k(x,s)\leq Cs^{-k}\int_{B_s(x)}W_f(u;y,s)\ud\mu(y).
$$
\end{cor}

Indeed, the corollary above is a direct consequence of Theorem \ref{beta2MEMS} and the lemma below.

\begin{lem}\label{geqlemMEMS}
Let $ \ga>0 $, $ k\in\Z\cap[0,n-2] $, $ 0<s\leq 1 $, and $ x\in\R^n $. Assume that $ u\in (C_{\loc}^{0,\al}\cap H_{\loc}^1\cap L_{\loc}^{-p})(B_{15s}(x)) $ is a stationary solution of \eqref{MEMSeq} with respect to $ f\in M_{\loc}^{2\al+n-4+\ga,2}(B_{15s}(x)) $, satisfying $ u(x)=0 $ and
$$
[u]_{C^{0,\al}(\ol{B}_{10s}(x))}+[f]_{M^{2\al+n-4+\ga,2}(B_{10s}(x))}\leq\Lda. 
$$
For any $ \va>0 $, there exists $ \delta>0 $, depending only on $ \va,\ga,\Lda,n $, and $ p $ such that if $ u $ is $ (0,\delta) $-symmetric in $ B_{5s}(x) $ but not $ (k+1,\va) $-symmetric, then
$$
\inf_{V\in\bG(n,k+1)}\(s^{2-2\al-n}\int_{B_{5s}(x)}|V\cdot\na u|^2\)>\delta.
$$
\end{lem}

\begin{proof}
Using Proposition \ref{ScalingProp} and Remark \ref{scalerem}, we assume that $ s=1 $ and $ x=0 $. If the result is not true, then there exist $ \va_0>0 $, a sequence of stationary solutions of \eqref{MEMSeq}, denoted by $ \{u_i\}\subset (C_{\loc}^{0,\al}\cap H_{\loc}^1\cap L_{\loc}^{-p})(B_{15}) $ with respect to $ \{f_i\}\subset M_{\loc}^{2\al+n-4+\ga,2}(B_{15}) $, and $ \{V_i\}\subset\bG(n,k+1) $ such that for any $ i\in\Z_+ $, the following properties hold.
\begin{itemize}
\item $ u_i $ and $ f_i $ satisfy
\be
[u_i]_{C^{0,\al}(\ol{B}_{10})}+[f_i]_{M^{2\al+n-4+\ga,2}(B_{10})}\leq\Lda.\label{uifiass4}
\ee
\item $ u_i $ is $ (0,i^{-1}) $-symmetric but not $ (k+1,\va_0) $-symmetric in $ B_5 $. In particular, there exists a $ 0 $-symmetric function $ h_i\in C_{\loc}^{0,\al}(\R^n) $ such that 
\be
\|(u_i-u_i(0))-h_i\|_{L^{\ift}(B_5)}<5^{\al}i^{-1}.\label{uimiimiuspMEMS}
\ee
\item $ u_i $ satisfies the inequality
\be
\int_{B_5}|V_i\cdot\na 
u_i|^2<i^{-1}.\label{B4B3deltaMEMS}
\ee
\end{itemize}
By \eqref{uifiass4} and \eqref{uimiimiuspMEMS}, it follows from Proposition \ref{propConv} and \ref{propConv2} that there exist $ V_{\ift}\in\bG(n,k+1) $ and $ u_{\ift}\in C^{0,\al}(\ol{B}_{10})\cap H_{\loc}^1(B_{10}) $ such that up to a subsequence,
\begin{gather*}
u_i-u_i(0)\to u_{\ift}\text{ strongly in }(H_{\loc}^1\cap L^{\ift})(B_{10}),\\
h_i\to u_{\ift}\text{ strongly in }L^{\ift}(B_5).
\end{gather*}
Using Lemma \ref{ConSym}, we see that $ u_{\ift} $ is $ 0 $-symmetric. The inequality \eqref{B4B3deltaMEMS} yields that $ u_{\ift} $ is invariant with respect to $ V_{\ift} $ in $ B_5 $. Thus, $ u_{\ift} $ is $ (k+1) $-symmetric with respect to $ V_{\ift} $. For sufficiently large $ i\in\Z_+ $, it follows from \eqref{uimiimiuspMEMS} that
$$
\|T_{0,5}(u_i-u_i(0))-u_{\ift}\|_{L^{\ift}(B_1)}<\va_0.
$$
It is a contradiction to the assumption that $ u_i $ is not $ (k+1,\va_0) $-symmetric in $ B_5 $.
\end{proof}

\begin{proof}[Proof of Corollary \ref{beta22MEMS}]
Applying Lemma \ref{geqlemMEMS}, we choose sufficiently small $ \delta=\delta(\va,\ga,\Lda,n,p)>0 $ such that the condition \eqref{lowerV} is satisfied. Theorem \ref{beta2MEMS} directly implies the result.
\end{proof}

\subsection{Proof of Theorem \ref{beta2MEMS}} We first need some basic ingredients, which lead to an explicit representation of the $ k $-dimensional displacements. For a probability Radon measure $ \mu $ on $ B_1 $ ($ \mu(B_1)=1 $), define
\be
x_{\op{cm}}:=x_{\op{cm}}(\mu)=\int_{B_1}y\ud\mu(y).\label{centermass}
\ee
We call $ x_{\op{cm}} $ the center of mass for $ \mu $. 

\begin{defn}\label{defvj}
Inductively, we define $ \{(\lda_i,v_i)\}_{i=1}^n\subset\R_{\geq 0}\times\R^n $ as follows. Let 
$$
\lda_1:=\lda_1(\mu):=\max_{|v|^2=1}\int_{B_1}|(y-x_{\op{cm}})\cdot v|^2\ud\mu(y).
$$
Define $ v_1:=v_1(\mu) $ with $ |v_1|=1 $ such that
$$
\lda_1=\int_{B_1}|(y-x_{\op{cm}})\cdot v_1|^2\ud\mu(y).
$$
Given $ \{(\lda_j,v_j)\}_{j=1}^i $, we define $ \lda_{i+1} $ by
$$
\lda_{i+1}:=\lda_{i+1}(\mu):=\max_{\substack{|v|^2=1,\,\,v\cdot v_j=0,\\j\in\Z\cap[1,i]}}\int_{B_1}|(y-x_{\op{cm}})\cdot v|^2\ud\mu(y),
$$
and $ v_{i+1}:=v_{i+1}(\mu) $ is a unit vector with
$$
\lda_{i+1}:=\int_{B_1}|(y-x_{\op{cm}})\cdot v_{i+1}|^2\ud\mu(y).
$$
\end{defn}

By standard results of linear algebra, $ \{v_i\}_{i=1}^n $ is an orthonormal basis of $ \R^n $, and 
\be
\lda_1\geq\lda_2\geq...\geq\lda_n\geq 0.\label{ldaorder}
\ee
For $ j\in\Z\cap[1,n] $, let
\be
L_j:=L_j(\mu):=x_{\op{cm}}+\op{span}\{v_i\}_{i=1}^j\in\bA(n,j).\label{Vkmu}
\ee
Through the definitions of $ \{(\lda_i,v_i)\}_{i=1}^n $ as above, we can represent $ D_{\mu}^k(0,1) $ by the following lemma.

\begin{lem}[\cite{NV17}, Lemma 7.4]\label{NVlem7.4} 
Let $ \{(\lda_i,v_i)\}_{i=1}^n $ be given in Definition \ref{defvj}. If $ \mu $ is a probability Radon measure on $ B_1 $, then for any $ k\in\Z\cap[1,n] $, the minimum
$$
m_{n,k}(\mu):=\min_{L\in\bA(n,k)}\int_{B_1}\dist^2(y,L)\ud\mu(y)
$$
attains at $ L_k $, defined by \eqref{Vkmu}. Precisely, there holds
\be
m_{n,k}(\mu)=\int_{B_1}\dist^2(x,L_k)\ud\mu(y)=\left\{\begin{aligned}
&\sum_{i=k+1}^n\lda_i&\text{ if }&k\in\Z\cap[0,n-1],\\ 
&\,\,\,\,\,\,\,0&\text{ if }&k=n.
\end{aligned}\right.\label{mnkrepre}
\ee
\end{lem}

\begin{lem}[\cite{NV17}, Lemma 7.5]\label{NVlem7.5}
Let $ \{(\lda_i,v_i)\}_{i=1}^n $ be given in Definition \ref{defvj}. If $ \mu $ is a probability Radon measure on $ B_1 $, then
\be
\int_{B_1}((y-x_{\op{cm}})\cdot v_i)(y-x_{\op{cm}})\ud\mu(y)=\lda_iv_i\label{El}
\ee
for any $ i\in\Z\cap[1,n] $, where
\be
\lda_i=\int_{B_1}|(y-x_{\op{cm}})\cdot v_i|^2\ud\mu(y).\label{ldaiform}
\ee
\end{lem}

Applying this lemma and Proposition \ref{MonFor}, we have the following result.

\begin{lem}\label{L2lem1}
Let $ \ga>0 $ and $ \mu $ be a probability Radon measure. Assume that $ u\in(C_{\loc}^{0,\al}\cap H_{\loc}^1\cap L_{\loc}^{-p})(B_{20}) $ is a stationary solution of \eqref{MEMSeq} with respect to $ f\in M_{\loc}^{2\al+n-4+\ga,2}(B_{20}) $. Let $ \{(\lda_i,v_i)\}_{i=1}^n $ be given in Definition \ref{defvj}. There exists a constant $ C>0 $, depending only on $ n $ and $ p $ such that
\be
\lda_i\int_{B_5}|v_i\cdot\na u|^2\leq C\int_{B_1}W_f(u;y,1)\ud\mu(y)\label{viB5}
\ee
for any $ i\in\Z\cap[1,n] $.
\end{lem}
\begin{proof}
Fix $ i\in\Z\cap[1,n] $. If $ \lda_i=0 $, the result is trivially true. Thus, we let $ \lda_i>0 $. Taking inner product for both sides of \eqref{El} by $ \na u(z) $ with $ z\in B_5 $, it follows that
\be
\lda_i(v_i\cdot\na u(z))=\int_{B_1}((y-x_{\op{cm}})\cdot v_i)((y-x_{\op{cm}})\cdot\na u(z))\ud\mu(y).\label{mucm}
\ee
The definition of $ x_{\op{cm}} $ given by \eqref{centermass} implies that
$$
\int_{B_1}((y-x_{\op{cm}})\cdot v_i)\ud\mu(y)=0.
$$
Moreover, for any $ z\in B_5 $,
$$
\int_{B_1}((y-x_{\op{cm}})\cdot v_i)((z-x_{\op{cm}})\cdot\na u(z)-\al u(z))\ud\mu(y)=0.
$$
Incorporating with \eqref{mucm}, we have
$$
\lda_i(v_i\cdot\na u(z))=\int_{B_1}((y-x_{\op{cm}})\cdot v_i)((y-z)\cdot\na u(z)+\al u(z))\ud\mu(y).
$$
For any $ z\in B_5 $, due to \eqref{ldaiform} and Cauchy's inequality, there holds
\begin{align*}
\lda_i^2|v_i\cdot\na u(z)|^2&\leq\(\int_{B_1}|(y-x_{\op{cm}})\cdot v_i|^2\ud\mu(y)\)\(\int_{B_1}|(z-y)\cdot\na u(z)-\al u(z)|^2\ud\mu(y)\)\\
&=\lda_i\int_{B_1}|(z-y)\cdot\na u(z)-\al u(z)|^2\ud\mu(y)
\end{align*}
Integrating with respect to $ z\in B_5 $ for both sides of the above, we have
\be
\begin{aligned}
&\lda_i\int_{B_5}|v_i\cdot\na u|^2\leq\int_{B_1}\(\int_{B_5}|(z-y)\cdot\na u-\al u|^2\ud z\)\ud\mu(y).
\end{aligned}\label{zmuz}
\ee
Note that Corollary \ref{coruse} yields that
$$
\int_{B_5}|(z-y)\cdot\na u-\al u|^2\ud z\leq C\int_{B_6(y)}|(z-y)\cdot\na u-\al u|^2\ud z\leq C(n,p)W_f(u;y,1)
$$
for any $ y\in B_1 $. This, together with \eqref{zmuz}, directly shows \eqref{viB5}.
\end{proof}

\begin{proof}[Proof of Theorem \ref{beta2MEMS}]
Using Proposition \ref{ScalingProp} and a normalization of the measure $ \mu $, it can be assumed that $ s=1 $, $ x=0 $, and $ \mu $ is a probability Radon measure on $ B_1 $. For $ k=n $, \eqref{mnkrepre} implies $ D_{\mu}^n(0,1)=0 $. Thus, there is nothing to prove. Without loss of generality, let $ k\leq n-1 $. For $ \{(\lda_i,v_i)\}_{i=1}^n $ given in Definition \ref{defvj}, with $ \{L_j\}_{j=1}^n $ defined by \eqref{Vkmu}, we set $ \{V_j\}_{j=1}^n\subset\R^n $ as subspaces of $ \R^n $ such that 
$$
V_j:=L_j-x_{\op{cm}}=\op{span}\{v_i\}_{i=1}^j 
$$
with $ j\in\Z\cap[1,n] $. It follows from \eqref{ldaorder} and Lemma \ref{NVlem7.4} that
\be
\min_{L\subset\bA(n,k)}\int_{B_1}\dist^2(y,L)\ud\mu(y)=\sum_{i=k+1}^n\lda_i\leq(n-k)\lda_{k+1}.\label{dismufor}
\ee
Lemma \ref{L2lem1} yields that 
$$
\lda_i\int_{B_5}|v_i\cdot\na u|^2\leq C(n,p)\int_{B_1} W_f(u;y,1)\ud\mu(y)
$$
for any $ i\in\Z\cap[1,n] $. Consequently, 
$$
\sum_{i=1}^{k+1}\(\lda_i\int_{B_5}|v_i\cdot\na u|^2\)\leq C(n,p)\int_{B_1}W_f(u;y,1)\ud\mu(y).
$$
Again, by \eqref{ldaorder}, we have
$$
\lda_{k+1}\int_{B_5}|V_{k+1}\cdot\na u|^2\leq\sum_{i=1}^{k+1}\(\lda_i\int_{B_5}|v_i\cdot\na u|^2\)\leq C(n,p)\int_{B_1}W_f(u;y,1)\ud\mu(y).
$$
Owing to \eqref{lowerV}, this gives that
$$
\tau\lda_{k+1}\leq C\lda_{k+1}\int_{B_5}|V_{k+1}\cdot\na u|^2\leq C(n,p)\int_{B_1}W_f(u;y,1)\ud\mu(y),
$$
and then
$$
\lda_{k+1}\leq C(n,p)\tau^{-1}\int_{B_1}W_f(u;y,1)\ud\mu(y).
$$
The results now follow from \eqref{dismufor}.
\end{proof}

\section{Covering lemmas}\label{CoveringSection}

In this section, we establish several significant covering lemmas associated with quantitative stratification.

\begin{lem}[Main covering lemma]\label{maincoverMEMS}
Let $ \ga>0 $, $ k\in\Z\cap[1,n-2] $, $ 0<r<R\leq 1 $, and $ x_0\in B_{R_0} $. Assume that $ u\in(C_{\loc}^{0,\al}\cap H_{\loc}^1\cap L_{\loc}^{-p})(B_{4R_0}) $ is a stationary solution of \eqref{MEMSeq} with respect to $ f\in M_{\loc}^{2\al+n-4+\ga,2}(B_{4R_0}) $, satisfying 
$$
[u]_{C^{0,\al}(\ol{B}_{2R_0})}+[f]_{M^{2\al+n-4+\ga,2}(B_{2R_0})}\leq\Lda.
$$
For any $ \va>0 $, there exist $ \delta,\delta'>0 $, depending only on $ \va,\ga,\Lda,n $, and $ p $ such that if
\be
[f]_{M^{2\al+n-4+\ga,2}(B_{20R}(x_0))}<\delta,\label{fMdeltasmall}
\ee
then we have a collection of balls $ \{B_r(x)\}_{x\in\cC} $, satisfying $ \#\cC<+\ift $,
$$
S_{\va,\delta'r}^k(u)\cap B_R(x_0)\subset\bigcup_{x\in\cC}B_r(x),
$$
and the following properties.
\begin{enumerate}[label=$(\theenumi)$]
\item For any $ x\in\cC $, $ S_{\va,\delta'r}^k(u)\cap B_r(x)\neq\emptyset $.
\item There is $ C_{\op{M}}>0 $ depending only on $ \va,\ga,\Lda,n $, and $ p $ such that $ (\#\cC)r^k\leq C_{\op{M}}R^k $.
\end{enumerate} 
\end{lem}

To show this main covering lemma, we need the following two auxiliary results, referred to as the first and the second covering lemmas.

\begin{lem}[The first covering lemma]\label{cover1MEMS}
Let $ k\in\Z\cap[1,n-2] $, $ 0<\rho<\f{1}{1000} $, $ 0<r<R\leq 1 $, and $ x_0\in B_{R_0} $. Assume that $ u\in (C_{\loc}^{0,\al}\cap H_{\loc}^1\cap L_{\loc}^{-p})(B_{4R_0}) $ is a stationary solution of \eqref{MEMSeq} with respect to $ f\in M_{\loc}^{2\al+n-4+\ga,2}(B_{4R_0}) $, satisfying 
$$
[u]_{C^{0,\al}(\ol{B}_{2R_0})}+[f]_{M^{2\al+n-4+\ga,2}(B_{2R_0})}\leq\Lda.
$$
For any $ \va>0 $, there exist $ \delta,\delta'>0 $ depending only on $ \va,\ga,\Lda,n,p $, and $ \rho $ such that if
\be
[f]_{M^{2\al+n-4+\ga,2}(B_{20R}(x_0))}<\delta,\label{fcondition}
\ee
then we have a collection of balls $ \{B_{r_x}(x)\}_{x\in\cD} $, satisfying $ \#\cD<+\ift $,
$$
S_{\va,\delta'r}^k(u)\cap B_R(x_0)\subset\bigcup_{x\in\cD}B_{r_x}(x),
$$
and the following properties.
\begin{enumerate}[label=$(\theenumi)$]
\item $ r_x\geq r $ for any $ x\in\cD $, and
$$
\sum_{x\in\cD}r_x^k\leq C_{\op{I}}R^k,
$$
where $ C_{\op{I}}>0 $ depends only on $ n $
\item If $ x\in\cD $, then either $ r_x=r $, or there is $ L(x,r_x)\in\bA(n,k-1) $ such that
$$
\left\{y\in B_{2r_x}(x):\vt_f\(u;y,\f{\rho r_x}{20}\)>E-\delta\right\}\subset B_{\f{\rho r_x}{10}}(L(x,r_x))\cap B_{2r_x}(x),
$$
where 
\be
E:=E(x_0,R):=\sup_{y\in B_{2R}(x_0)}\vt_f(u;y,R).\label{EMdef}
\ee
\end{enumerate}
\end{lem}

\begin{lem}[The second covering lemma]\label{Cover2MEMS}
Let $ \ga>0 $, $ k\in\Z\cap[1,n-1] $, $ 0<r<R\leq 1 $, and $ x_0\in B_{R_0} $. Assume that $ u\in (C_{\loc}^{0,\al}\cap H_{\loc}^1\cap L_{\loc}^{-p})(B_{4R_0}) $ is a stationary solution of \eqref{MEMSeq} with respect to $ f\in M_{\loc}^{2\al+n-4+\ga,2}(B_{4R_0}) $, satisfying 
$$
[u]_{C^{0,\al}(\ol{B}_{2R_0})}+[f]_{M^{2\al+n-4+\ga,2}(B_{2R_0})}\leq\Lda.
$$
For any $ \va>0 $, there exist $ \delta,\delta'>0 $, depending only on $ \va,\ga,\Lda,n $, and $ p $ such that if 
$$
[f]_{M^{2\al+n-4+\ga,2}(B_{20R}(x_0))}<\delta,
$$
then we have a collection of balls $ \{B_{r_x}(x)\}_{x\in\cC} $, satisfying $ \#\cC<+\ift $,
$$
S_{\va,\delta'r}^k(u)\cap B_R(x_0)\subset\bigcup_{x\in\cC}B_{r_x}(x),
$$
and the following properties.
\begin{enumerate}[label=$(\theenumi)$]
\item $ r_x\geq r $ for any $ x\in\cC $, and 
\be
\sum_{x\in\cC}r_x^k\leq C_{\op{II}}R^k,\label{cover2estim}
\ee
where $ C_{\op{II}}>0 $ depends only on $ n $
\item For any $ x\in\cC $, either $ r_x=r $, or
\be
\sup_{y\in B_{2r_x}(x)}\vt_{f}(u;y,r_x)\leq E-\delta,\label{finalcover2ME}
\ee
where $ E $ is given by \eqref{EMdef}.
\end{enumerate}
\end{lem}

\subsection{Proof of Lemma \ref{maincoverMEMS}, given Lemma \ref{Cover2MEMS}}

For sufficiently small $ \delta,\delta'=\delta,\delta'(\va,\ga,\Lda,n,p)>0 $, we will inductively construct a covering of $ S_{\va,\delta'r}^k(u)\cap B_R(x_0) $, denoted by $ \{B_{r_x}(x)\}_{x\in\cC_i} $ with $ i\in\Z_+ $ such that
$$
S_{\va,\delta r}^k(u)\cap B_R(x_0)\subset\bigcup_{x\in\cC_i}B_{r_x}(x)=\bigcup_{x\in\cC_i^{(1)}}B_{r_x}(x)\cup\bigcup_{x\in\cC_i^{(2)}}B_{r_x}(x),
$$
and the following properties hold.
\begin{enumerate}[label=$(\op{C}_{\op{M}}\theenumi)$]
\item\label{CM1} For any $ x\in\cC_i $, 
\be
S_{\va,\delta'r}^k(u)\cap B_R(x_0)\cap B_{r_x}(x)\neq\emptyset.\label{Svadeltapr}
\ee
\item\label{CM2} If $ x\in\cC_i^{(1)} $, then $ r_x=r $.
\item\label{CM3} If $ x\in\cC_i^{(2)} $, then $ r_x>r $, and
\be
\sup_{y\in B_{2r_x}(x)}\vt_f(u;y,r_x)\leq E-i\delta,\label{Ci2coverMEMS}
\ee
where $ E=E(0,R) $ is given by \eqref{EMdef}.
\item\label{CM4} We have the estimate
\be
\sum_{x\in\cC_i}r_x^k\leq (1+C_{\op{II}}(n))^iR^k,\label{CiestimMEMS}
\ee
where $ C_{\op{II}}(n)>0 $ is the constant given in Lemma \ref{Cover2MEMS}.
\end{enumerate}
First, assume that we have already constructed such coverings. It follows from the first point of Proposition \ref{propupture} that there exists $ C'=C'(\ga,\Lda,n,p)>0 $ such that $ E $ has the upper bound
\be
E\leq C'(\ga,\Lda,n,p),\label{Ebound}
\ee
and if for some $ y\in B_{R_0} $, $ 0<s\leq 1 $, there holds
\be
\vt_f(u;y,s)<-C'(\ga,\Lda,n,p),\label{vtfuyssmall}
\ee
then $ u(y)\geq s^{\al} $. Moreover, Lemma \ref{lowernva} implies the existence of $ \sg=\sg(\va,\ga,\Lda,n,p)>0 $ being sufficiently small such that $ u $ is $ (n,\va) $-symmetric in $ B_{\sg s}(y) $. For $ x\in\cC_i^{(2)} $, and $ y\in B_{r_x}(x) $, by \eqref{Ci2coverMEMS}, we have
$$
\theta_f(u;y,r_x)\leq E-i\delta\leq C(\ga,\Lda,n,p)-i\delta(\va,\ga,\Lda,n,p).
$$
Choosing sufficiently large $ i_0=i_0(\va,\ga,\Lda,n,p)>0 $, we see that \eqref{vtfuyssmall} is satisfied for $ y $ and $ r_x $ if $ i>i_0 $. This implies that $ u $ is $ (n,\va) $-symmetric in $ B_{\sg r_x}(y) $. Since $ r_x>r $, if we further choose smaller $ \delta'=\delta'(\va,\ga,\Lda,n,p)>0 $ such that $ \delta'r<\sg r<\sg r_x $, then $ y\notin S_{\va,\delta'r}^k(u) $. By the arbitrariness of $ y\in B_{r_x}(x) $, we have $ S_{\va,\delta'r}^k(u)\cap B_{r_x}(x)=\emptyset $, a contradiction to \ref{CM1}. Thus, if $ i>i_0 $, then $ \cC_i^{(2)}=\emptyset $, and the result follows from \eqref{CiestimMEMS}.

For $ i=1 $, the properties \ref{CM1}-\ref{CM4} follow directly from the application of Lemma \ref{Cover2MEMS} to the ball $ B_R(x_0) $. Assume that \ref{CM1}-\ref{CM4} hold for $ i\in\Z_+ $. We will conduct the construction for $ i+1 $. For any $ x\in\cC_i^{(2)} $, Lemma \ref{Cover2MEMS} yields $ \{B_{r_y}(y)\}_{y\in\cC_{x,i}} $ such that 
$$
S_{\va,\delta r}^k(u)\cap B_{r_x}(x)\subset\bigcup_{y\in C_{x,i}}B_{r_y}(y)=\bigcup_{y\in C_{x,i}^{(1)}}B_{r_y}(y)\cup\bigcup_{y\in C_{x,i}^{(2)}}B_{r_y}(y),
$$
satisfying the following facts.
\begin{itemize}
\item For any $ y\in\cC_{x,i} $, $ B_{2r_y}(y)\subset B_{2r_x}(x) $.
\item If $ y\in\cC_{x,i}^{(1)} $, then $ r_y=r $.
\item If $ y\in\cC_{x,i}^{(2)} $, then 
$$
\sup_{z\in B_{2r_y}(y)}\vt_f(u;z,r_y)\leq\sup_{z\in B_{2r_x}(x)}\vt_f(u;z,r_x)-\delta\leq E-(i+1)\delta,
$$
where we have used \eqref{Ci2coverMEMS} for the second inequality.
\item We have the estimate
\be
\sum_{y\in\cC_{x,i}}r_y^k\leq C_{\op{II}}(n)r_x^k.\label{CxiCx2}
\ee
\end{itemize}
Define $ \{\cC_{i+1}^{(j)}\}_{j=1,2} $ as
$$
\cC_{i+1}^{(1)}=\cC_i^{(1)}\cup\bigcup_{x\in \cC_i^{(2)}}\cC_{x,i}^{(1)},\quad\cC_{i+1}^{(2)}=\bigcup_{x\in \cC_i^{(2)}}\cC_{x,i}^{(2)},\quad\text{and}\quad\cC_{i+1}=\cC_{i+1}^{(1)}\cup\cC_{i+1}^{(2)}.
$$
According to \eqref{CxiCx2} and \eqref{CiestimMEMS}, \begin{align*}
\sum_{x\in\cC_{i+1}}r_x^k&\leq\sum_{x\in\cC_i^{(1)}}r_x^k+\sum_{x\in\cC_i^{(2)}}\sum_{y\in\cC_{x,i}}r_y^k\leq (1+C_{\op{II}}(n))\(\sum_{x\in\cC_i}r_x^k\)\leq (1+C_{\op{II}}(n))^{i+1}R^k,
\end{align*}
which completes the proof.

\subsection{Proof of Lemma \ref{Cover2MEMS}, given Lemma \ref{cover1MEMS}}

Up to a translation, we assume that $ x_0=0 $. Letting $ 0<\rho<\f{1}{1000} $ be determined later, we can choose $ \ell\in\Z_+ $ such that 
\be
\(\f{\rho}{20}\)^{\ell}R<r\leq\(\f{\rho}{20}\)^{\ell-1}R.\label{lchose}
\ee
For $ i\in\Z\cap[1,\ell] $, and $ \delta,\delta'=\delta(\va,\ga,\Lda,n,p)>0 $ sufficiently small, we will inductively construct collections of balls $ \{B_{r_x}(x)\}_{x\in\cR_i\cup\cF_i\cup\cB_i} $ such that
$$
S_{\va,\delta'r}^k(u)\cap B_R\subset\bigcup_{x\in\cR_i}B_r(x)\cup\bigcup_{x\in\cF_i}B_{r_x}(x)\cup\bigcup_{x\in\cB_i}B_{r_x}(x),
$$
and the following properties hold.
\begin{enumerate}[label=($\op{C}_{\op{II}}$\theenumi)]
\item \label{Bp1}For any $ x\in\cR_i\cup\cF_i\cup\cB_i $, $ r_x\geq r $ and $ B_{2r_x}(x)\subset B_{2R} $.
\item \label{Bp2}If $ x\in\cR_i $, then $ r_x=r $.
\item \label{Bp3}If $ x\in\cF_i $, then
$$
\sup_{y\in B_{2r_x}(x)}\vt_f(u;y,r_x)\leq E-\delta,
$$
where $ E=E(0,R) $.
\item \label{Bp4}If $ x\in\cB_i $, neither of the properties in \ref{Bp2} and \ref{Bp3} is true, and 
\be
r<r_x\leq\(\f{\rho}{20}\)^iR.\label{rBx}
\ee
\item \label{Bp5}For any $ i\in\Z\cap[1,\ell] $,
\be
\sum_{x\in\cR_i\cup\cF_i}r_x^k\leq C_{\op{I}}(n)\(\sum_{j=0}^i\f{1}{10^j}\)R^k\quad\text{and}\quad\sum_{x\in\cB_i}r_x^k\leq\f{R^k}{10^i}.\label{C3es}
\ee
\end{enumerate}

For $ i=\ell $, \eqref{lchose} and \eqref{rBx}, implies that $ \cB_{\ell}=\emptyset $. Then $ \{B_{r_x}(x)\}_{\cR_{\ell}\cup\cF_{\ell}} $ is the desired covering. By properties \ref{Bp2}, if $ x\in\cR_{\ell} $, then $ r_x=r $. It follows from \ref{Bp3} that for any $ x\in\cF_{\ell} $, there holds \eqref{finalcover2ME}. Moreover, the estimate \eqref{cover2estim} is a direct consequence of \eqref{C3es} with $ i=\ell $.

\subsubsection*{Step 1. Preliminaries of the construction}

We first fix a ball $ B_{2s}(x)\subset B_{2R} $ and consider a covering of $ S_{\va,\delta'r}^k(u)\cap B_s(x) $. Indeed, we construct $ \{B_{r_y}(y)\}_{y\in\cR_x\cup\cF_x\cup\cB_x} $ such that
$$
S_{\va,\delta'r}^k(u)\cap B_s(x)\subset\bigcup_{y\in\cR_x}B_{r_y}(y)\cup\bigcup_{y\in\cF_x}B_{r_y}(y)\cup\bigcup_{y\in\cB_x}B_{r_y}(y),
$$
satisfying the following facts.
\begin{enumerate}[label=($\op{C}_{\op{II}}$1.\theenumi)]
\item\label{C211} For any $ y\in\cR_x\cup\cF_x\cup\cB_x $, $ r_y\geq r $ and $ B_{2r_y}(y)\subset B_{2s}(x) $.
\item\label{C212} If $ y\in\cR_x $, then $ r_y=r $.
\item\label{C213} If $ y\in\cF_x $, then
\be
\sup_{z\in B_{2r_y}(y)}\vt_f(u;z,r_y)\leq E-\delta.\label{energydrop}
\ee
\item\label{C214} If $ y\in\cB_x $, then $
r<r_y\leq\f{\rho s}{20} $.
\item\label{C215} We have estimates
$$
\sum_{y\in\cR_x\cup\cF_x}r_y^k\leq C_{\op{I}}(n)s^k\quad\text{and}\quad
\sum_{y\in\cB_x}r_y^k\leq\f{s^k}{10}.
$$
\end{enumerate}

We first apply Lemma \ref{cover1MEMS} to the ball $ B_s(x) $ and choose sufficiently small constants $ \delta,\delta'=\delta,\delta(\va,\ga,\Lda,n,p,\rho)>0 $ to obtain a covering 
$$
S_{\va,\delta'r}^k(u)\cap B_s(x)\subset\bigcup_{y\in\cD}B_{r_y}(y),
$$
with $ \#\cD<+\ift $, $ r_y\geq r $, and
$$
\sum_{x\in\cD}r_y^k\leq C_{\op{II}}(n)s^k
$$
for any $ y\in\cD $. Moreover, if $ y\in\cD $, then $ B_{2r_y}(y)\subset B_{2s}(x) $, and either $ r_y=r $ or there exists $ L(y,r_y)\in\bA(n,k-1) $ such that
\be
\left\{z\in B_{2r_y}(y):\vt_f\(u;z,\f{\rho r_y}{20}\)>E(x,s)-\delta\right\}\subset B_{\f{\rho r_y}{10}}(L(y,r_y))\cap B_{2r_y}(y),\label{BsFcontain}
\ee
where 
$$
E(x,s):=\sup_{z\in B_{2s}(x)}\vt(u;z,s).
$$
We classify the points in $ \cD $ into two subcollections. Precisely, we let $ \cD=\cD^{(r)}\cup\cD^{(+)} $, depending on the standards below.
\begin{itemize}
\item If $ y\in\cD^{(r)} $, then $ \f{\rho r_y}{20}\leq r $.
\item If $ y\in\cD^{(+)} $, then $ \f{\rho r_y}{20}>r $.
\end{itemize}
Next, we will refine balls with centers in $ \cD $ through recovering.

For $ y\in\cD^{(r)} $, we cover $ B_{r_y}(y) $ with balls $ \{B_r(z)\}_{z\in\cR_{x}^{(y)}} $, satisfying
\be
B_{r_y}(y)\subset\bigcup_{z\in\cR_x^{(y)}}B_r(z),\quad\#\cR_x^{(y)}\leq C(n)\rho^{-n},\label{RxCnbx}
\ee
and $ B_{2r}(z)\subset B_{2s}(x) $ for any $ z\in\cR_x^{(y)} $. We define the collection of all these centers of balls by
$$
\cR_x:=\bigcup_{y\in\cD^{(r)}}\cR_x^{(y)}.
$$

For $ y\in\cD^{(+)} $, since $ r_y>\f{20r}{\rho}>r $, we have \eqref{BsFcontain}. Consider a covering of $ B_{r_y}(y) $ with balls of radius $ \f{\rho r_y}{20}>r $ centered inside this ball, denoted by $ \{B_{\f{\rho r_y}{20}}(z)\}_{z\in\cB_x^{(y)}\cup\cF_x^{(y)}} $ such that balls in the collection $
\{B_{\f{\rho r_y}{40}}(z)\}_{z\in\cB_x^{(y)}\cup\cF_x^{(y)}} $ are pairwise disjoint and
$$
B_{r_y}(y)\subset\bigcup_{z\in\cB_x^{(y)}}B_{\f{\rho r_y}{20}}(z)\cup \bigcup_{z\in\cF_x^{(y)}}B_{\f{\rho r_y}{20}}(z),
$$
where
\be
\left\{z\in B_{2r_y}(y):\vt_f\(u;z,\f{\rho r_y}{20}
\)>E(x,s)-\delta\right\}\cap\bigcup_{z\in\cF_x^{(y)}}B_{\f{\rho r_y}{10}}(z)=\emptyset.\label{Fzdef}
\ee
For any $ \zeta\in\cB_{x}^{(y)} $, by \eqref{BsFcontain}, there exists $ \zeta'\in B_{\f{\rho r_y}{10}}(\zeta) $ such that
$$
\zeta'\in\left\{z\in B_{2r_y}(y):\vt_f\(u;z,\f{\rho r_y}{20}\)>E(x,s)-\delta\right\}\subset B_{\f{\rho r_y}{10}}(L(y,r_y))\cap B_{r_y}(y).
$$
Thus, we have $
\zeta\in B_{\rho r_y}(L(y,r_y))\cap B_{r_y}(y) $. By the arbitrariness of $ \zeta $,
\be
\cB_{x}^{(y)}\subset B_{\rho r_y}(L(y,r_y))\cap B_{r_y}(y).\label{cBxyBrho}
\ee
If $ z\in\cF_x^{(y)} $, then $ r_z=\f{\rho r_y}{20} $. \eqref{Fzdef} and the property $ B_{\f{\rho r_y}{10}}(z)\subset B_{2r_y}(y) $ imply
$$
\sup_{\zeta\in B_{2r_z}(z)}\vt_f(u;\zeta,r_z)=\sup_{\zeta\in B_{\f{\rho r_y}{10}}(z)}\vt_f\(u;\zeta,\f{\rho r_y}{20}\)\leq E(x,s)-\delta\leq E-\delta.
$$
Since $ \{B_{\f{\rho r_y}{40}}(z)\}_{z\in\cB_{x}^{(y)}} $ are pairwise disjoint, \eqref{cBxyBrho} and the fact that $ L(y,r_y)\in\bA(n,k-1) $ give that
\be
\#\cF_{x}^{(y)}\leq C(n)\rho^{-n}\quad\text{and}\quad
\#\cB_{x}^{(y)}\leq C(n)\rho^{1-k}.\label{BxFxcount}
\ee

Define
$$
\cB_x:=\bigcup_{y\in\cD^{(+)}}\cB_x^{(y)}\quad\text{and}\quad\cF_x:=\bigcup_{y\in\cD^{(+)}}\cF_x^{(y)}.
$$
It follows from \eqref{RxCnbx} and \eqref{BxFxcount} that
\begin{align*}
\sum_{y\in\cR_x\cup\cF_x}r_y^k&\leq C\rho^{-n+k}\(\sum_{y\in\cD}r_y^k\)\leq C(n)\rho^{k-n}C_{\op{II}}(n)s^k,\\
\sum_{y\in\cB_x}r_y^k&\leq C(n)\rho\(\sum_{y\in\cD}r_y^k\)\leq C(n)\rho C_{\op{II}}(n)s^k.
\end{align*}
Choosing $ 0<\rho<\f{1}{1000} $ sufficiently small, there is $ C_{\op{II}}(n)>0 $, satisfying
\be
\sum_{y\in\cR_x\cup\cF_x}r_y^k\leq C_{\op{II}}(n)s^k\quad\text{and}\quad\sum_{y\in\cB_x}r_y^k\leq\f{s^k}{10}.\label{induces}
\ee
By the analysis of all above, properties \ref{C211}-\ref{C215} hold.

\subsubsection*{Step 2. Inductive constructions} 

Given the preliminary result in the previous step, we can now conduct inductive constructions for our covering. For $ i=1 $, we apply the results in Step 1 to the ball $ B_R $, and the properties for such a case follow directly. Assume that for $ i\in\Z\cap[1,\ell-1] $, there is a covering of $ S_{\va,\delta'r}^k(u)\cap B_R $, given by
$$
S_{\va,\delta'r}^k(u)\cap B_R\subset\bigcup_{x\in\cR_i}B_r(x)\cup\bigcup_{x\in\cF_i}B_{r_x}(x)\cup\bigcup_{x\in\cB_i}B_{r_x}(x),
$$
satisfying \ref{Bp1}-\ref{Bp5}. Again, applying the results in Step 1 to balls in $ \{B_{r_x}(x)\}_{x\in\cB_i} $ respectively, we let
$$
\cR_{i+1}=\cR_i\cup\bigcup_{x\in\cB_i}\cR_x,\quad\cF_{i+1}=\cF_i\cup\bigcup_{x\in\cB_i}\cF_x,\quad\text{and}\quad\cB_{i+1}=\bigcup_{x\in\cB_i}\cB_x.
$$
Consequently, this covering satisfies \ref{Bp1}-\ref{Bp4}. It remains to show \ref{Bp5}. Indeed, estimates in \eqref{induces} imply that
$$
\sum_{x\in\cR_{i+1}\cup\cF_{i+1}}r_x^k\leq\sum_{x\in\cR_i\cup\cF_i}r_x^k+\sum_{x\in\cB_i}\sum_{y\in\cR_x\cup\cF_x}r_y^k\leq C_{\op{II}}(n)\(\sum_{j=0}^{i+1}\f{1}{10^j}\)R^k,
$$
and
$$
\sum_{x\in\cB_{i+1}}r_x^k\leq \sum_{x\in\cB_i}\sum_{y\in\cB_x}r_y^k\leq\f{R^k}{10^{i+1}},
$$
which completes the proof.

\subsection{Proof of Lemma \ref{cover1MEMS}} Up to a translation and further coverings, without loss of generality, we let $ 0<r<\f{R}{1000}<\f{1}{1000^2} $ and $ x_0=0 $. Fix $ \va>0 $ and $ 0<\rho<\f{1}{1000} $. There exists $ \ell\in\Z_+ $ such that
\be
\rho^{\ell}R\leq r<\rho^{\ell-1}R.\label{choiceofrRell}
\ee

\begin{defn}
Let $ 0<\delta<1 $, $ 0<s<R $, and $ x\in B_{2R} $. Suppose that $ B_{2s}(x)\subset B_{2R} $. Define
\begin{align*}
F_{\delta}(x,s)&:=\left\{y\in B_{2s}(x):\vt_f\(u;y,\f{\rho s}{20}\)>E-\delta\right\},\\
F_{\delta}'(x,s)&:=\left\{y\in B_{2s}(x):\vt_f\(u;y,s\)-\vt_f\(u;y,\f{\rho s}{20}\)<\delta\right\}.
\end{align*}
\end{defn}

Given the definition of $ E $ by \eqref{EMdef}, since $ 0<s<R $, it follows from Proposition \ref{MonFor} that  for any $ y\in F_{\delta}(x,s) $,
$$
\vt_f\(u;y,s\)-\vt_f\(u;y,\f{\rho s}{20}\)\leq\vt_f(u;y,R)-\vt_f\(u;y,\f{\rho s}{20}\)<E-(E-\delta)=\delta.
$$
As a result,
\be
F_{\delta}(x,s)\subset F_{\delta}'(x,s).\label{FdeltaFprime}
\ee
If there exists $ L(0,R)\in\bA(n,k-1) $ such that
$$
F_{\delta}(0,R)\subset B_{\f{\rho R}{10}}(L(0,R))\cap B_{2R},
$$
then $ \{B_R\} $ is the desired covering. Thus, it is natural to assume that $ F_{\delta}(0,R) $ $ \f{\rho R}{20} $-effectively spans $ L'(0,R)\in\bA(n,k) $.

We will choose sufficiently small $ \delta,\delta'=\delta,\delta'(\va,\ga,\Lda,n,p,\rho)>0 $ and construct a collection of balls $ \{B_{r_x}(x)\}_{x\in\cD_i} $ with $ \cD_i:=\cB_i\cup\cG_i $ for $ i\in\Z\cap[1,\ell] $, satisfying
\be
S_{\va,\delta'r}^k(u)\cap B_R\subset\bigcup_{x\in\cB_i}B_{r_x}(x)\cup\bigcup_{x\in\cG_i}B_{r_x}(x),\label{CoverSMEMS}
\ee
and the following properties.
\begin{enumerate}[label=($\op{C}_{\op{I}}$\theenumi)]
\item \label{Mp1}The balls in the collection $ \{B_{\f{r_x}{10}}(x)\}_{x\in\cD_i} $ are pairwise disjoint.
\item \label{Mp2} For any $ x\in\cD_i $, $ B_{2r_x}(x)\subset B_{2R} $, and $
S_{\va,\delta' r}^k(u)\cap B_R\cap B_{\f{r_x}{2}}(x)\neq\emptyset $.
\item \label{Mp3}If $ x\in\cB_i $, then $ r_x\geq\rho^iR $, and there exists $ L(x,r_x)\in\bA(n,k-1) $ such that
$$
F_{\delta}(x,r_x)\subset B_{\f{\rho r_x}{10}}(L(x,r_x))\cap B_{2r_x}(x).
$$
\item \label{Mp4}If $ x\in\cG_i $ and $ i\in\Z\cap[1,\ell-1] $, then $ r_x=\rho^iR $ and $ F_{\delta}(x,r_x) $ $ \f{\rho r_x}{20} $-effectively spans $ L'(x,r_x)\in\bA(n,k) $. If $ x\in\cG_{\ell} $, then $ r_x=r $.
\item \label{Mp5}For any $ x\in\cD_i $, we have
$$
\vt_f\(u;x,\f{r_x}{20}\)>E-\xi,
$$
where $ \xi=\xi(\va,\ga,\Lda,n,p,\rho)>0 $ is to be determined later.
\item \label{Mp6}There exists $ \tau=\tau(\va,\ga,\Lda,n,p,\rho)>0 $ such that for any $ x\in\cD_i $ and $ t\in[r_x,R] $,
\be
\inf_{V\in\bG(n,k+1)}\(t^{2-2\al-n}\int_{B_t(x)}|V\cdot\na u|^2\)>\tau.\label{VgeqMEMS}
\ee
\item \label{Mp7}If $ i=\ell $, then
\be
\sum_{x\in\cD_{\ell}}r_x^k\leq C_{\op{I}}(n)R^k.\label{BlGlMEMS}
\ee
\end{enumerate}

By \ref{Mp1}-\ref{Mp7}, the collection of balls $ \{B_{r_x}(x)\}_{x\in\cD_{\ell}} $ satisfies all the desired properties in Lemma \ref{cover1MEMS}, making it suitable for our needs.

Before giving the total details, we first outline the construction of $ \{B_{r_x}(x)\}_{x\in\cD_i} $. The procedure can be divided into four steps. The first three steps involve the inductive constructions of coverings, as described in \eqref{CoverSMEMS}, which satisfies properties \ref{Mp1}-\ref{Mp6}. In Step 1, we will establish the base case for the induction. In Step 2, we focus on inductive arguments for the proof from $ i $ to $ i+1 $ when $ i\in\Z\cap[1,\ell-1] $. Step 3 involves constructing the covering from $ \ell-1 $ to $ \ell $. Finally, in Step 4, we will prove the essential estimate \eqref{BlGlMEMS} in property \ref{Mp7}, where we will apply Reifenberg-type results (Theorem \ref{Rei1}).

\subsubsection*{Step 1. The base of the induction}

We begin by recalling that we have already assumed that $ F_{\delta}(0,R) $ $ \f{\rho R}{20} $-effectively spans $ L'(0,R)\in\bA(n,k) $. Thus, by applying \eqref{FdeltaFprime} and Proposition \ref{FpropMEMS} with $ \beta=\f{\rho}{20} $ and $ s=R $, there exist $ \delta,\delta'=\delta,\delta'(\va,\ga,\Lda,n,p,\rho)>0 $, leading to
\be
S_{\va,\delta'r}^k(u)\cap B_R\subset S_{\va,\delta'R}^k(u)\cap B_R\subset B_{\f{\rho R}{10}}(L'(0,R))\cap B_R.\label{includeSva2}
\ee
Next, we choose balls $ \{B_{\rho R}(x)\}_{x\in\cD_1} $ such that
\be
B_{\f{\rho R}{10}}(L'(0,R))\cap B_R\subset\bigcup_{x\in\cD_1}B_{\rho R}(x),\label{coverBrhoR}
\ee
and the following properties hold.
\begin{itemize}
\item We have
\be
\cD_1\subset L'(0,R)\cap B_{\f{3R}{2}}.\label{cD1Lprime}
\ee
\item For any $ x \in \mathcal{D}_1 $,
\be
S_{\va,\delta'r}^k(u)\cap B_R\cap B_{\f{\rho R}{2}}(x)\neq\emptyset.\label{BRinte}
\ee
\item The balls in $ \{B_{\f{\rho R}{10}}(x)\}_{x\in\cD_1} $ are pairwise disjoint.
\end{itemize}
Given $ \cD_1 $, we further divide it into $ \cD_1=\cB_1\cup\cG_1 $, where
\be
\{B_{r_x}(x)\}_{x\in\cD_1}=\{B_{\rho R}(x)\}_{x\in\cD_1}:=\{B_{\rho R}(x)\}_{x\in\cB_1}\cup\{B_{\rho R}(x)\}_{x\in\cG_1},\label{ballsD1}
\ee
satisfying the following facts.
\begin{itemize}
\item If $ x\in\cB_1 $, then there exists a $ L(x,\rho R)\in\bA(n,k-1) $ such that
$$
F_{\delta}(x,\rho R)\subset B_{\f{\rho^2R}{10}}(L(x,\rho R))\cap B_{\rho R}(x).
$$
\item If $ x\in\cG_1 $, then $ F_{\delta}(x,\rho R) $ $ \f{\rho^2R}{20} $-effectively spans $ L'(x,\rho R)\in\bA(n,k) $.
\end{itemize}
According to \eqref{includeSva2} and \eqref{coverBrhoR}, balls in $ \{B_{r_x}(x)\}_{x\in\cD_1} $ form a covering of $ S_{\va,\delta'r}^k(u)\cap B_R $, thus confirming that \eqref{CoverSMEMS} holds for $ i=1 $. Since $ 0<\rho<\f{1}{1000} $, we have $
B_{2\rho R}(x)\subset B_{2R} $ for any $ x\in\cD_1 $. Consequently, Properties \ref{Mp1} and \ref{Mp2} are satisfied. Also, \ref{Mp3} and \ref{Mp4} follows from the definition of $ \cB_1 $ and $ \cG_1 $. We claim that if $
\delta=\delta(\va,\Lda,n,p,\rho,\xi)>0 $ is sufficiently small with \eqref{fcondition} holding true, then
\be
\vt_f\(u;x,\f{\rho R}{20}\)>E-\xi,\label{rhoRux}
\ee
for any $ x\in\cD_1 $, where $ \xi>0 $ will be determined later. This claim directly implies \ref{Mp5}. The definition \eqref{EMdef} yields that
$$
\sup_{y\in B_{2R}}\vt_f(u;y,R)\leq E.
$$
Since $ F_{\delta}(0,R) $ $ \f{\rho R}{20} $-effectively spans $ L'(0,R) $, we can apply Lemma \ref{lempinchMEMS} with $ \beta=\f{\rho}{20} $ and $ s=R $, implying the estimate in \eqref{rhoRux} for any $ x\in L'(0,R)\cap B_{2R} $ as long as $ \delta=\delta(\va,\ga,\Lda,n,p,\rho)>0 $ is sufficiently small, and the condition \eqref{fcondition} is satisfied. Due to \eqref{cD1Lprime}, it also holds for any $ x\in\cD_1 $. Thus, we obtain this claim. To establish the induction base, we only need to verify the property \ref{Mp6} for balls given in \eqref{ballsD1}. For fixed $ y\in\cD_1 $ and $ t\in[\rho R,R] $, we must show the estimate \eqref{VgeqMEMS}. By \eqref{BRinte}, there exists
\be
y'\in S_{\va,\delta'r}^k(u)\cap B_R\cap B_{\f{\rho R}{2}}(y).\label{yprimepro}
\ee
By Remark \ref{remkplus}, there exist sufficiently small $ \sg,\sg'=\sg,\sg'(\va,\ga,\Lda,n,p,\rho)\in(0,1) $ such that if $ u $ is not $ (k+1,\va) $-symmetric in $ B_s(y') $ for any $ s\in[\f{\sg't}{2},\f{t}{2}] $, then
\be
\inf_{V\in\bG(n,k+1)}\[\(\f{t}{2}\)^{2-2\al-n}\int_{B_{\f{t}{2}}(y')}|V\cdot\na u|^2\]>\sg.\label{VGnk1}
\ee
Let $ \delta'=\delta'(\va,\ga,\Lda,n,p,\rho)>0 $ be sufficiently small such that $ \delta'r\leq\f{\sg't}{2} $. Combining this with \eqref{yprimepro}, we conclude that $ u $ is not $ (k+1,\va) $-symmetric in $ B_s(y') $ for any $ \delta'r\leq s<1 $, and then \eqref{VGnk1} holds. \eqref{yprimepro} and the fact that $ t\geq\rho R $ give that $ B_{\f{t}{2}}(y')\subset B_t(y) $. Thus, we have
\be
\begin{aligned}
t^{2-2\al-n}\int_{B_t(y)}|V\cdot\na u|^2&\geq 2^{2-2\al-n}\[\(\f{t}{2}\)^{2-2\al-n}\int_{B_{\f{t}{2}}(y')}|V\cdot\na u|^2\]\geq 2^{2-2\al-n}\sg:=\tau
\end{aligned}\label{geqga}
\ee
for any $ V\in\bG(n,k+1) $. Consequently, when $ i=1 $, the property \ref{Mp6} follows, completing the proof of the base case for the induction.

\subsubsection*{Step 2. Construction from \texorpdfstring{$ i $}{} to \texorpdfstring{$ i+1 $}{} for \texorpdfstring{$ i\in\Z\cap[1,\ell-2] $}{}}

Let us assume that for $ i\in\Z\cap[1,\ell-2] $, the covering \eqref{CoverSMEMS} has been constructed with properties \ref{Mp1}-\ref{Mp6}. We aim to give the construction for $ i+1 $ that satisfies properties \ref{Mp1}-\ref{Mp6}. Generally, we will recover the balls centered at $ \cG_i $ while keeping those centered at $ \cB_i $ unchanged. Fix $ x\in\cG_i $, we have $ r_x=\rho^iR $. By \ref{Mp4} for $ i $, $ F_{\delta}(x,r_x) $ $ \f{\rho^{i+1}R}{20} $-effectively spans $ L'(x,r_x)\in\bA(n,k) $. Using \eqref{FdeltaFprime}, we can choose $ \delta,\delta'=\delta,\delta'(\va,\ga,\Lda,n,p,\rho)>0 $ sufficiently small, and apply Proposition \ref{FpropMEMS} with $ \beta=\f{\rho}{10} $ and $ s=\rho^iR $. Thus,
\be
S_{\va,\delta'r}^k(u)\cap B_{\rho^iR}(x)\subset B_{\f{\rho^{i+1}R}{10}}(L'(x,\rho^iR))\cap B_{\rho^iR}(x).\label{SvaR5MEMS}
\ee
Define
\be
A_i:=\(\bigcup_{y\in\cG_i}\(L'(y,\rho^iR)\cap B_{\f{3\rho^iR}{2}}(y)\)\)\backslash\(\bigcup_{y\in\cB_i}B_{\f{4r_y}{5}}(y)\).\label{defAiMEMS}
\ee
We claim that
\be
(S_{\va,\delta'r}^k(u)\cap B_R)\backslash\(\bigcup_{y\in\cB_i}B_{r_y}(y)\)\subset B_{\f{\rho^{i+1}R}{5}}(A_i).\label{SvadeltarBryMEMS}
\ee
Indeed for any $ z $ in the right-hand side of above, properties \ref{Mp1}-\ref{Mp4} for $ i $ give that
\be
B_{\f{\rho^{i+1}R}{5}}(z)\cap\(\bigcup_{y\in\cB_i}B_{\f{4r_y}{5}}(y)\)=\emptyset.\label{emptyz}
\ee
It follows from \eqref{CoverSMEMS} that there exists $ x'\in\cG_i $ such that $ z\in B_{\rho^iR}(x') $. The property \eqref{SvaR5MEMS} yields
$$
L'(x',\rho^iR)\cap B_{\f{3\rho^iR}{2}}(x')\cap B_{\f{\rho^{i+1}R}{5}}(z)\neq\emptyset.
$$
This, together with \eqref{emptyz}, implies that $ A_i\cap B_{\f{\rho^{i+1}R}{5}}(z)\neq\emptyset $, which leads to $ z\in B_{\f{\rho^{i+1}R}{5}}(A_i) $. As a result, the claim \eqref{SvadeltarBryMEMS} holds. Define $ \cD_{A_i}\subset A_i $ as a maximal subset of points, satisfying
\be
\dist(y,z)\geq\f{\rho^{i+1}R}{5}\label{5disjo}
\ee
for any $ y,z\in\cD_{A_i} $, where the maximality means that $ \#\cD_{A_i} $ is the largest possible among those satisfying \eqref{5disjo}. By the definition of $ A_i $, the balls in the collection
$$
\{B_{\f{\rho^{i+1}R}{10}}(y)\}_{y\in\cD_{A_i}}\cup\{B_{\f{r_y}{10}}(y)\}_{y\in\cB_i}
$$
are pairwise disjoint. Using the maximality property of $ \cD_{A_i} $, it follows that
\be
(S_{\va,\delta'r}^k(u)\cap B_R)\backslash\(\bigcup_{y\in\cB_i}B_{r_y}(y)\)\subset\bigcup_{y\in\cD_{A_i}}B_{\f{2\rho^{i+1}R}{5}}(y).\label{Svadeltaeus}
\ee
Without changing the notation, we eliminate balls in $ \{B_{\rho^{i+1}R}(y)\}_{y\in\cD_{A_i}} $ such that
\be
S_{\va,\delta'r}^k(u)\cap B_R\cap B_{\f{\rho^{i+1}R}{2}}(y)\neq\emptyset\label{DAiSvaMEMS}
\ee
for any $ y\in\cD_{A_i} $. After this elimination, the inclusion property in \eqref{Svadeltaeus} still preserves. Classify centers of balls in $ \cD_{A_i} $ into subcollections $ \wt{\cB}_{i+1} $ and $ \wt{\cG}_{i+1} $ such that we have the properties as follows.
\begin{itemize}
\item If $ y\in\wt{\cB}_{i+1} $, then there exists $ L(y,\rho^{i+1}R)\in\bA(n,k-1) $ such that
$$
F_{\delta}(y,\rho^{i+1}R)\subset B_{\f{\rho^{i+2}R}{10}}(L(y,\rho^iR))\cap B_{\rho^{i+1}R}(y).
$$
\item If $ y\in\wt{\cG}_{i+1} $, then $ F_{\delta}(y,\rho^{i+1}R) $ $ \f{\rho^{i+2}R}{20} $-effectively spans $ L'(y,\rho^iR)\in\bA(n,k) $.
\end{itemize}
For $ y\in\wt{\cB}_{i+1}\cup\wt{\cG}_{i+1} $, let $ r_y=\rho^{i+1}R $. We define $
\cB_{i+1}:=\cB_i\cup\wt{\cB}_{i+1} $ and $ \cG_{i+1}:=\wt{\cG}_{i+1} $. As a result, \ref{Mp1}-\ref{Mp4} are satisfied for $ i+1 $, and we only need to prove \ref{Mp5} and \ref{Mp6}. For fixed $ y\in\wt{\cB}_{i+1}\cup\wt{\cG}_{i+1} $, there is $ x\in\cG_i $ such that
\be
y\in L'(x,\rho^{i}R)\cap B_{\f{3\rho^iR}{2}}(x).\label{yLprime}
\ee
Since $ x\in\cD_i $, we have $ r_x=\rho^iR $ and $ F_{\delta}(x,\rho^{i}R) $ $ \f{\rho^{i+1}R}{20} $-effectively spans $ L'(x,\rho^iR)\in\bA(n,k) $. Given the definition of $ E=E(0,R) $ in \eqref{EMdef} and \ref{Mp2} for $ i $, together with Proposition \ref{MonFor}, we see that
\be
\sup_{y\in B_{2\rho^iR}(x)}\vt_f(u;y,\rho^iR)\leq \sup_{y\in B_{2R}}\vt_f(u;y,\rho^iR)\leq E.\label{simMEMS}
\ee
Combining with \eqref{yLprime}, the application of Lemma \ref{lempinchMEMS} with $ \beta=\f{\rho}{20} $ and $ s=\rho^iR $ implies that if $ \delta=\delta(\va,\ga,\Lda,n,p,\rho,\xi)>0 $ is sufficiently small and \eqref{fcondition} is satisfied, then
$$
\vt_f\(u;y,\f{r_y}{20}\)=\vt_f\(u;y,\f{\rho^{i+1}R}{20}\)>E-\xi.
$$
Consequently, we have \ref{Mp5} for $ i+1 $. Fix $ t\in[\rho^{i+1}R,R] $. By \eqref{DAiSvaMEMS}, there exists
$$
y'\in S_{\va,\delta'r}^k(u)\cap B_R\cap B_{\f{\rho^{i+1}R}{2}}(y).
$$
Using almost the same arguments as for $ i=1 $ in the proof of \eqref{geqga}, we can employ Remark \ref{remkplus} to obtain that any point $ y\in\cD_{A_i} $ satisfies \ref{Mp6}. Therefore, we hae verified \ref{Mp1}-\ref{Mp6} for $ i+1 $ and this step is completed.

\subsubsection*{Step 3. Construction for \texorpdfstring{$ \cD_{\ell} $}{}}

Assume that we have completed the construction of the covering \eqref{CoverSMEMS} for $ i=\ell-1 $, satisfying \ref{Mp1}-\ref{Mp6}. We now intend to give the construction for $ i=\ell $. For $ x\in\cG_{\ell-1} $, $ r_x=\rho^{\ell-1}R $. By \ref{Mp4} for $ i=\ell-1 $, $ F_{\delta}(x,\rho^{\ell-1}R) $ $ \f{\rho^{\ell}R}{20} $-effectively spans $ L'(x,\rho^{\ell-1}R)\in\bA(n,k) $. Choosing sufficiently small $ \delta=\delta(\va,\ga,\Lda,n,p,\rho)>0 $, it follows from Proposition \ref{FpropMEMS} with $ \beta=\f{\rho}{10} $ and $ s=\rho^{\ell-1}R $ that
\be
S_{\va,\delta'r}^k(u)\cap B_{\rho^{\ell-1}R}(x)\subset B_{\f{\rho^{\ell}R}{10}}(L'(x,\rho^{\ell-1}R))\cap B_{\rho^{\ell-1}R}(x).\label{Svadelta2}
\ee
Recall the definition $ A_{\ell-1} $ given by \eqref{defAiMEMS}. With the application of almost the same arguments in the proof of \eqref{SvadeltarBryMEMS}, it follows from \eqref{choiceofrRell} that
$$
(S_{\va,\delta'r}^k(u)\cap B_R)\backslash\(\bigcup_{y\in\cB_{\ell-1}}B_{r_y}(y)\)\subset B_{\f{2\rho^{\ell}R}{5}}(A_{\ell-1})\subset B_{\f{2r}{5}}(A_{\ell-1}).
$$
Let $ \cD_{A_{\ell-1}} $ be the maximal subset of $ A_{\ell-1} $ such that $ \dist(y,z)\geq\f{r}{5} $ for any $ y,z\in\cD_{A_{\ell-1}} $. Consequently, the balls in the collection
$$
\{B_{\f{r}{10}}(y)\}_{y\in\cD_{A_{\ell-1}}}\cup\{B_{\f{r_y}{10}}(y)\}_{y\in\cB_{\ell-1}}
$$
are pairwise disjoint and
$$
(S_{\va,\delta'r}^k(u)\cap B_R)\backslash\(\bigcup_{y\in\cB_{\ell-1}}B_{r_y}(y)\)\subset\bigcup_{y\in\cD_{A_{\ell-1}}} B_{\f{2r}{5}}(y),
$$
due to the choice of $ D_{A_{\ell-1}} $. Additionally, similar to \eqref{DAiSvaMEMS}, we assume that $
S_{\va,\delta'r}^k(u)\cap B_R\cap B_{\f{r}{2}}(y)\neq\emptyset $ for any $ y\in\cD_{A_{\ell-1}} $. Define $
\cB_{\ell}:=\cB_{\ell-1} $ and $ \cG_{\ell}:=\cD_{A_{\ell-1}} $ such that for any $ y\in\cG_{\ell} $, $ r_y=r $. By the construction above for $ i=\ell $, properties \ref{Mp1}-\ref{Mp4} hold. Analogous to \eqref{simMEMS}, we deduce from \eqref{EMdef} and \ref{Mp2} for $ \ell-1 $ that
$$
\sup_{y\in B_{2\rho^{\ell-1}R}(x)}\vt_f(u;y,\rho^{\ell-1}R)\leq\sup_{y\in B_{2R}}\vt_f(u;y,R)\leq E.
$$
Applying Lemma \ref{lempinchMEMS} with $ \beta=\f{\rho}{20} $ and $ s=\rho^{\ell-1}R $, it implies that for sufficiently small $ \delta=\delta(\va,\xi,\Lda,n,p,\rho)>0 $ such that \eqref{fcondition} is satisfied,
$$
\vt_f\(u;y,\f{r}{20}\)\geq\vt_f\(u;y,\f{\rho^{\ell}R}{20}\)>E-\xi
$$
for any $ y\in\cG_{\ell} $, where we have also used \eqref{choiceofrRell} and Proposition \ref{MonFor}. Moreover, we can show \eqref{VgeqMEMS} for $ y\in\cD_{\ell} $ using methods similar to those in the proof of \eqref{geqga}.

\subsubsection*{Step 4. Proof of \eqref{BlGlMEMS}}

Here, we denote $ \cD_{\ell} $ by $ \cD $ for simplicity and define
$$
\mu_{\cD}:=\sum_{y\in\cD}\w_kr_y^k\delta_y,
$$
where $
\wt{\cD}_t:=\cD\cap\{r_y\leq t\} $ and $ \mu_t:=\mu_{\cD} $. Note that $ \mu_t=\mu_{\cD}\llcorner\wt{\cD}_t\ll\mu_{\cD} $, namely for any $ A\subset\R^n $,
$$
\mu_{\cD}(A)=0\Ra(\mu\llcorner\wt{\cD}_t)(A)=0.
$$
Choose $ N\in\Z_+ $ such that
\be
2^{N-1}r<\f{R}{1000}\leq 2^Nr.\label{Rrchoose}
\ee

We will prove by the induction for $ j\in\Z\cap[0,N] $ and $ x\in B_{\f{11R}{10}} $, there holds
\be
\mu_{2^jr}(B_{2^jr}(x))=\sum_{y\in\wt{\cD}_{2^jr}\cap B_{2^jr}(x)}r_y^k\leq C_{\op{I}}'(n)(2^jr)^k.\label{jjplus1}
\ee
The estimate \eqref{BlGlMEMS} follows from \eqref{jjplus1} and simple covering arguments. Precisely, we have
$$
B_{\f{501R}{500}}\subset\bigcup_{i=1}^{N_1}B_{\f{R}{1000}}(x_i),
$$
where $ \{x_i\}_{i=1}^{N_1}\subset B_{\f{11R}{10}} $ with $ N_1\in\Z_+ $, satisfying $ N_1\leq C(n) $. Then \ref{Mp2} and $ 0<\rho<\f{1}{1000} $ yield
\be
\cD\subset B_{\f{501R}{500}}\subset B_{\f{11R}{10}}.\label{cDin}
\ee
As a result, it follows from \eqref{Rrchoose} and \eqref{jjplus1} that
\begin{align*}
\sum_{y\in\cD}r_y^k&\leq C(n)\mu_{\cD}\(B_{\f{501R}{500}}\)\leq C(n)\(\sum_{i=1}^{N_1}\mu_{\cD}(B_{2^Nr}(x_i))\)\leq C(n)C_{\op{I}}'(n)R^k.
\end{align*}
The constant $ C_{\op{I}}>0 $ in \eqref{BlGlMEMS} is chosen as $
C_{\op{I}}(n)=C(n)C_{\op{I}}'(n)>0 $.

By \ref{Mp3}, the balls $ \{B_{\f{r_y}{10}}(y)\}_{y\in\cD} $ are pairwise disjoint. According to \ref{Mp4}, we see that for any $ y\in\cD $, $ r_y\geq r $. These two observation imply the existence of the constant $ C_{\op{I}}'(n)>0 $ such that for any $ x\in B_{\f{11R}{10}} $, $
\mu_r(B_r(x))\leq C_{\op{I}}'(n)r^k $. As a result, \eqref{jjplus1} holds for $ j=0 $.

Assume that the estimate \eqref{jjplus1} holds true for any $ i\in\Z\cap[1,j] $ with $ j\in\Z\cap[1,N-1] $. We will show the property for $ j+1 $. The idea is to obtain a rough bound first and then refine it using Theorem \ref{Rei1}, the Reifenberg-type result. We start with the rough bound in the discrete form. Later, we will perform some adjustments.

\begin{lem}\label{mu2jlem}
For any $ x\in B_{\f{11R}{10}} $,
\be
\mu_{2^{j+1}r}(B_{2^{j+1}r}(x))\leq C(n)C_{\op{I}}'(n)(2^{j+1}r)^k.\label{mu2j}
\ee
\end{lem}
\begin{proof}
We first cover $ \cD\cap B_{2^{j+1}r}(x) $ with balls $ \{B_{2^jr}(x_i)\}_{i=1}^{N_2} $ such that $ \{x_i\}_{i=1}^{N_2}\subset\cD $ and $ N_2\leq C(n) $. Precisely, we have
$$
\cD\cap B_{2^{j+1}r}(x)\subset\bigcup_{i=1}^{N_2}B_{2^jr}(x_i).
$$
It follows from \eqref{cDin} and the estimate for $ j $ that
\be
\mu_{2^jr}(B_{2^{j+1}r}(x))\leq\sum_{i=1}^{N_2}\mu_{2^jr}(B_{2^jr}(x_i))\leq C(n)C_{\op{I}}(n)(2^jr)^k.\label{mu2j2j1plus}
\ee
By the definition of $ \mu_{2^{j+1}r} $,
\be
\mu_{2^{j+1}r}=\mu_{2^jr}+\sum_{y\in\cD,r_y\in(2^jr,2^{j+1}r]}\w_kr_y^k\delta_y.\label{decomposemuj}
\ee
Since $ \{B_{\f{r_y}{10}}(y)\}_{y\in\cD} $ are pairwise disjoint, it leads to
$$
\#\{y\in\cD\cap B_{2^{j+1}r}(x):r_y\in(2^jr,2^{j+1}r]\}\leq C(n).
$$
Consequently,
$$
\(\sum_{y\in\cD,r_y\in(2^jr,2^{j+1}r]}\w_kr_y^k\delta_y\)(B_{2^{j+1}r}(x))\leq C(n)(2^{j+1}r)^k.
$$
This, together with \eqref{mu2j2j1plus} and \eqref{decomposemuj}, implies that if $ C_{\op{I}}'(n)>1 $ is sufficiently large, then \eqref{mu2j} holds.
\end{proof}

Given \eqref{mu2j} in Lemma \ref{mu2jlem}, we introduce the continuous form of the rough bound, which is a direct consequence of a dyadic representation of the radius.

\begin{lem}\label{musBslem}
If $ s\in(0,2^{j+1}r) $ and $ x\in B_{\f{11R}{10}} $, then
\be
\mu_s(B_s(x))\leq C(n)C_{\op{I}}'(n)s^k.\label{musBs}
\ee
\end{lem}
\begin{proof}
Using the base of the induction, namely, \eqref{mu2j} with $ j=0 $, \eqref{musBs} is true for $ 0<s<r $. As a result, we let $ s\geq r $. For fixed $ s\in[r,2^{j+1}r) $, there is $ N_3\in\Z\cap[1,j] $ such that $ 2^{N_3}r\leq s<2^{N_3+1}r $. Applying the assumption of the induction and Lemma \ref{mu2jlem}, it follows that
$$
\mu_s(B_s(x))\leq\mu_{2^{N_3+1}r}(B_{2^{N_3+1}r}(x))\leq C(n)C_{\op{I}}'(n)(2^{N_3+1}r)^k\leq C(n)C_{\op{I}}'(n)s^k
$$
for any $ x\in B_{\f{11R}{10}} $, which implies \eqref{musBs}.
\end{proof}

Furthermore, we have the following lemma, which we regard as the final version of the rough bound.

\begin{lem}\label{murj1pluslem}
For any $ r\leq s<\f{2^{j+1}r}{10} $ and $ x\in B_{\f{11R}{10}} $,
\be
\mu_{2^{j+1}r}(B_{4s}(x))\leq C(n)C_{\op{I}}'(n)s^k.\label{murj1plus}
\ee
\end{lem}
\begin{proof}
To show \eqref{murj1plus}, we only need to verify that
\be
\mu_{2^{j+1}r}(B_s(z))\leq C(n)C_{\op{I}}'(n)s^k\label{onlyneed}
\ee
for any $ r\leq s<\f{2^{j+1}r}{10} $ and $ z\in\cD $. Indeed, for the case that $ x\in B_{\f{11R}{10}} $, if $ \cD\cap B_{4s}(x)=\emptyset $, by the definition of $ \mu_{2^{j+1}r} $, \eqref{murj1plus} follows directly and there is nothing to prove. Otherwise, we have $ \cD\cap B_{4s}(x)=\{x_i\}_{i=1}^{N_4} $. Furthermore, there is a subset of $ \{x_i\}_{i=1}^{N_4} $, denoted by $ \{x_i'\}_{i=1}^{N_5} $ such that
$$
\{x_i\}_{i=1}^{N_4}\subset\bigcup_{i=1}^{N_5}B_s(x_i'),
$$
and balls in the collection $ \{B_{\f{s}{3}}(x_i')\}_{i=1}^{N_5} $ are pairwise disjoint. Thus, $ N_5\leq C(n) $. Using \eqref{onlyneed} to $ x_i' $, it yields that
$$
\mu_{2^{j+1}r}(B_{4s}(x))\leq\sum_{i=1}^{N_5}\mu_{2^{j+1}r}(B_s(x_i'))\leq C(n)C_{\op{I}}'(n)s^k,
$$
which implies \eqref{murj1plus}.

Fix $ z\in\cD $ and $ r\leq s<\f{2^{j+1}r}{10} $. For $ y\in\supp(\mu)\cap B_s(z) $, since balls in $ \{B_{\f{r_{\zeta}}{10}}(\zeta)\}_{\zeta\in\cD} $ are pairwise disjoint, we have $ \f{r_y}{10}\leq|y-z|\leq s $. This implies that $ y\in\wt{\cD}_{10s} $. By the arbitrariness of $ y $, we get $ \supp(\mu)\cap B_s(z)\subset\wt{\cD}_{10s} $, and then
\be
\mu_{2^{j+1}r}(B_s(z))\leq\mu_{10s}(B_s(z))\leq\mu_{10s}(B_{10s}(z))\leq C(n)C_{\op{I}}'(n)s^k.\label{mu2jplus1}
\ee
Thus, \eqref{onlyneed} holds. Here in \eqref{mu2jplus1}, for the last inequality, we have used \eqref{cDin}, Lemma \ref{musBslem}, and the property that $ 10s\in(0,2^{j+1}r) $.
\end{proof}

Next, we will use the Reifenberg-type results to complete the proof. Define
$$
\mu:=\mu_{2^{j+1}r}\llcorner B_{2^{j+1}r}(x),
$$
where $ x\in B_{\f{11R}{10}} $. The proof of \eqref{jjplus1} is reduced to the estimate
\be
\mu(B_{2^{j+1}r}(x))\leq C_{\op{I}}'(n)(2^{j+1}r)^k.\label{wanttoshow}
\ee
For $ y\in\cD $, let
$$
\wt{W}_f(u;y,s):=\left\{\begin{aligned}
&\vt_f(u;y,2s)-\vt_f(u;y,s)&\text{ for }&\f{r_y}{10}\leq s<R,\\
&0&\text{ for }&0<s<\f{r_y}{10}.
\end{aligned}\right.
$$

\begin{lem}\label{beta2leqlem}
If $ y\in\cD $ and $ 0<s<\f{R}{10} $, then
\be
D_{\mu}^k(y,s)\leq C(\va,\ga\Lda,n,p,\rho)s^{-k}\int_{B_{2s}(y)}\wt{W}_f(u;z,2s)\ud\mu(z).\label{beta2leq}
\ee
\end{lem}
\begin{proof}
If $ 0<s<\f{r_y}{10} $, by \ref{Mp3}, the balls in $ \{B_{\f{r_z}{10}}(z)\}_{z\in\cD} $ are pairwise disjoint, and the left-hand side of \eqref{beta2leq} is $ 0 $, so the result is trivially true.

If $ \f{r_y}{10}<s<\f{R}{10} $, then \ref{Mp5} implies
$$
\inf_{V\in\bG(n,k+1)}\((10s)^{2-2\al-n}\int_{B_{10s}(y)}|V\cdot\na u|^2\)>\tau.
$$
Thus, \eqref{beta2leq} follows from Theorem \ref{beta2MEMS}.
\end{proof}

Letting $ t\in\R_+ $ and $ y\in\R^n $ be such that
\be
s\leq t<\f{2^{j+1}r}{10}\quad\text{and}\quad y\in B_{2^{j+1}r}(x).\label{st2j1es}
\ee
We have that if $ z\in B_t(y) $, then $ B_{2s}(z)\subset B_{3t}(y) $. For any $ \zeta\in B_{3t}(y) $, we claim that
\be
\mu(B_{4s}(\zeta))\leq C(n)C_{\op{I}}'(n)s^k.\label{coveress}
\ee
Without loss of generality, we let $ x\in B_{\f{251R}{250}} $. If not, it follows from \eqref{Rrchoose} and \eqref{cDin} that $
\cD\cap B_{2^{j+1}r}(x)=\emptyset $, and the left-hand side of \eqref{coveress} is $ 0 $, so there is nothing to prove. Now, for $ \zeta\in B_{3t}(y) $, by \eqref{st2j1es}, there holds
$$
|\zeta|\leq|x|+|y-x|+|\zeta-y|\leq\f{251R}{250}+2^{j+1}r+3t\leq\f{11R}{10}.
$$
Consequently, Lemma \ref{murj1pluslem} leads to
$$
\mu(B_{4s}(\zeta))\leq\mu_{2^{j+1}r}(B_{4s}(\zeta))\leq C(n)C_{\op{I}}'(n)s^k,
$$
which implies \eqref{coveress}. Integrating \eqref{beta2leq} for both sides on $ B_t(y) $, we obtain
\begin{align*}
\int_{B_t(y)}D_{\mu}^k(z,s)\ud\mu(z)&\leq \f{C}{s^k}\int_{B_t(y)}\(\int_{B_{2s}(z)}\wt{W}_f\(u;\zeta,2s\)\ud\mu(\zeta)\)\ud\mu(z)\\
&\leq\f{C}{s^k}\int_{B_t(y)}\(\int_{B_{3t}(y)}\chi_{B_{2s}(z)}(\zeta)\wt{W}_f\(u;\zeta,2s\)\ud\mu(\zeta)\)\ud\mu(z)\\
&\leq\f{C}{s^k}\int_{B_{3t}(y)}\(\int_{B_t(y)}\chi_{B_{2s}(\zeta)}(z)\ud\mu(z)\)\wt{W}_f\(u;\zeta,2s\)\ud\mu(\zeta)\\
&\leq\f{C}{s^k}\int_{B_{3t}(y)}\mu(B_{2s}(\zeta))\wt{W}_f\(u;\zeta,2s\)\ud\mu_{2^{j+1}r}(\zeta)\\
&\leq C(\va,\ga,\Lda,n,p,\rho)C_{\op{I}}'(n)\int_{B_{3t}(y)}\wt{W}_f\(u;\zeta,2s\)\ud\mu_{2^{j+1}r}(\zeta),
\end{align*}
where for the last inequality, we have used \eqref{coveress}. Moreover, we deduce that
\be
\begin{aligned}
&\int_{B_t(y)}\(\int_0^tD_{\mu}^k(z,s)\f{\ud s}{s}\)\ud\mu(z)\\
&\quad\quad\quad\quad\leq C(\va,\ga,\Lda,n,p,\rho)C_{\op{I}}'(n)\int_{B_{3t}(y)}\(\int_0^t\wt{W}_f\(u;z,2s\)\f{\ud s}{s}\)\ud\mu_{2^{j+1}r}(z).
\end{aligned}\label{beta2k1}
\ee
On the other hand, using Proposition \ref{MonFor} and \ref{Mp5}, we obtain
\begin{align*}
\int_0^t\wt{W}_f\(u;z,2s\)\f{\ud s}{s}&=\int_{\f{r_z}{10}}^t\wt{W}_f\(u;z,2s\)\f{\ud s}{s}\\
&=\int_{\f{r_z}{10}}^t\[\vt_f(u;z,4s)-\vt_f\(u;z,2s\)\]\f{\ud s}{s}\\
&=\(\int_{\f{t}{2}}^t+\int_{\f{r_z}{10}}^{\f{t}{2}}\)\vt_f(u;z,4s)\f{\ud s}{s}-\(\int_{\f{r_z}{5}}^t+\int_{\f{r_z}{10}}^{\f{r_z}{5}}\)\vt_f\(u;z,2s\)\f{\ud s}{s}\\
&=\int_{\f{t}{2}}^t\[\vt_f(u;z,4s)-\vt_f\(u;z,\f{2r_zs}{5t}\)\]\f{\ud s}{s}\\
&\leq C\[\vt_f(u;z,4t)-\vt_f\(u,z,\f{r_z}{5}\)\]\\\
&\leq C(\va,\ga,\Lda,n,p,\rho)\xi
\end{align*}
for any $ z\in\cD $ and $ 0<t<\f{2^{j+1}r}{10} $. This, together \eqref{beta2k1} and Lemma \ref{musBslem}, implies that
\begin{align*}
\int_{B_t(y)}\(\int_0^tD_{\mu}^k(z,s)\f{\ud s}{s}\)\ud\mu(z)&\leq C'(\va,\ga,\Lda,n,p,\rho)C_{\op{I}}'(n)\xi\mu(B_{3t}(y))\\
&\leq C'(\va,\ga,\Lda,n,p,\rho)C_{\op{I}}'(n)\xi t^k
\end{align*}
for any $ y\in B_{2^{j+1}r}(x) $ and $ 0<t<\f{2^{j+1}r}{10} $. If $ \xi=\xi(\va,\Lda,n,p,\rho)>0 $ is sufficiently small such that
$$
C'(\va,\ga,\Lda,n,p,\rho)C_{\op{I}}'(n)\xi<\delta_{\op{R}},
$$
then Theorem \ref{Rei1} leads to the estimate
$$
\mu(B_{2^{j+1}r}(x))\leq C_{\op{R}}(n)(2^{j+1}r)^k.
$$
Choosing $ C_{\op{I}}'(n)>C_{\op{R}}(n) $, we deduce \eqref{wanttoshow}, which completes the proof.

\section{Proof of main theorems}\label{Mainproof}

\subsection{Proof of Theorem \ref{quantitativethm}}
According to \eqref{quantitativethmass} and Lemma \ref{Inclusion1}, we have
\be
[u]_{C^{0,\al}(\ol{B}_{2R_0})}+[f]_{M^{2\al+n-4+\f{\ga}{2},2}(B_{2R_0})}\leq\Lda',\label{ga2ass}
\ee
where $ \Lda' $ depends only on $ \ga,\Lda,n $, and $ p $. We now choose $ \delta,\delta'=\delta,\delta'(\va,\ga,\Lda,n,p)\in(0,1) $ such that the result in Lemma \ref{maincoverMEMS} holds, under the assumption \eqref{ga2ass}. Again, by Lemma \ref{Inclusion1}, since $ f\in M_{\loc}^{2\al+n-4+\ga,2}(B_{4R_0}) $, it yields that for any $ x\in B_{R_0} $ and $ 0<R<10^{-6}R_0 $,
$$
[f]_{M^{2\al+n-4+\f{\ga}{2},2}(B_{20R}(x))}\leq CR^{\f{\ga}{4}}[f]_{M^{2\al+n-4+\ga,2}(B_{20R}(x))}\leq C(\ga,\Lda,n,p)R^{\f{\ga}{4}}.
$$
There exists $ R_1=R_1(\va,\ga,\Lda,n,p)\in(0,10^{-6}R_0) $
such that if $ 0<R<R_1 $, then
\be
[f]_{M^{2\al+n-4+\f{\ga}{2},2}(B_{20R}(x))}<\delta.\label{R1esti}
\ee
Assuming that $ 0<r'<\f{R_1}{2} $, we cover $ S_{\va,\delta'r'}^k(u) $ with 
\be
S_{\va,\delta'r'}^k(u)\subset\bigcup_{i=1}^{N_1}B_{\f{R_1}{2}}(x_i),\label{Svadeltarpri}
\ee
where $ \{x_i\}_{i=1}^{N_1}\subset B_{R_0} $ and 
\be
1\leq N_1\leq C(\va,\ga,\Lda,n,p).\label{Nleqvaga}
\ee
Using \eqref{R1esti}, by the choices of $ \delta $ and $ \delta' $, Lemma \ref{maincoverMEMS} implies that we can further cover $ S_{\va,\delta r'}^k(u)\cap B_{\f{R_1}{2}}(x_i) $ by the collection of balls $ \{B_{r'}(x)\}_{x\in\cC^{(i)}} $ with
$$
S_{\va,\delta'r'}^k(u)\cap B_{\f{R}{2}}(x_i)\subset\bigcup_{x\in\cC^{(i)}}B_{r'}(x),
$$
satisfying the estimate
$$
\sup_{i\in\Z\cap[1,N_1]}(\#\cC^{(i)})(r')^k\leq CR_1^k\leq C(\va,\ga,\Lda,n,p).
$$
Thus, we have, for any $ i\in\Z\cap[1,N_1] $,
\be
\cL^n\(B_{r'}(S_{\va,\delta'r'}^k(u)\cap B_\f{R_1}{2}(x_i))\)\leq \cL^n\(B_{r'}\(\bigcup_{x\in\cC^{(i)}}B_{r'}(x)\)\)\leq C(\va,\ga,\Lda,n,p)(r')^{n-k}.\label{Brprimes}
\ee
Moreover, \eqref{Svadeltarpri} and \eqref{Nleqvaga} show that
$$
\cL^n(B_{r'}(S_{\va,\delta'r'}^k(u))\leq\sum_{i=1}^{N_1}\cL^n(B_{r'}(S_{\va,\delta'r'}^k(u)\cap B_R(x_i)))\leq C(\va,\ga,\Lda,n,p)(r')^{n-k}.
$$
If $ 0<r<\f{\delta'R_1}{2} $, there exists $ 0<r'<\f{R_1}{2} $ such that $ r=\delta'r' $. The inequality \eqref{Brprimes} yields that
\be
\cL^n(B_{r}(S_{\va,r}^k(u))\leq \cL^n(B_{r'}(S_{\va,\delta'r'}^k(u))\leq C(\va,\ga,\Lda,n,p)r^{n-k}.\label{Lnleq}
\ee
On the other hand, if $ \f{\delta'R_1}{2}\leq r<1 $, then
$$
\cL^n(B_{r}(S_{\va,r}^k(u))\leq \cL^n(B_{R_0})\leq C(\va,\ga,\Lda,n,p)r^{n-k}.
$$
This, together with \eqref{Lnleq}, implies \eqref{quanti1}, and then \eqref{quanti11}. Let $ 0<R<1 $, $ 0<s<R $, and $ x\in B_{R_0} $. Assume that $ 0<R<R_1 $. According to \eqref{R1esti}, Lemma \ref{maincoverMEMS} shows that there is a covering of $ S_{\va,\delta's}^k(u)\cap B_R(x) $, denoted by $ \{B_s(y)\}_{y\in\cC} $ such that 
$$
S_{\va,\delta's}^k(u)\cap B_R(x)\subset\bigcup_{y\in\cC}B_s(y),\quad\text{and}\quad(\#\cC)s^k\leq C(\va,\ga,\Lda,n,p)R^k.
$$
As a result,
\be
\HH_s^k(S_{\va}^k(u)\cap B_R(x))\leq C(n)(\#\cC)s^k\leq C(\va,\ga,\Lda,n,p)R^k.\label{Ahlfors1}
\ee
On the other hand, if $ R_1\leq R<1 $, then we have a covering of $ B_R(x) $ by
$$
B_R(x)\subset\bigcup_{i=1}^{N_2}B_{\f{R_1}{2}}(x_i'),\quad N_2\leq C(\va,\ga,\Lda,n,p).
$$
Thus, we get
$$
\HH_s^k(S_{\va}^k(u)\cap B_R(x))\leq\sum_{i=1}^{N_2}\HH_s^k\(S_{\va}^k(u)\cap B_{\f{R_1}{2}}(x_i')\)\leq C(\va,\ga,\Lda,n,p)R^k.
$$
This, together with \eqref{Ahlfors1}, implies that for any $ 0<s<R $,
$$
\HH_s^k(S_{\va}^k(u)\cap B_R(x))\leq C(\va,\ga,\Lda,n,p)R^k.
$$
Taking $ s\to 0^+ $, \eqref{quanti2} follows directly.

Now, we show the third property of Theorem \ref{quantitativethm}. By Lemma \ref{decomSkuseSva}, we have
\be
S^k(u)\cap B_{R_0}=\bigcup_{i\in\Z_+}S_{i^{-1}}^k(u).\label{Skiminus1}
\ee
Thus, it remains to show the rectifiability of $ S_{\va}^k(u) $ for any $ \va>0 $. Let $ S\subset S_{\va}^k(u) $ be such that $ \HH^k(S)>0 $. For any $ x\in S_{\va}^k(u) $ and $ 0<r\leq 1 $, we define
$$
g_f(u;x,r):=\vt_f(u;x,r)-\vt(u;x).
$$
According to Proposition \ref{propupture} and \ref{MonFor}, we have that for any $ x\in S_{\va}^k(u) $, $ \lim_{r\to 0^+}g_f(u;x,r)=0 $, and $ g_f(u;\cdot,r) $ is bounded. The dominated convergence theorem yields that for any $ \sg>0 $, there exists $ r_0=r_0(f,\sg,u)>0 $ such that 
$$
\f{1}{\HH^k(S)}\int_Sg_f(u;x,10r_0)\ud\HH^k(x)\leq\sg.
$$
By average arguments, there is an $ \HH^k $-measurable set $ E\subset S $ such that $ \HH^k(E)\leq\sg\HH^k(S) $ and $ g_f(u;x,10r_0)\leq\sg $ for any $ x\in F:=S\backslash E $. Cover $ F $ by finite number of balls $ \{B_{r_0}(y_i)\}_{i=1}^{N_3} $ such that $ \{y_i\}_{i=1}^{N_3}\subset F $. We claim that if $ \sg=\sg(\va,\ga,\Lda,n,p)>0 $ is sufficiently small, then for any $ i\in\Z\cap[1,N_3] $, $ F\cap B_{r_0}(x_i) $ is $ k $-rectifiable. If such a claim is true, repeating this procedure to $ S $ for countably times, we finally obtain that $ S $ is $ k $-rectifiable. The arbitrariness for the choice of $ S $ implies that $ S_{\va}^k(u) $ is $ k $-rectifiable. Let us show this claim. Without loss of generality, we only consider the ball $ B_{r_0}(x_1) $ and assume that $ 0<r_0<\f{1}{100} $. By the assumption of $ E $, we have 
\be
g_f(u;z,10r_0)=\vt_f(u;z,10r_0)-\vt(u;z)\leq\sg\label{G10r0small}
\ee 
for any $ z\in F $. Choosing $ \sg=\sg(\sg',\Lda,n,p)>0 $ sufficiently small, we can apply Lemma \ref{SmaHomMEMS} to obtain that $ u $ is $ (0,\sg') $-symmetric in $ B_{5s}(z) $ for any $ 0<s\leq r_0 $, where $ \sg'>0 $ is to be determined later. For $ z\in F\subset S_{\va}^k(u) $, $ u $ is not $ (k+1,\va) $-symmetric in $ B_{5s}(z) $. Choosing $ \sg'=\sg'(\va,\Lda,n,p)>0 $ sufficient small and using Corollary \ref{beta22MEMS}, we deduce that for any $ z\in F $ and $ 0<s\leq r_0 $,
$$
D_{\mu}^k(z,s)\leq C(\va,\ga,\Lda,n,p)s^{-k}\int_{B_s(z)}W_f(u;\zeta,s)\ud\mu(\zeta),
$$
where $ \mu:=\HH^k\llcorner F $. Integrating with respect to $ z $ for both sides of the inequality above on $ B_r(x) $ with $ x\in B_{r_0}(x_1) $ and $ 0<r\leq r_0 $, we have
\begin{align*}
\int_{B_r(x)}D_{\mu}^k(z,s)\ud\mu(z)&\leq Cs^{-k}\int_{B_r(x)}\(\int_{B_s(z)}W_f(u;\zeta,s)\ud\mu(\zeta)\)\ud\mu(z)\\
&\leq Cs^{-k}\int_{B_r(x)}\(\int_{B_{r+s}(x)}\chi_{B_s(z)}(\zeta)W_f(u;\zeta,s)\ud\mu(\zeta)\)\ud\mu(z)\\
&\leq Cs^{-k}\int_{B_{r+s}(x)}\HH^k(F\cap B_s(\zeta))W_f(u;\zeta,s)\ud\mu(\zeta)\\
&\leq C(\va,\ga,\Lda,n,p)\int_{B_{r+s}(x)}W_f(u;z,s)\ud\mu(z).
\end{align*}
For the last inequality above, we have used \eqref{quanti2}. It follows that
\begin{align*}
\int_{B_r(x)}\(\int_0^rD_{\mu}^k(z,s)\f{\ud s}{s}\)\ud\mu(z)&\leq C\int_{B_{2r}(x)}\(\int_0^r(\vt_f(u;z,2s)-\vt_f(u;z,s))\f{\ud s}{s}\)\ud\mu(z)\\
&=C\int_{B_{2r}(x)}\(\int_0^{2r}\vt_f(u;z,s)\f{\ud s}{s}-\int_0^r\vt_f(u;z,s)\f{\ud s}{s}\)\ud\mu(z)\\
&=C\int_{B_{2r}(x)}\(\int_r^{2r}\vt_f(u;z,s)\f{\ud s}{s}\)\ud\mu(z)\\
&\leq C(\va,\ga,\Lda,n,p)r^k
\end{align*}
for any $ x\in B_{r_0}(x_1) $ and $ 0<r\leq r_0 $, where for the last inequality, we have used \eqref{quanti2} and Proposition \ref{propupture}. Theorem \ref{Rei2} now implies that $ F\cap B_{r_0}(x_1) $ is $ k $-rectifiable. As a result, we prove the claim as desired.

Additionally, with the help of Proposition \ref{QuaConSplMEMS} and almost the same methods in the proof of Theorem 1.4 and 1.5 of \cite{NV17}, we can obtain that for $ \HH^k $-a.e. $ x\in S_{\va}^k(u) $ or $ S^k(u) $, there exists $ V\in\bG(n,k) $ such that any tangent function of $ u $ at $ x $ is $ k $-symmetric with respect to $ V $. Also, see \cite{FWZ24} for similar arguments in the proof of Theorem 1.21.

\subsection{Proof of Theorem \ref{main1}}

By Proposition \ref{propHolder} and Lemma \ref{Inclusion}, without loss of generality, we can assume that for some $ \ga=\ga(n,p,q)>0 $,
$$
[u]_{C^{0,\al}(B_2)}+[f]_{M^{2\al+n-4+\ga,2}(B_2)}\leq\Lda.
$$
The estimate \eqref{Mincontent} follows from Proposition \ref{corcomMEMS} and the first property of Theorem \ref{quantitativethm} with $ k=n-2 $. For \eqref{enhancement}, we only show the case when $ j=1 $ and $ f\equiv 0 $, and the general case follows from almost the same argument. By standard regularity theory of elliptic equations, if $ \inf_{B_{2s}(x)}u\geq\sg s^{\al} $, then
$$
\|\na u\|_{L^{\ift}(B_s(x))}\leq C(\sg,\Lda,n,p)s^{\al-1}.
$$
It implies that
$$
\{x\in B_1:u(x)\geq\va r^{\al}\}\subset\{x\in B_1:|\na u(x)|\leq C_0r^{\al-1}\},
$$
where $ C_0>0 $ depends only on $ \Lda,n,p $, and $ q $. Thus, we have
$$
\{x\in B_1:|\na u(x)|>C_0r^{\al-1}\}\subset\{x\in B_1:u(x)<\va r^{\al}\}
$$
for any $ 0<r<1 $. Letting $ \lda=Cr^{\al-1} $, \eqref{Mincontent} yields \eqref{enhancement}. The $ (n-2) $-rectifiability is a consequence of Remark \ref{remn2recti} and the third property of Theorem \ref{quantitativethm}.

\appendix

\section{}\label{App1}

\subsection{The spaces of Morrey and Campanato}\label{App11}

Let $ \lda\geq 0 $ and $ 1\leq q<+\ift $. Assume that $ \om\subset\R^n $ is a domain. We define the Morrey space $ M^{\lda,q}(\om) $ by
$$
M^{\lda,q}(\om):=\{f\in L^q(\om):[f]_{M^{\lda,q}(\om)}<+\ift\},
$$
where the seminorm $ [\cdot]_{M^{\lda,q}(\om)} $ is given by
$$
[f]_{M^{\lda,q}(\om)}:=\sup_{x\in \om,\,\,0<r<\diam(\om)}\(r^{-\lda}\int_{\om\cap B_r(x)}|f|^q\)^{\f{1}{q}}.
$$
The Campanato space $ \M^{\lda,q}(\om) $ is defined by
$$
\M^{\lda,q}(\om):=\{f\in L^q(\om):[f]_{\M^{\lda,q}(\om)}<+\ift\}
$$
with the seminorm
$$
[f]_{\M^{\lda,q}(\om)}:=\sup_{x\in \om,\,\,0<r<\diam(\om)}\(r^{-\lda}\int_{\om\cap B_r(x)}\left|f-\dashint_{\om\cap B_r(x)}f\right|^q\)^{\f{1}{q}}.
$$
The local Morrey and Campanato spaces $ M_{\loc}^{\lda,q}(\om) $ and $ \M_{\loc}^{\lda,q}(\om) $ are defined as
\begin{align*}
M_{\loc}^{\lda,q}(\om):=\{f\in L_{\loc}^q(\om):f\in M^{\lda,q}(K)\text{ for any }K\subset\subset \om\},\\
\M_{\loc}^{\lda,q}(\om):=\{f\in L_{\loc}^q(\om):f\in \M^{\lda,q}(K)\text{ for any }K\subset\subset \om\}.
\end{align*}

We first present some inclusion results of Morrey spaces. The proofs follow from direct calculations, and we omit them for simplicity.

\begin{lem}\label{Inclusion1}
Assume that $ \om $ is a bounded domain. If $ 0\leq\lda_1\leq\lda_2 $ and $ 1\leq q<+\ift $, then
\begin{align*}
M^{\lda_2,q}(\om)\subset M^{\lda_1,q}(\om)\quad\text{and}\quad
M_{\loc}^{\lda_2,q}(\om)\subset M_{\loc}^{\lda_1,q}(\om).
\end{align*}
In particular, for any $ f\in M^{\lda_2,q}(\om) $,
$$
[f]_{M^{\lda_1,q}(\om)}\leq C(\diam(\om))^{\f{\lda_2-\lda_1}{q}}[f]_{M^{\lda_2,q}(\om)},
$$
where $ C>0 $ depends only on $ \lda_1,\lda_2,n,q_1 $, and $ q_2 $.
\end{lem}

\begin{lem}\label{Inclusion}
Assume that $ \om $ is a bounded domain. If $ 1\leq q_1<q_2<+\ift $ and $ \lda\geq 0 $ satisfy $ \f{q_1}{q_2}+\f{\lda}{n}<1 $, then $ M^{\lda,q_1}(\om)\subset L^{q_2}(\om) $. Moreover, for any $ f\in L^{q_2}(\om) $,
$$
[f]_{M^{\lda,q_1}(\om)}\leq C(\diam(\om))^{n(\f{1}{q_1}-\f{1}{q_2})-\f{\lda}{q_1}}\|f\|_{L^{q_2}(\om)},
$$
where $ C>0 $ depends only on $ \lda,n,q_1 $, and $ q_2 $.
\end{lem}

The lemma below gives the compactness of the Morrey space.

\begin{lem}\label{MorreyL2}
Assume that $ \om $ is a bounded domain. Let $ 1<q<+\ift $ and $ \lda\geq 0 $. If $ \{f_i\} $ is a sequence in $ M^{\lda,q}(\om) $ such that $ [f_i]_{M^{\lda,q}(\om)} $ is uniformly bounded, then there exists $ C>0 $ depending only on $ \lda $ and $ q $ such that
\be
\|f_i\|_{L^q(\om)}\leq C(\diam(\om))^{\f{\lda}{q}}[f_i]_{M^{\lda,q}(\om)},\label{MorreyLpbound}
\ee
and up to a subsequence, $ f_i\wc f_{\ift} $ weakly in $ L^q(\om) $ with $ f_{\ift}\in M^{\lda,q}(\om) $.
\end{lem}
\begin{proof}
The estimate \eqref{MorreyLpbound} follows directly from the definition of $ [\cdot]_{M^{\lda,q}(\om)} $. Since $ 1<q<+\ift $, there exists $ f_{\ift}\in L^q(\om) $ such that up to a subsequence, we have $ f_i\wc f_{\ift} $ weakly in $ L^q(\om) $. For any $ x\in \om $ and $ 0<r<\diam(\om) $, by the property of weak convergence, we have
$$
r^{-\lda}\int_{\om\cap B_r(x)}|f_{\ift}|^q\leq\liminf_{i\to+\ift}\(r^{-\lda}\int_{\om\cap B_r(x)}|f_i|^q\),
$$
which implies that $ f_{\ift}\in M^{\lda,q}(\om) $.
\end{proof}

Using the Campanato space, we have the characterization of H\"{o}lder's space.

\begin{lem}[\cite{GM05}, Theorem 5.5]\label{Campanuse}
Let $ 1\leq q<+\ift $ and $ n<\lda\leq n+q $. Suppose that $ \om\subset\R^n $ is a bounded domain and assume that there exists $ c_0>0 $ such that
$$
\cL^n(\om\cap B_r(x))\geq c_0r^n
$$
for any $ x\in \om $ and $ 0<r<\diam(\om) $. Then $ \M^{\lda,q}(\om)=C^{0,\f{\lda-n}{q}}(\ol{\om}) $. In particular,
$$
\f{1}{C}[f]_{C^{0,\f{\lda-n}{q}}(\ol{\om})}\leq[f]_{\M^{\lda,q}(\om)}\leq C[f]_{C^{0,\f{\lda-n}{q}}(\ol{\om})},
$$
where $ C>0 $ depends only on $ c_0,\lda,n $, and $ q $.
\end{lem}

\subsection{Some results for elliptic equations} We will recall some Liouville-type properties and regularity results for some elliptic equations in this subsection.

\begin{defn}\label{defudeltau}
Let $ \om\subset\R^n $ be a domain. We call $ u\in H_{\loc}^1(\om) $ a stationary solution of the equation $ u\Delta u=0 $ in $ \om $ if the following two properties hold.
\begin{enumerate}
\item $ u\Delta u=0 $ in the weak sense that for any $ \vp\in C_0^{\ift}(\om) $,
\be
\int_{\om}(|\na u|^2\vp+(\na u\cdot\na\vp)u)=0.\label{udeltau1}
\ee
\item For any $ Y\in C_0^{\ift}(\om,\R^n) $,
\be
\int_{\om}(|\na u|^2\op{div}Y-2DY(\na u,\na u))=0.\label{udeltau2}
\ee
\end{enumerate}
\end{defn}

The following lemma is a Liouville-type result for the solutions given by Definition \ref{defudeltau}.

\begin{lem}[\cite{DWW16}, Theorem 2.2]\label{ulapulem}
Let $ u\in(C_{\loc}^{0,\al}\cap H_{\loc}^1)(\R^n) $ be such that $ u\geq 0 $ and
\be
\sup_{R>0}[u]_{C^{0,\al}(\ol{B}_R)}\leq\Lda.\label{BRholder}
\ee
If $ u $ is a stationary solution of $ u\Delta u=0 $ in $ \R^n $, then $ u $ is a constant function.
\end{lem}

As a direct application of the above lemma, we obtain the following Liouville-type property of harmonic functions. We can also get such a result by applying the average formula for harmonic functions.

\begin{cor}\label{Liouvillecla}
Assume that $ u $ is a harmonic function in $ \R^n $ and satisfies \eqref{BRholder}, then $ u $ is a constant function.
\end{cor}
\begin{proof}
Given $ \Delta u=0 $, we see that $ u $ satisfies \eqref{udeltau1} and \eqref{udeltau2}. As a result, it follows from Lemma \ref{ulapulem} that $ u $ is a constant function.
\end{proof}

The following lemma gives the interior regularity estimate for weak solutions of $ -\Delta u=f $, where $ f $ is in the Morrey space. It is analogous to similar results of Chapter 5 in \cite{GM05}, a consequence of the iteration argument and Lemma \ref{Campanuse}.

\begin{lem}\label{regular1}
Let $ 2<\lda<4 $, $ 0<r\leq 1 $ and $ x\in\R^n $. Assume that $ f\in M^{n-\lda,2}(B_{2r}(x)) $ and $ u\in H^1(B_{2r}(x)) $ is a weak solution of $ -\Delta u=f $ in the sense that for any $ \vp\in C_0^{\ift}(B_{2r}(x)) $,
$$
\int_{B_{2r}(x)}\na u\cdot\na\vp=\int_{B_{2r}(x)}f\vp.
$$
Then $ u\in C^{0,\f{4-\lda}{2}}(\ol{B}_r(x)) $ and
$$
[u]_{C^{0,\f{4-\lda}{2}}(\ol{B}_r(x))}\leq C\[\(r^{\lda-n-2}\int_{B_{2r}(x)}|\na u|^2\)^{\f{1}{2}}+[f]_{M^{n-\lda,2}(B_{2r}(x))}\],
$$
where $ C>0 $ depends only on $ \lda $ and $ n $.
\end{lem}

To show this lemma, we need the following result. It is a consequence of standard iteration arguments.

\begin{lem}[\cite{GM05}, Lemma 5.13]\label{iteration}
Let $ A,\beta_1,\beta_2>0 $, and $ r>0 $ with $ \beta_1>\beta_2 $. Assume that $ \psi:[0,+\ift)\to[0,+\ift) $ is a nondecreasing function satisfying that for any $ 0<\rho\leq r $,
$$
\psi(\rho)\leq A\[\(\f{\rho}{r}\)^{\beta_1}+\va\]\psi(r)+Br^{\beta_2}.
$$
There exist $ \va_0,C>0 $, depending only on $ A,\beta_1,\beta_2 $ such that if $ 0<\va\leq\va_0 $, then for any $ 0\leq\rho\leq r $,
$$
\rho^{-\beta_2}\psi(\rho)\leq C(r^{-\beta_2}\psi(r)+B).
$$
\end{lem}

\begin{proof}[Proof of Lemma \ref{regular1}]
Set $ \Lda:=[f]_{M^{n-\lda,2}(B_{2r}(x))} $. For $ y\in B_r(x) $, we choose $ v\in H^1(B_r(y)) $ such that it is a weak solution of the following Dirichlet problem.
$$
\left\{\begin{aligned}
\Delta v&=0&\text{ in }&B_r(y),\\
v&=u&\text{ on }&\pa B_r(y).
\end{aligned}\right.
$$
Such a weak solution exists due to applying the Lax-Milgram theorem. Let $ w:=u-v $. As a result, $ w\in H_0^1(B_r(y)) $ is a weak solution of the Dirichlet problem
\be
\left\{\begin{aligned}
-\Delta w&=f&\text{ in }&B_r(y),\\
w&=0&\text{ on }&\pa B_r(y).
\end{aligned}\right.\label{wfunction}
\ee
By the interior regularity estimate of harmonic function, we have that for any $ 0<\rho\leq r $,
\be
\int_{B_{\rho}(y)}|\na v|^2\leq C(n)\(\f{\rho}{r}\)^n\int_{B_r(y)}|\na v|^2.\label{harmovuse}
\ee
Define the functional
$$
\psi(\rho):=\int_{B_{\rho}(y)}|\na u|^2,\quad 0<\rho\leq r.
$$
Since $ w $ vanishes on $ \pa B_r(y) $, we can test \eqref{wfunction} with $ w $ itself and obtain from Cauchy's inequality and Poincar\'{e}'s inequality that
\begin{align*}
\int_{B_r(y)}|\na w|^2&\leq\(\int_{B_r(y)}|f|^2\)^{\f{1}{2}}\(\int_{B_r(y)}|w|^2\)^{\f{1}{2}}\leq C(\lda,n)\Lda r^{\f{n-\lda}{2}+1}\(\int_{B_r(y)}|\na w|^2\)^{\f{1}{2}}.
\end{align*}
Thus, we have
$$
\int_{B_r(y)}|\na w|^2\leq C(\lda,n)\Lda^2r^{n-\lda+2}.
$$
This, together with \eqref{harmovuse}, implies that
\begin{align*}
\psi(\rho)&\leq 2\int_{B_{\rho}(y)}|\na v|^2+2\int_{B_{\rho}(y)}|\na w|^2\leq C(\lda,n)\[\(\f{\rho}{r}\)^n\psi(r)+\Lda^2r^{n-\lda+2}\].
\end{align*}
Applying Lemma \ref{iteration} with $ B=\Lda^2 $, $ \beta_1=n $, and $ \beta_2=n-\lda+2 $, it follows that for any $ 0<\rho\leq r $,
\be
\rho^{\lda-n-2}\psi(\rho)\leq C(\lda,n)(r^{\lda-n-2}\psi(r)+\Lda^2).\label{rhon2esti}
\ee
Noting the basic property
$$
\dashint_{B_{\rho}(y)}u:=\underset{a\in\R}{\argmin}\int_{B_{\rho}(y)}|u-a|^2,
$$
Poincar\'{e}'s inequality yields that for any $ 0<\rho<r $,
$$
\rho^{\lda-n-4}\int_{B_r(x)\cap B_{\rho}(y)}\left|u-\dashint_{B_r(x)\cap B_{\rho}(y)}u\right|^2\leq\rho^{\lda-n-4}\int_{B_{\rho}(y)}\left|u-\dashint_{B_{\rho}(y)}u\right|^2\leq C(n)\rho^{\lda-n-2}\psi(\rho).
$$
Combining with \eqref{rhon2esti} and the arbitrariness of $ y\in B_r(x)$, we get by Lemma \ref{Campanuse} that
$$
[u]_{C^{0,\f{4-\lda}{2}}(\ol{B}_r(x))}^2\leq C(\lda,n)\(r^{\lda-n-2}\int_{B_{2r}(x)}|\na u|^2+\Lda^2\),
$$
which completes the proof.
\end{proof}

Next, we recall Harnack's inequality.

\begin{lem}\label{LemHL}
Let $ f\in L^q(B_2) $ with $ q\in(\f{n}{2},+\ift) $. Suppose that $ u\in H^1(B_2) $ is a subsolution of $ -\Delta u=f $ in the weak sense. Namely, for any $ \vp\in C_0^{\ift}(B_2) $ with $ \vp\geq 0 $ in $ B_2 $,
$$
\int_{B_2}\na u\cdot\na\vp\leq\int_{B_2}f\varphi.
$$
Then $ u^{+}=\max\{u,0\}\in L_{\loc}^{\ift}(B_2) $ and satisfies the inequality
$$
\|u^{+}\|_{L^{\ift}(B_1)}\leq C(\|u\|_{L^1(B_2)}+\|f\|_{L^q(B_2)}),
$$
where $ C>0 $ depends only $ \Lda,n,p $, and $ q $.
\end{lem}

\begin{proof}
This lemma is a particular case of Theorem 4.1 in \cite{HL11} and is a consequence of the De Giorgi-Moser-Nash iteration method.
\end{proof}

\section*{Acknowledgements}

The authors want to thank Haotong Fu for his enlightening discussions regarding some details in this paper. They also thank Prof. Kelei Wang and Dr. Ke Wu for identifying some errors in the original manuscript and providing valuable advice. The authors are partially supported by the National Key R$\&$D Program of China under Grant 2023YFA1008801 and NSF of China under Grant 12288101.

\end{document}